\let\withclearpagecommands\a   

\documentclass[journal]{IEEEtran} 
\usepackage[T1]{fontenc}
\usepackage{fancyhdr}
\usepackage{afterpage}

\usepackage{mathtools,lipsum}

\usepackage{cuted}
\usepackage{tikz}
\usetikzlibrary{arrows,chains,matrix,positioning,scopes,shapes}
\pgfdeclarelayer{background}
\pgfdeclarelayer{behind}
\pgfdeclarelayer{above}
\pgfdeclarelayer{glass}
\pgfsetlayers{background,behind,main,above,glass}
\makeatletter
\tikzset{join/.code=\tikzset{after node path={%
\ifx\tikzchainprevious\pgfutil@empty\else(\tikzchainprevious)%
edge[every join]#1(\tikzchaincurrent)\fi}}}
\makeatother

\tikzset{>=stealth',every on chain/.append style={join},
         every join/.style={->}}
\tikzstyle{labeled}=[execute at begin node=$\scriptstyle,
   execute at end node=$]

\makeatletter
\def\ps@headings{%
\def\@oddhead{\mbox{}\scriptsize\rightmark \hfil \thepage}%
\def\@evenhead{\scriptsize\thepage \hfil \leftmark\mbox{}}%
\def\@oddfoot{}%
\def\@evenfoot{}}
\makeatother 
\IEEEoverridecommandlockouts

\usepackage{color}
\usepackage{amsfonts}
\usepackage{graphicx}
\usepackage[captionskip=0.25cm]{subfig}
\usepackage{pifont}

\usepackage[skip=0.1cm]{caption}

\usepackage{tikz}
\usepackage{times}
\usepackage{cite}
\usepackage{lettrine}
\usepackage{amsthm}
\renewcommand\qedsymbol{$\blacksquare$}
\allowdisplaybreaks[4]
\usepackage{array} 
\usepackage{amssymb}

\usepackage{stfloats}
\usepackage{footnote}
\usepackage{booktabs}
\usepackage{array}
\usepackage{algorithm}
\usepackage{algpseudocode}
\usepackage{subeqnarray}
\usepackage{cases}
\usepackage{threeparttable}
\usepackage{color}
\usepackage{url}
\usepackage{bm}
\usepackage{physics}

\usepackage{hyperref}
\usepackage[margin=0.5in]{geometry}
\usepackage{scalerel}

\usepackage{balance}

\usepackage{nccmath}
\usepackage[strict]{changepage}
\newtheorem{theorem}{Theorem}

\newtheorem{lemma}[theorem]{Lemma}

\newtheorem{proposition}[theorem]{Proposition}
\newtheorem{example}[theorem]{Example}
\newtheorem{remark}[theorem]{Remark}

\newtheorem{conjecture}[theorem]{Conjecture}
\newtheorem{definition}[theorem]{Definition}

\newcommand{\eproposition}{\hfill$\square$}
\newcommand{\epropositionproof}{\hfill\qedsymbol}
\newcommand{\elemma}{\hfill$\square$}
\newcommand{\elemmaproof}{\hfill\qedsymbol}
\newcommand{\etheorem}{\hfill$\square$}
\newcommand{\etheoremproof}{\hfill\qedsymbol}
\newcommand{\edefinition}{\hfill$\triangle$}

\newcommand{\eexample}{\hfill$\square$}

\newcommand{\econjecture}{\hfill$\square$}
\newcommand{\eremark}{\hfill$\triangle$}

\newcommand{\vect}[1]{ \bm{#1} }

\newcommand{\set}[1]{\mathcal{#1}}

\newcommand{\vgam}{ \vect{\gamma} }
\newcommand{\setEfull}{ \mathcal{E} }
\newcommand{\setF}{ \mathcal{F} }
\newcommand{\sfN}{ \mathsf{N} }

\newcommand{\sR}{\mathbb{R}}

\newcommand{\sZ}{\mathbb{Z}}
\newcommand{\sZpp}{\mathbb{Z}_{\geq 1}}

\newcommand{\matr}[1]{ \bm{#1} }
\newcommand{\defeq}{ \triangleq }

\newcommand{\fr}[1]{ f_{ \mathrm{r}, #1 } }
\newcommand{\fc}[1]{ f_{ \mathrm{c}, #1 } }

\newcommand{\setx}[1]{ \set{X}_{#1} }

\newcommand{\setX}{ \set{X} }

\newcommand{\vvu}{\bm{u}}

\newcommand{\vui}{\vvu_{i}}
\newcommand{\vuj}{\vvu_{j}}

\newcommand{\supp}{\mathrm{supp}}

\newcommand{\ZBn}[1]{\perm_{ \mathrm{B}, #1 }}

\newcommand{\perm}{ \mathrm{perm} }

\newcommand{\permb}{ \mathrm{perm}_{\mathrm{B}} }
\newcommand{\permbM}[1]{\perm_{ \mathrm{B}, #1 }}

\newcommand{\perms}{ \mathrm{perm}_{\mathrm{S}} }
\newcommand{\permscs}{ \mathrm{perm}_{\mathrm{scS}} }
\newcommand{\permscsM}[1]{ \mathrm{perm}_{\mathrm{scS},#1} }

\newcommand{\pmtheta}{p_{\mtheta}}

\newcommand{\vsigma}{ \bm{\sigma} }

\newcommand{\setC}{\set{C}}
\newcommand{\setP}{\set{P}}
\newcommand{\setR}{\set{R}}
\newcommand{\setS}{\set{S}}

\newcommand{\vgamRCp}{\vgam_{\setR,\setC}}
\newcommand{\gamRCp}{\gamma_{\setR,\setC}}

\newcommand{\CM}[1]{\Cgen{M}{#1}}
\newcommand{\Ctwo}[1]{\Cgen{2}{#1}}
\newcommand{\Cgen}[2]{C_{#1,#2}}

\newcommand{\CBM}[1]{\CBgen{M}{#1}}
\newcommand{\CBtwo}[1]{\CBgen{2}{#1}}
\newcommand{\CBgen}[2]{C_{\mathrm{B},#1,#2}}

\newcommand{\CscSgen}[2]{C_{\mathrm{scS},#1,#2}}

\newcommand{\mtheta}{\matr{\theta}}

\newcommand{\isigmaoi}{\bigl( i,\sigma_1(i) \bigr)}

\newcommand{\tvgam}{\tilde{\vgam}}
\newcommand{\tgam}{\tilde{\gamma}}

\newcommand{\sZp}{\sZ_{\geq 0}}

\newcommand{\vtheta}{\matr{\theta}}
\newcommand{\mP}{\matr{P}}

\newcommand{\mU}{\matr{U}}

\newcommand{\sfG}{\mathsf{G}}

\newcommand{\FBthe}{F_{ \mathrm{B}, \mtheta }}
\newcommand{\UBthe}{U_{ \mathrm{B}, \mtheta }}
\newcommand{\HBthe}{H_{ \mathrm{B} }}

\newcommand{\FscSthe}{F_{ \mathrm{scS}, \mtheta }}
\newcommand{\UscSthe}{U_{ \mathrm{scS}, \mtheta }}
\newcommand{\HscSthe}{H_{ \mathrm{scS} }}

\newcommand{\setA}{ \set{A} }
\newcommand{\suppmtheta}{\operatorname{supp}(\mtheta)}

\newcommand{\hvgamRCp}{ \hat{\vgam}_{\setR,\setC} }
\newcommand{\hgamRCp}{ \hat{\gamma}_{\setR,\setC} }

\newcommand{\vgamzo}{\vgam}
\newcommand{\gamzo}{\gamma}

\newcommand{\HG}{H_{ \mathrm{G} }}

\newcommand{\FGthe}{F_{ \mathrm{G}, \mtheta }}
\newcommand{\UGthe}{U_{ \mathrm{G}, \mtheta }}

\newcommand{\vp}{\bm{p} }

\newcommand{\PiA}[1]{ \Pi_{\setA(#1)} }
\newcommand{\Gamnthe}{ \Gamma_{n}(\vtheta) }
\newcommand{\GamMnthe}{ \Gamma_{M,n}(\vtheta) }
\newcommand{\Gampnthe}{ \Gamma'_{n}(\vtheta) }

\newcommand{\GamMn}{ \Gamma_{M,n} }
\newcommand{\Gamtwon}{ \Gamma_{2,n} }

\newcommand{\tPsi}{\tilde{\Psi}}
\newcommand{\eg}{\textit{e.g.}}
\newcommand{\ie}{\textit{i.e.}}

\newcommand{\e}{\mathrm{e}}

\newcommand{\bigformulatop}[3]{%
  \begin{figure*}[!t]
    \normalsize
    \setcounter{equation}{#1}
    #3

    \setcounter{equation}{#2}
    \hrulefill
    \vspace*{4pt}
  \end{figure*}
}

\IEEEoverridecommandlockouts

\begin{document}








\title{Degree-$M$ Bethe and Sinkhorn Permanent Based Bounds on
       the Permanent of a Non-negative Matrix}


\author{Yuwen~Huang,~\IEEEmembership{Graduate Student Member, IEEE}, 
        Navin~Kashyap,~\IEEEmembership{Senior Member, IEEE}, 
        and Pascal~O.~Vontobel,~\IEEEmembership{Fellow, IEEE}%
\thanks{
      Manuscript received 9 September 2023; accepted 13 February 2024; date of current version 6 March 2024.
      The work of N.~Kashyap was supported in part by the grant
      DST/INT/RUS/RSF/P-41/2021 from the Department of Science and Technology,
      Government of India.
      The work of P.~O.~Vontobel was supported in part by a grant from the
      Research Grants Council of the Hong Kong Special Administrative Region,
      China (Project No.\ CUHK 14208319).
      An earlier version of this paper was presented in part at the IEEE
      International Symposium on Information Theory (ISIT), Taipei, Taiwan,
      June~2023~[DOI: 10.1109/ISIT54713.2023.10206878].
      \textit{(Corresponding author: Yuwen Huang.)}
}
\thanks{Y.~Huang is with the 
          Department of Information Engineering,
          The Chinese University of Hong Kong, Hong Kong SAR
          (e-mail: hy018@ie.cuhk.edu.hk).}%
\thanks{N.~Kashyap is with the 
        Department of Electrical Communication Engineering, 
        Indian Institute of Science, Bengaluru 560012, 
        India (e-mail: nkashyap@iisc.ac.in).}
\thanks{P.~O.~Vontobel is with the 
        Department of Information Engineering and 
        the Institute of Theoretical Computer Science and Communications,
        The Chinese University of Hong Kong, Hong Kong SAR
        (e-mail: pascal.vontobel@ieee.org).}%
\thanks{Communicated by X. Wang, Associate Editor for Networking and Computation.}
}


\maketitle


\begin{abstract}
   The permanent of a non-negative square matrix can be well approximated by
  finding the minimum of the Bethe free energy function associated with some
  suitably defined factor graph; the resulting approximation to the permanent
  is called the Bethe permanent. Vontobel gave a combinatorial
  characterization of the Bethe permanent via degree-$M$ Bethe permanents,
  which are based on degree-$M$ covers of the underlying factor graph. In this
  paper, we prove a degree-$M$-Bethe-permanent-based lower bound on the
  permanent of a non-negative matrix, which solves a conjecture proposed by
  Vontobel in [IEEE Trans. Inf. Theory, Mar.\ 2013]. We also prove a
  degree-$M$-Bethe-permanent-based upper bound on the permanent of a
  non-negative matrix. In the limit $M \to \infty$, these lower and upper
  bounds yield known Bethe-permanent-based lower and upper bounds on the
  permanent of a non-negative matrix. Moreover, we prove similar results for
  an approximation to the permanent known as the (scaled) Sinkhorn permanent.
\end{abstract}


\begin{IEEEkeywords}
  Permanent,
  non-negative matrix,
  graphical model,
  Bethe approximation,
  Sinkhorn approximation.
\end{IEEEkeywords}


\section{Introduction}
\label{sec:introduction:1}


For any $n \in \sZpp$, define $[n] \defeq \{ 1, 2, \ldots, n \}$ and let
$\setS_{[n]}$ be the set of all $n!$ bijections from $[n]$ to $[n]$. Recall
the definition of the permanent of a real square matrix (see, \eg,
\cite{minc_marcus_1984}).


\begin{definition}\label{sec:1:def:13}
	Let $n \in \sZpp$ and let
  $ \mtheta \defeq \bigl( \theta(i,j) \bigr)_{\! i,j\in [n]}$ be a real matrix of
  size $n \times n$. The permanent of $ \mtheta $ is defined to be
  \begin{align}
    \perm(\mtheta)
      &\defeq
         \sum_{\sigma \in \setS_{[n]}} 
           \prod_{i \in [n]}
             \theta\bigl( i, \sigma(i) \bigr) \, .
               \label{eq:def:permanent:1} \\[-1.00cm] \nonumber
  \end{align}
  \edefinition
\end{definition}


In this paper, we consider only matrices $\mtheta$ where there is at least one
$\sigma \in \setS_{[n]}$ such that
$\prod_{i \in [n]} \theta\bigl( i, \sigma(i) \bigr) \neq 0$.



Computing the exact permanent is in the complexity class \#P, where \#P is the
set of the counting problems associated with the decision problems in the
class NP. Notably, already the computation of the permanent
of matrices that contain only zeros and ones is \#P-complete~\cite{Valiant1979}.


While being able to compute the exact permanent of a matrix might be desirable
in some circumstances, many applications require only a good approximation to
the permanent. In the rest of this paper, we focus on the important special
case of approximating the permanent of a non-negative real matrix, \ie, a
matrix where all entries are non-negative real numbers. In particular, we
study two graphical-model-based methods for approximating the permanent of a
non-negative matrix. The first method is motivated by Bethe-approximation /
sum-product-algorithm (SPA) based methods proposed in~\cite{Chertkov2008,
  Huang2009} and studied in more detail in~\cite{Vontobel2013a}. The second
method is motivated by Sinkhorn's matrix scaling algorithm~\cite{Linial2000}.


The main idea of these methods is to construct a factor graph~\cite{Forney2001, Loeliger2004a} whose partition
function equals the permanent.
Consequently, the permanent can be obtained by
finding the minimum of the Gibbs free energy function associated with the
factor graph. However, because finding the minimum of the Gibbs free energy
function is in general intractable unless $n$ is small, one considers the
minimization of some functions that suitably approximate the Gibbs free energy
function.


A first approximation to the Gibbs free energy function is given by the Bethe
free energy function~\cite{Yedidia2005}; the resulting approximation to
$\perm(\mtheta)$ is called the Bethe permanent and denoted by
$ \permb(\mtheta)$. Vontobel~\cite{Vontobel2013a} showed that for the relevant
factor graph, the Bethe free energy function is convex and that the SPA
converges to its minimum.
(This is in contrast to the situation for factor
graphs in general, where the Bethe free energy function associated with a
factor graph is not convex and the SPA is not guaranteed to converge to its
minimum.) The Bethe permanent $\permb(\mtheta)$ satisfies
\begin{align}
  1 
    &\leq 
       \frac{\perm(\mtheta)}{ \permb(\mtheta)} 
     \leq 
        2^{n/2}\, .
        \label{eq:ratio:permanent:bethe:permanent:1}
\end{align}
The first inequality in~\eqref{eq:ratio:permanent:bethe:permanent:1} was
proven by Gurvits~\cite{Gurvits2011} with the help of an inequality of
Schrijver~\cite{Schrijver1998}. Later on, an alternative proof based on
results from real stable polynomials was given by Straszak and Vishnoi
in~\cite{Straszak2019} and Anari and Gharan in~\cite{Anari2021}. The second
inequality in~\eqref{eq:ratio:permanent:bethe:permanent:1} was conjectured by
Gurvits~\cite{Gurvits2011} and proven by Anari and
Rezaei~\cite{Anari2019}. Both inequalities
in~\eqref{eq:ratio:permanent:bethe:permanent:1} are the best possible given
only the knowledge of the size of the matrix~$\mtheta$.

A second approximation to the Gibbs free energy function is given by the
Sinkhorn free energy function; the resulting approximation to $\perm(\mtheta)$
is called the Sinkhorn permanent and denoted by $\perms(\mtheta)$. (The
Sinkhorn free energy function was formally defined in~\cite{Vontobel2014}
based on some considerations in~\cite{Huang2009}. Note that
  $\perms(\mtheta)$ appears already in~\cite{Linial2000} as an approximation
  of $\perm(\mtheta)$ without being called the Sinkhorn permanent.) The
Sinkhorn free energy function is easily seen to be convex and it can be
efficiently minimized with the help of Sinkhorn's matrix scaling
algorithm~\cite{Linial2000}, thereby motivating its name. In this paper, we
will actually work with the scaled Sinkhorn permanent
$\permscs(\mtheta) \defeq \e^{-n} \cdot \perms(\mtheta)$, a variant of the
Sinkhorn permanent $\perms(\mtheta)$ that was introduced
in~\cite{N.Anari2021}. The scaled Sinkhorn permanent $\permscs(\mtheta)$
satisfies
\begin{align}
  \e^n \cdot \frac{n!}{n^n}
    &\leq 
       \frac{\perm(\mtheta)}{\permscs(\mtheta)}
     \leq 
       \e^n \, .
         \label{eq:ratio:permanent:scaled:Sinkhorn:permanent:1}
\end{align}
The first inequality in~\eqref{eq:ratio:permanent:scaled:Sinkhorn:permanent:1}
follows from van der Waerden's inequality proven by
Egorychev~\cite{Egorychev1981} and Falikman~\cite{Falikman1981}, whereas the
second equality follows from a relatively simple to prove inequality. Both
inequalities~\eqref{eq:ratio:permanent:scaled:Sinkhorn:permanent:1} are the
best possible given only the knowledge of the size of the
matrix $\mtheta$. Note that the value on the LHS
of~\eqref{eq:ratio:permanent:scaled:Sinkhorn:permanent:1} is approximately
$\sqrt{2 \pi n}$.


While the definition of $\perm(\mtheta)$ in~\eqref{eq:def:permanent:1} clearly
has a combinatorial flavor (\eg, counting weighted perfect matchings in a
complete bipartite graph with two times $n$ vertices), it is, a priori, not
clear if a combinatorial characterization can be given for $\permb(\mtheta)$
and $\permscs(\mtheta)$. However, this is indeed the case.


Namely, using the results in~\cite{Vontobel2013}, a finite-graph-cover-based
combinatorial characterization was given for $\permb(\mtheta)$
in~\cite{Vontobel2013a}. Namely,
\begin{align*}
  \permb(\mtheta)
    &= \limsup_{M \to \infty} \,
         \permbM{M}(\mtheta)\, ,
\end{align*}
where the degree-$M$ Bethe permanent is defined to be
\begin{align*}
  \permbM{M}(\mtheta)
    &\defeq
      \sqrt[M]{
        \bigl\langle
          \perm( \mtheta^{\uparrow \mP_{M}} )
        \bigr\rangle_{\mP_{M} \in \tPsi_{M} }
     } \, , 
\end{align*}
where the angular brackets represent an arithmetic average and where
$\mtheta^{\uparrow \mP_{M}}$ represents a degree-$M$ cover of $\mtheta$
defined by a collection of permutation matrices $\mP_{M}$. (See
Definition~\ref{def:matrix:degree:M:cover:1} for the details.)


In this paper we offer a combinatorial characterization of the scaled Sinkhorn
permanent $\permscs(\mtheta)$. Namely,
\begin{align*}
  \permscs(\mtheta)
    &= \limsup_{M \to \infty} \, 
         \permscsM{M}(\mtheta)\, , 
\end{align*}
where $\permscs(\mtheta)$ was defined
in~\cite{N.Anari2021}, and where
the degree-$M$ scaled Sinkhorn permanent is defined to be \\[-0.35cm]
\begin{align}
  \permscsM{M}(\mtheta)
    &\defeq 
       \sqrt[M]{
         \perm
           \bigl(
             \langle
               \mtheta^{\uparrow \mP_{M}}
             \rangle_{\mP_{M} \in \tPsi_{M} }
           \bigr)
       } \nonumber \\
    &= \sqrt[M]{
         \perm( \mtheta \otimes \mU_{M,M})
       } \, ,\nonumber
\end{align}
where $\otimes$ denotes the Kronecker product of two matrices and where
$ \mU_{M,M} $ is the matrix of size $M \times M$ with all entries equal to
$1/M$.  (See Definition~\ref{def:Sinkhorn:degree:M:cover:1} and
Proposition~\ref{sec:1:prop:14} for the details.)~\footnote{Note that the quantity
  $\sqrt[M]{ \perm( \mtheta \otimes \mU_{M,M}) }$ appears in the literature at
  least as early as~\cite{Bang1976,Friedland1979}. These papers show that for
  $M = 2^s$, $s \in \sZp$, it holds that
  $\sqrt[M]{ \perm( \mtheta \otimes \mU_{M,M}) } \leq \perm(\mtheta)$, which,
  in our notation, implies $1 \leq \perm(\mtheta) / \permscsM{M}(\mtheta)$ for
  $M = 2^s$, $s \in \sZp$.}%
~\footnote{Barvinok~\cite{Barvinok2010} considered combinatorial expressions
  in the same spirit as $\sqrt[M]{ \perm( \mtheta \otimes \mU_{M,M}) }$, but
  the details and the limit are different.}


\begin{figure}[t]
  \begin{center}
    \hspace{-2.25cm}\begin{tikzpicture}
\pgfmathsetmacro{\Ws}{0.9};
\begin{pgfonlayer}{main}
  \node (ZB1)     at (0,0) [] {\scriptsize$\qquad\qquad\qquad\quad\left. \ZBn{M}(\vtheta) \right|_{M = 1} = \perm(\vtheta)$};
  \node (ZBM)     at (0,\Ws) [] {\scriptsize$\ZBn{M}(\vtheta)$};
  \node (ZBinfty) at (0,2*\Ws) {\scriptsize$\qquad\qquad\qquad\qquad
  \left. \ZBn{M}(\vtheta) \right|_{M \to \infty} 
  = \permb(\vtheta)$};
  \draw[]
    (ZB1) -- (ZBM) -- (ZBinfty);
\end{pgfonlayer}
\end{tikzpicture}
    \hspace{-1.75cm}\begin{tikzpicture}
\pgfmathsetmacro{\Ws}{0.9};
\begin{pgfonlayer}{main}
  \node (ZscS1)     at (0,0) [] {\scriptsize$\qquad\qquad\qquad\quad\left. 
                               \permscsM{M}(\vtheta) \right|_{M = 1} 
                                 = \perm(\vtheta)$};
  \node (ZscSM)     at (0,\Ws) [] {\scriptsize$\permscsM{M}(\vtheta)$};
  \node (ZscSinfty) at (0,2*\Ws) {\scriptsize$\qquad\qquad\qquad\qquad
  \left. \permscsM{M}(\vtheta) \right|_{M \to \infty} 
  = \permscs(\vtheta)$};
  \draw[]
    (ZscS1) -- (ZscSM) -- (ZscSinfty);
\end{pgfonlayer}
\end{tikzpicture}
  \end{center}
  \caption{Combinatorial characterizations of the 
    Bethe and scaled Sinkhorn permanents.}
  \label{fig:combinatorial:characterization:1}
  \vspace{0.00cm} 
\end{figure}



These combinatorial characterizations of the Bethe and scaled Sinkhorn
permanent are summarized in Fig.~\ref{fig:combinatorial:characterization:1}.


Besides the quantities $\permbM{M}(\mtheta)$ and $\permscsM{M}(\mtheta)$ being
of inherent interest, they are particularly interesting toward understanding
the ratios $\perm(\mtheta) / \permb(\mtheta)$ and
$\perm(\mtheta) / \permscs(\mtheta)$ for general non-negative matrices or for
special classes of non-negative matrices.


For example, as discussed in~\cite[Sec.~VII]{Vontobel2013}, one can consider
the equation
\begin{align*}
  \hspace{-0.25cm}
  \frac{\perm(\mtheta)}{\permb(\mtheta)}
    &= \frac{\perm(\mtheta)}{\permbM{2}(\mtheta)}
       \!\cdot\!
       \frac{\permbM{2}(\mtheta) }{\permbM{3}(\mtheta)}
       \!\cdot\!
       \frac{\permbM{3}(\mtheta) }{\permbM{4}(\mtheta)}
       \!\cdots \, .
\end{align*}
If one can give bounds on the ratios appearing on the RHS of this equation,
then one obtains bounds on the ratio on the LHS of this equation.



Alternatively, one can consider the equation
\begin{align}
  \underbrace{
    \frac{\perm(\mtheta)}
         {\permb(\mtheta)}
  }_{\text{\ding{192}}}
    &= \underbrace{
         \frac{\perm(\mtheta)}
              {\permbM{2}(\mtheta)}
       }_{\text{\ding{193}}}
       \cdot
       \underbrace{
         \frac{\permbM{2}(\mtheta)}
              {\permb(\mtheta)}
       }_{\text{\ding{194}}} \, .
         \label{eq:perm:ratios:1}
\end{align}
Based on numerical and analytical results, Ng and Vontobel~\cite{KitShing2022}
concluded that it appears that for many classes of matrices of interest, the
ratio~\ding{193} behaves similarly to the ratio~\ding{194}. Therefore,
understanding the ratio~\ding{193} goes a long way toward understanding the
ratio~\ding{192}. (Note that the ratio~\ding{193} is easier to analyze than
the ratio~\ding{192} for some classes of matrices of interest.)


Results similar in spirit to the above results were given for the Bethe
approximation to the partition function of other graphical models (see, \eg,
\cite{NIPS2012_4649,Vontobel:16:1,Csikvari2022}).


\ifx\withclearpagecommands\x
\clearpage
\fi

\section{Main Contributions}
\label{sec:main:constributions:1}


This section summarizes our main contributions. All the details and proofs
will be given later on.


Let $n, M \in \sZpp$, let $\mtheta$ be a fixed non-negative matrix of size
$n \times n$, and let $\GamMn$ be the set of doubly stochastic matrices of
size $n \times n$ where all entries are integer multiples of $1/M$. Moreover,
for $\vgam \in \GamMn$, let
$\mtheta^{ M \cdot \vgam } \defeq \prod_{i,j \in [n]} \bigl( \theta(i,j)
\bigr)^{\! M \cdot \gamma(i,j)}$. (Note that in this expression, all exponents
are non-negative integers.) Throughout this paper, we consider $ \mtheta $
such that $ \perm(\mtheta) $, $ \permbM{M}(\mtheta) $, and
$ \permscsM{M}(\mtheta) $ are positive real-valued.


\begin{lemma}
  \label{lem: expression of permanents w.r.t. C}

  There are collections of non-negative real numbers
  \mbox{$\bigl\{ \CM{n}( \vgam ) \bigr\}_{\vgam \in \GamMn}$,\!
  $\bigl\{ \CBM{n}( \vgam ) \bigr\}_{\vgam \in \GamMn}$,\!
  $\bigl\{ \CscSgen{M}{n}( \vgam ) \bigr\}_{\vgam \in \GamMn}$} such that
  \begin{align}
    \bigl( \perm(\mtheta) \bigr)^{\! M} 
      &= \sum_{\vgam \in \GamMn} 
           \mtheta^{ M \cdot \vgam }
           \cdot
           \CM{n}( \vgam )\, , 
             \label{sec:1:eqn:43} \\
    \bigl( \permbM{M}(\mtheta) \bigr)^{\! M} 
      &= \sum_{\vgam \in \GamMn} 
           \mtheta^{ M \cdot \vgam }
           \cdot
           \CBM{n}( \vgam )\, ,
             \label{sec:1:eqn:190}  \\
    \bigl( \permscsM{M}(\mtheta) \bigr)^{\! M}
      &= \sum_{\vgam \in \GamMn}
           \mtheta^{ M \cdot \vgam } \cdot \CscSgen{M}{n}( \vgam )\, .
             \label{sec:1:eqn:168} \\[-0.75cm] \nonumber
  \end{align}
  \elemma
\end{lemma}


With the help of the inequalities
  in~\eqref{eq:ratio:permanent:bethe:permanent:1}
  and~\eqref{eq:ratio:permanent:scaled:Sinkhorn:permanent:1}, along with some
  results that are established in this paper, we can make the following
  statement about the coefficients $\CM{n}( \vgam )$, $\CBM{n}( \vgam )$, and
  $\CscSgen{M}{n}( \vgam )$ in Lemma~\ref{lem: expression of permanents
    w.r.t. C}.


\begin{theorem}
  \label{thm: inequalities for the coefficients}

  For every $\vgam \in \GamMn$, the coefficients $\CM{n}( \vgam )$,
  $\CBM{n}( \vgam )$, and $\CscSgen{M}{n}( \vgam )$ satisfy
  \begin{align}
    1 
      &\leq 
         \frac{ \CM{n}(\vgam) }{\CBM{n}(\vgam) }
       \leq
         \Bigl( 2^{n/2} \Bigr)^{\! M-1}\, ,
           \label{sec:1:eqn:58} \\
    \Biggl( \frac{M^{M}}{M!} \Biggr)^{\!\!\! n}
    \cdot
    \biggl(
      \frac{ n! }{ n^{n} }
    \biggr)^{\!\! M-1} 
      &\leq
         \frac{\CM{n}(\vgam)}{\CscSgen{M}{n}(\vgam)} 
       \leq 
         \Biggl( \frac{M^{M}}{M!} \Biggr)^{\!\! n}\, .
           \label{sec:1:eqn:180}
  \end{align}
  \etheorem
\end{theorem}

Combining Lemma~\ref{lem: expression of permanents w.r.t. C} and
Theorem~\ref{thm: inequalities for the coefficients}, we can make the
following statement.


\begin{theorem}
  \label{th:main:permanent:inequalities:1}

  It holds that
  \begin{align}
    1 
      &\leq
         \frac{\perm(\mtheta)}{\permbM{M}(\mtheta)}
       \leq
         \Bigl( 2^{n/2} \Bigr)^{\!\! \frac{M-1}{M}}
           \label{sec:1:eqn:147}\, , \\
    \frac{M^{n}}{(M!)^{n/M}}
    \cdot
    \biggl( \frac{n!}{n^{n}} \biggr)^{\!\! \frac{M-1}{M}}
      &\leq
         \frac{ \perm(\mtheta) }{ \permscsM{M}(\mtheta) }
       \leq
         \frac{M^{n}}{(M!)^{n/M}}\, .
           \label{sec:1:eqn:200}
  \end{align}
  Note that in the limit $M \to \infty$ we recover the inequalities
  in~\eqref{eq:ratio:permanent:bethe:permanent:1}
  and~\eqref{eq:ratio:permanent:scaled:Sinkhorn:permanent:1}.
  \etheorem
\end{theorem}

Finally, we discuss the following asymptotic characterization of the
coefficients $\CM{n}( \vgam )$, $\CBM{n}( \vgam )$, and
$\CscSgen{M}{n}( \vgam )$.


\begin{proposition}
  \label{prop:coefficient:asymptotitic:characterization:1}
  
  Let $\vgam \in \GamMn$. It holds that
  \begin{align}
    \CM{n}( \vgam )
      &= \exp\bigl( M \cdot \HG'(\vgam) + o(M) \bigr)\, ,
           \label{sec:1:eqn:125} \\
    \CBM{n}( \vgam )
      &= \exp\bigl( M \cdot \HBthe(\vgam) + o(M) \bigr)\, ,
           \label{sec:1:eqn:206}  \\
    \CscSgen{M}{n}( \vgam )
      &= \exp\bigl( M \cdot \HscSthe(\vgam) + o(M) \bigr)\, ,
           \label{sec:1:eqn:207}
  \end{align}
  where $\HG'(\vgam)$, $\HBthe(\vgam)$, and $\HscSthe(\vgam)$ are the
  modified Gibbs entropy function, the Bethe entropy function, and the
  scaled Sinkhorn entropy function, respectively, as defined in Section~\ref{sec:free:energy:functions:1}.

  \vspace{-0.02 cm}
  \eproposition
\end{proposition}


Let us put the above results into some context:
\begin{itemize}

\item The first inequality in~\eqref{sec:1:eqn:58} proves the first (weaker)
  conjecture in~\cite[Conjecture 52]{Vontobel2013a}.

\item Smarandache and Haenggi~\cite{Smarandache2016} claimed a proof of the
  second (stronger) conjecture
  in~\cite[Conjecture~52]{Vontobel2013a},\footnote{The second
      (stronger) conjecture in~\cite[Conjecture~52]{Vontobel2013a} stated that
      for each $ \mP_{M} \in \tPsi_{M} $, there is a collection of
      non-negative real numbers
      $\bigl\{ C_{\mathrm{B},M,n,\mP_{M}}( \vgam ) \bigr\}_{\vgam \in \GamMn}$
      such that
      \begin{align*}
        \perm( \mtheta^{\uparrow \mP_{M}} ) = \sum_{\vgam \in \GamMn}
      \mtheta^{ M \cdot \vgam } \cdot C_{\mathrm{B},M,n,\mP_{M}}( \vgam )\, ,
      \end{align*}
      and that for every $ \vgam \in \GamMn $, it holds that
      $ \CM{n}(\vgam) / C_{\mathrm{B},M,n,\mP_{M}}( \vgam ) \geq 1.  $}
  which would imply the first (weaker) conjecture in~\cite[Conjecture
  52]{Vontobel2013a}. However, the proof in~\cite{Smarandache2016} is
  incomplete.

\item The first inequality in~\eqref{sec:1:eqn:147} proves the first (weaker)
  conjecture in~\cite[Conjecture 51]{Vontobel2013a}.

\item The proof of the first inequality in~\eqref{sec:1:eqn:58}, and with that
  the proof of the first inequality in~\eqref{sec:1:eqn:147}, uses an
  inequality of 
  Schrijver~\cite{Schrijver1998}. By now, there are various
  proofs of Schrijver's inequality and also various proofs of the first
  inequality in~\eqref{eq:ratio:permanent:bethe:permanent:1} available in the
  literature, however, the ones that we are aware of are either quite involved
  (like Schrijver's original proof of his inequality) or use reasonably
  sophisticated machinery like real stable polynomials. Given the established
  veracity of the first inequality in~\eqref{sec:1:eqn:147}, it is hoped that
  a more basic proof for this inequality, and with that of the first
  inequality in~\eqref{eq:ratio:permanent:bethe:permanent:1}, can be
  found. Similar statements can be made for the other inequalities
  in~\eqref{sec:1:eqn:147} and~\eqref{sec:1:eqn:200}.

  In fact, for $M = 2$, a basic proof of the inequalities
  in~\eqref{sec:1:eqn:147} is available: it is based
  on~\cite[Proposition~2]{KitShing2022} (which appears in this paper as
  Proposition~\ref{prop:ratio:perm:permbM:2:1}), along with the observation
  that the quantity $c(\sigma_1,\sigma_2)$ appearing therein is bounded as
  $0 \leq c(\sigma_1,\sigma_2) \leq n/2$. Another basic proof of the first
  inequality in~\eqref{sec:1:eqn:147} is based on a straightforward
  generalization of the result in~\cite[Lemma 4.2]{Csikvari2017} from
  $\{0,1\}$-valued matrices to arbitrary non-negative matrices and then
  averaging over $2$-covers.

  Moreover, for the case where $\mtheta$ is proportional to the all-one matrix
  and where $M \in \sZpp$ is arbitrary, a basic proof of the first inequality
  in~\eqref{sec:1:eqn:147} was given in~\cite[Appendix~I]{Vontobel2013a}.
  However, presently it is not clear how to generalize this proof to other
  matrices $\mtheta$.

\item For $M = 2$, the lower and upper bounds in~\eqref{sec:1:eqn:147} are
  exactly the square root of the lower and upper bounds
  in~\eqref{eq:ratio:permanent:bethe:permanent:1}, respectively, thereby
  corroborating some of the statements made after~\eqref{eq:perm:ratios:1}.

\item Although we take advantage of special properties of the factor graphs
  under consideration, potentially the techniques in this paper can be
  extended toward a better understanding of the Bethe approximation to the
  partition function of other factor graphs. (For more on this
    topic, see the discussion in Section~\ref{sec:7}.)

\end{itemize}


The rest of this paper is structured as
follows. Section~\ref{sec:basic:definition:and:notations:1} introduces some
basic definitions and notations. Afterwards, Section~\ref{sec:1} gives a
factor graph representation of the permanent and
Section~\ref{sec:free:energy:functions:1} defines some free energy functions
associated with this factor graph. The main results, \ie, Lemma~\ref{lem:
  expression of permanents w.r.t. C}, Theorem~\ref{thm: inequalities for the
  coefficients}, Theorem~\ref{th:main:permanent:inequalities:1}, and
Proposition~\ref{prop:coefficient:asymptotitic:characterization:1}, are then
proven in
Sections~\ref{sec:recursions:coefficients:1}--\ref{sec:asymptotic:1}.
Section~\ref{apx:22} has a closer look at the special case $M = 2$ and
connects some of the results in this paper with results in other
papers. Finally, Section~\ref{sec:7} presents some conclusions
and open problems. Various (longer) proofs have been
collected in the appendices.


\ifx\withclearpagecommands\x
\clearpage
\fi

\section{Basic Definitions and Notations}
\label{sec:basic:definition:and:notations:1}


Let $n$ be a positive integer and let $\mtheta \in \sR_{\geq 0}^{n \times
  n}$. We define $\suppmtheta$ to be the support of $\mtheta$, \ie,
\begin{align*}
  \suppmtheta 
    &\defeq
       \bigl\{ 
         (i,j) \in [n] \times [n] 
       \bigm|
         \theta(i,j) > 0
       \bigr\}\, .
\end{align*}


Consider two finite subsets $ \set{I} $ and $ \set{J} $ of $[n]$ such that
$ |\set{I}| = | \set{J} | $. The set $ \setS_{\set{I} \to \set{J}} $ is
defined to be the set of all bijections from $ \set{I} $ to $ \set{J} $. In
particular, if $ \set{I} = \set{J} $, then the set
$ \setS_{\set{I} \to \set{J}} $ represents the set of all permutations of the
elements in $ \set{I} $, and we simply write $ \setS_{\set{I}} $ instead of
$\setS_{\set{I} \to \set{J}}$. Given
$ \sigma \in \setS_{\set{I} \to \set{J}} $, we define the matrix
$ \mP_{\sigma} \in \sZp^{\set{I} \times \set{J}} $ to be
\begin{align*}
  \mP_{\sigma} 
    &\defeq 
       \bigl(
         P_{\sigma}( i, j )
       \bigr)_{\! (i,j) \in \set{I} \times \set{J}} \, , 
       \ \ \ 
  P_{\sigma}( i, j )
     \defeq 
       \begin{cases}
         1 & \sigma(i) = j \\
         0 & \text{otherwise}
     \end{cases}\, .
     \nonumber
\end{align*}
Note that $\mP_{\sigma}$ is a permutation matrix, \ie, a matrix where all
entries are equal to $0$ except for exactly one $1$ per row and exactly one
$1$ per column. Moreover, given a set $ \setS_{\set{I} \to \set{J}} $, the set
$\setP_{\set{I} \to \set{J}}$ is defined to be the set of the associated
permutation matrices, \ie,
\begin{align}
  \setP_{\set{I} \to \set{J}}
    &\defeq
       \bigl\{ 
         \mP_{\sigma} 
       \bigm|
         \sigma \in \setS_{\set{I} \to \set{J}}
       \bigr\}\, . 
         \label{sec:1:eqn:181}
\end{align}
If $\set{I} = \set{J}$, we simply write $\setP_{\set{I}}$ instead of
$\setP_{\set{I} \to \set{J}}$. A particularly important special case for this
paper is the choice $\set{I} = \set{J} = [n]$: the set $\setP_{[n]}$
represents the set of all permutation matrices of size $n \times n$ where the
rows and columns are indexed by $[n]$.


We define $\Gamma_n$ to be the set of all doubly stochastic
matrices of size $n \times n$, \ie,
\begin{align*}
  \Gamma_n
    &\defeq
       \left\{ 
         \vgam
         = \bigl(
             \gamma(i,j)
           \bigr)_{\! i,j \in [n]}
      \ \middle| \ 
        \begin{array}{l}
          \gamma(i,j) \in \sR_{\geq 0}, \, \forall i,j \in [n] \\
          \sum\limits_{j \in [n]} \gamma(i,j) = 1, \, \forall i \in [n] \\
          \sum\limits_{i \in [n]} \gamma(i,j) = 1, \, \forall j \in [n]
        \end{array}
      \right\}\, .
\end{align*}
Moreover, for $M \in \sZpp$, we define $\Gamma_{M,n}$ to be the subset of
$\Gamma_{n}$ that contains all matrices where the entries are multiples of
$1/M$, \ie,
\begin{align*}
  \Gamma_{M,n} 
    &\defeq 
       \bigl\{ 
         \vgam \in \Gamma_n
       \bigm|
         M \cdot \gamma(i,j) \in \sZp, \ \forall i,j \in [n]
       \bigr\}\, .
\end{align*}
Observe that $\Gamma_{1,n} = \setP_{[n]}$.


All logarithms in this paper are natural logarithms and the value of
$0 \cdot \log(0)$ is defined to be $ 0 $.


In this paper, given an arbitrary $ n \in \sZpp $, if there is no ambiguity,
we use $ \prod_{i,j} $, $ \prod_{i} $, $ \prod_{j} $, $ \sum_{i,j} $,
$ \sum_{i} $, and $ \sum_{j} $ for $ \prod_{i,j \in [n]} $, $ \prod_{i \in [n]} $,
$ \prod_{j \in [n]} $, $ \sum_{i,j \in [n]} $, $ \sum_{i \in [n]} $, and
$ \sum_{j \in [n]} $, respectively.


\ifx\withclearpagecommands\x
\clearpage
\fi

\section{An S-NFG Representation of the Permanent}
\label{sec:1}


Factor graphs are powerful graphical tools to depict the factorization of
multivariate functions. Many well-known approximation methods, \eg, the
Bethe~\cite{Yedidia2005} and Sinkhorn\cite{Linial2000} approximations, are
formulated based on factor graphs. In this paper, we consider standard normal
factor graphs (S-NFGs), \ie, factor graphs where the local functions are
non-negative real-valued, and variables are associated with
edges~\cite{Forney2001, Loeliger2004a}.


The main purpose of this section is to present an S-NFG whose partition
function equals the permanent of a non-negative matrix~$\mtheta$ of size
$n \times n$. There are different ways to develop such
S-NFGs~\cite{Greenhill2010, Vontobel2013a}. (For a comparison between the ways
in~\cite{Greenhill2010, Vontobel2013a}, see\cite[Section
VII-E]{Vontobel2013a}.) In this paper, we consider the S-NFG $ \sfN(\mtheta) $
defined in~\cite{Vontobel2013a}. Throughout this section, let $\mtheta$ be a
fixed non-negative matrix of size $n \times n$.


\begin{definition}
  \label{sec:1:def:5}
  \label{def: details of S-NFG of permanent}

  We define the S-NFG
  $ \sfN(\mtheta) \defeq \sfN( \setF, \setEfull, \set{X}) $ as follows (see
  also Fig.~\ref{sec:1:fig:15} for the special case $n = 3$).
  \begin{enumerate}

  \item The set of vertices (also called function nodes) is
    \begin{align*}
      \set{F}
        &\defeq
           \{ \fr{i} \}_{i \in [n]} 
           \ \cup \ 
           \{ \fc{j} \}_{j \in [n]}\, .
    \end{align*}
    Here, the letter ``$\mathrm{r}$'' in $ \fr{i} $ means that $ \fr{i} $
    corresponds to the $ i $-th row of $\mtheta$, whereas the letter
    ``$\mathrm{c}$'' in $ \fc{j} $ means that $ \fc{j} $ corresponds to the
    $ j $-th column of $\mtheta$.

  \item The set of edges is defined to be
    \begin{align*}
      \setEfull
        &\defeq
           [n] \times [n] 
          = \bigl\{ 
              (i,j) 
            \bigm|
              i , j \in [n] 
            \bigr\}\, .
    \end{align*}

  \item The alphabet associated with edge $ e = (i,j) \in \setEfull $ is
    $ \setx{e} = \setx{i,j} \defeq \{0,1\} $.  

  \item The set $ \set{X} $ is defined to be
    \begin{align*}
      \set{X}
        &\defeq
           \prod_{e}
             \setx{e} 
         = \prod_{i,j}
             \setx{i,j}
    \end{align*}
    and is called the configuration set of $ \sfN(\mtheta) $.

  \item A matrix like
    \begin{align*}
      \vgamzo
        &\defeq
           \bigl( 
             \gamzo(i,j)
           \bigr)_{\! (i,j) \in [n] \times [n]}
             \in \setX
    \end{align*} 
    is used to denote an element of $ \set{X} $, \ie, a configuration of
    $ \sfN(\mtheta) $. Here, $ \gamma(i,j) $ is associated with edge
    $ e = (i,j) \in \setEfull$. In the following, the $i$-th row and the
    $j$-th column of $\vgamzo$ will be denoted by $ \vgamzo(i,:) $ and
    $ \vgamzo(:,j) $, respectively.

  \item For a given configuration $\vgamzo \in \setX$, the value of the local
    functions of $\sfN(\mtheta)$ are defined as follows. Namely, for every
    $ i \in [n] $, the local function $ \fr{i} $ is defined to
    be\footnote{Note that, with a slight abuse of notation, we use the same
      label for a function node and its associated local function.}
    \begin{align*}
      \fr{i}\bigl( \vgamzo(i,:) \bigr)
        &\defeq
           \begin{cases}
             \sqrt{ \theta(i,j) } & \vgamzo(i,:) = \vuj  \\
             0                    & \text{otherwise}
          \end{cases}\, .
    \end{align*}
    For every $ j \in [n] $, the local function $ \fc{j} $ is defined to be
    \begin{align*}
        \fc{j}\bigl( \vgamzo(:,j) \bigr) \defeq
        \begin{cases}
            \sqrt{ \theta(i,j) } & \vgamzo(:,j) = \vui \\
            0 & \text{otherwise}
        \end{cases}\, .
    \end{align*}
    Here we used the following notation: the vector $\vuj$ stands for a row
    vector of length $n$ whose entries are all equal to~$0$ except for the
    $j$-th entry, which is equal to~$1$. The column vector $\vui$ is defined
    similarly.
  
  \item For a given configuration $\vgamzo \in \setX$, the value of the global
    function $g$ of $\sfN(\mtheta)$ is defined to be the product of the values
    of the local functions, \ie,
    \begin{align*}
      g(\vgamzo)
        &\defeq
           \Biggl( 
             \prod_{i} 
               \fr{i}\bigl( \vgamzo(i,:) \bigr)
           \Biggr)
          \cdot
          \Biggl(
            \prod_{j}
              \fc{j}\bigl( \vgamzo(:,j) \bigr) 
          \Biggr) \, .
    \end{align*}
  
  \item The set of valid configurations of $ \sfN(\mtheta) $ is
    defined to be
    \begin{align*}
      \setA(\mtheta)
        &\defeq
           \setA\bigl( \sfN(\mtheta) \bigr)
         \defeq 
           \bigl\{
             \vgamzo \in \setX
           \bigm|
             g(\vgamzo) > 0
           \bigr\}\, .
    \end{align*}
    
  \item The partition function of $ \sfN(\mtheta) $ is defined to be
    \begin{align*}
      Z\bigl( \sfN(\mtheta) \bigr)
        &\defeq 
           \sum_{\vgamzo \in \setX }
             g(\vgamzo)
         = \sum_{\vgamzo \in \setA(\mtheta) }
             g(\vgamzo)\, .
    \end{align*}

  \end{enumerate}
  \vspace{-0.25cm}
  \edefinition
\end{definition}


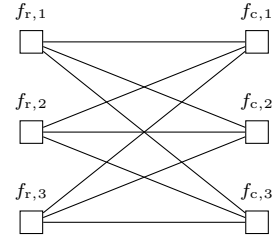
\begin{figure}[t]
  \centering
  \begin{tikzpicture}[node distance=2.2cm, on grid,auto,scale=1.50]
    \tikzstyle{state}=[shape=rectangle,fill=white,draw,minimum size=0.3cm]
    \pgfmathsetmacro{\ldis}{2} 
    \pgfmathsetmacro{\sdis}{0.5} 
    \pgfmathsetmacro{\scale}{0.8} 
    \begin{pgfonlayer}{glass}
        \foreach \i in {1,2,3}{
            \node[state] (f\i) at (0,-\i*\scale) [label=above: \scriptsize$\fr{\i}$] {};
            \node[state] (g\i) at (\ldis,-\i*\scale) [label=above: \scriptsize$\fc{\i}$] {};
        }
        \foreach \i in {1,2,3}{
            \foreach \j in {1,2,3}{
                \begin{pgfonlayer}{background}
                     \draw[]
                        (f\i) -- (g\j) ;
                \end{pgfonlayer}
            }
        }
    \end{pgfonlayer}
\end{tikzpicture}
\vspace{0.3cm}
  \caption{The S-NFG $ \sfN(\mtheta) $ for the special case $ n = 3 $.}
  \label{sec:1:fig:15}
  \vspace{-0.0cm}
\end{figure}


We make the following observations:
\begin{enumerate}

\item One can verify that $Z\bigl( \sfN(\mtheta) \bigr) = \perm(\mtheta)$.

\item If $ g(\vgamzo) > 0 $, then $ \vgamzo \in \setP_{[n]} $,
  \ie, $\vgamzo$ is a permutation matrix of size $n \times n$. (In
  fact, if all entries of $\mtheta$ are strictly positive, then
  $ g(\vgamzo) > 0 $ if and only if $ \vgamzo \in \setP_{[n]} $.)

\end{enumerate}


For the following definition, recall that $\Gamma_{n}$ is the set of all
doubly stochastic matrices of size $n \times n$ and that $\Gamma_{M,n}$ is the
set of all doubly stochastic matrices of size $n \times n$ where all entries
are multiples of $1/M$.


\begin{definition}
  \label{sec:1:def:10}

  Consider the S-NFG $ \sfN(\mtheta) $. Let $M \in \sZpp $.  We make the
  following definitions.
  \begin{enumerate}

  \item We define $\Gamnthe$ to be the set of matrices in $\Gamma_{n}$ whose
    support is contained in the support of $\mtheta$, \ie,
    \begin{align*}
      \Gamnthe
        &\defeq 
           \bigl\{ 
             \vgam \in \Gamma_{n}
           \bigm|
             \gamma(i,j) = 0 \text{ if } \theta(i,j) = 0
           \bigr\}\, .
    \end{align*}

  \item We define $\GamMnthe$ to be the set of matrices in $\Gamma_{M,n}$
    whose support is contained in the support of $\mtheta$, \ie,
     \begin{align*}
       \GamMnthe
         &\defeq 
            \bigl\{ 
              \vgam \in \Gamma_{M,n}
            \bigm|
              \gamma(i,j) = 0 \text{ if } \theta(i,j) = 0
            \bigr\}\, .
     \end{align*}

  \item We define $\PiA{\mtheta}$ to be the set of vectors representing
    probability mass functions over $\setP_{[n]}$ whose support is contained
    in $\setA(\mtheta)$, \ie,
    \begin{align*}
    &\!\!
    \PiA{\mtheta} \\
      &\!\!\defeq
         \left\{ 
           \vp = \bigl( p(\mP_{\sigma}) \bigr)_{\! \mP_{\sigma} \in \setP_{[n]}}
        \ \middle| \!
          \begin{array}{c}
            p(\mP_{\sigma}) \geq 0, \, 
              \forall \mP_{\sigma} \in \setP_{[n]} \\
            p(\mP_{\sigma}) = 0, \, 
              \forall \mP_{\sigma} 
                        \notin \set{A}(\mtheta) \\
            \sum\limits_{\mP_{\sigma} \in \setP_{[n]}} p(\mP_{\sigma}) = 1
          \end{array} \!
        \right\}\, .
    \end{align*}

  \item Let $\vgam \in \Gamma_n$. We define
    \begin{align*}
      \PiA{\mtheta}( \vgam )
        &\defeq 
          \left\{ 
              \vp 
              \in \PiA{\mtheta}
          \ \middle| \ 
              \sum_{\mP_{\sigma} \in \setP_{[n]}}
                p(\mP_{\sigma}) \cdot \mP_{\sigma} = \vgam
          \right\}\, .
    \end{align*}

   \item For any matrix $ \matr{B} \in \sR^{n \times n}_{\geq 0} $, we define
   \begin{align*}
       \setS_{[n]}(\matr{B})
         &\defeq 
            \bigl\{ 
              \sigma \in \setS_{[n]} 
            \bigm|
              B\bigl( i, \sigma(i) \bigr) > 0, \, \forall i \in [n]
            \bigr\} \\
         &= \left\{ 
              \sigma \in \setS_{[n]} 
            \ \middle| \ 
              \prod_i 
                B\bigl( i, \sigma(i) \bigr) > 0
            \right\}\, .
   \end{align*}
   
   \item We define 
     \begin{align*}
       \vsigma_{[M]}
         &\defeq
            (\sigma_{m})_{m \in [M]} \in (\setS_{[n]})^{M}\, , \\
       \mP_{\vsigma_{[M]}}
         &\defeq
            ( \mP_{\sigma_{m}} )_{ m \in [M] } \in (\setP_{[n]})^{M}\, .
     \end{align*}

   \item Let $ \vgam \in \Gamma_{M,n} $. We define
     \begin{align*}
       \mtheta^{ M \cdot \vgam }
         &\defeq
            \prod_{i,j}
              \bigl(
                \theta(i,j)
              \bigr)^{\! M \cdot \gamma(i,j)}\, .
     \end{align*}

   \item Let $\pmtheta$ be the probability mass function on $\setS_{[n]}$
     induced by~$\mtheta$, \ie,
     \begin{align*}
       \pmtheta(\sigma)
         &\defeq
            \frac{ \prod\limits_{i \in [n]} \theta\bigl( i,\sigma(i) \bigr)}
                 { \perm(\mtheta) }
          = \frac{ \mtheta^{\mP_{\sigma}}}
                 { \perm(\mtheta) }\, , 
                   \quad \sigma \in \setS_{[n]}\, .
     \end{align*}
     Clearly, $\pmtheta(\sigma) \geq 0$ for all $\sigma \in \setS_{[n]}$ and
     $\sum\limits_{\sigma \in \setS_{[n]}} \pmtheta(\sigma) = 1$.

   \end{enumerate}
  \edefinition
\end{definition}

Note that if $\matr{B}$ is a doubly stochastic matrix of size $n \times n$, then
  the set $\setS_{[n]}(\matr{B})$ is non-empty. Indeed, this follows from
  considering
   \begin{align*}
      \perm( \matr{B} ) 
      = \sum_{\sigma \in \setS_{[n]}(\matr{B})} 
      \prod_{i}
      B\bigl( i, \sigma(i) \bigr) 
      \overset{(a)}{\geq} n!/n^{n} > 0\, ,
   \end{align*}
   where step $(a)$ follows from van der Waerden's inequality, which was
   proven in~\cite[Theorem 1]{Egorychev1981} and~\cite[Theorem
   1]{Falikman1981}.

\begin{lemma}
  \label{sec:1:lemma:8}

  The set of valid configurations $ \setA(\mtheta) $ of $\sfN(\mtheta)$
  satisfies
  \begin{align*}
    \setA(\mtheta) = 
      \left\{ 
         \left. 
         \mP_{\sigma} \in \setP_{[n]}
      \ \right| \ 
          \sigma \in \setS_{[n]}(\vtheta)
      \right\}\, , 
  \end{align*}
  which implies a bijection between $ \setA(\mtheta) $ and
  $ \setS_{[n]}(\vtheta) $.
\end{lemma}


\begin{proof}
  This result follows from the Observation~2 after
  Definition~\ref{sec:1:def:5}.
\end{proof}


\ifx\withclearpagecommands\x
\clearpage
\fi

\section{Free Energy Functions Associated with 
               \texorpdfstring{$\sfN(\mtheta)$}{}}
\label{sec:free:energy:functions:1}


In this section, we introduce various free energy functions associated with
$\sfN(\mtheta)$. First we introduce the Gibbs free energy function, whose
minimum yields $\perm(\mtheta)$. Afterwards, we introduce the Bethe free
energy function and the scaled Sinkhorn free energy function, which are
approximations of the Gibbs free energy functions. Their minima yield the
Bethe permanent $\permb(\mtheta)$ and the scaled Sinkhorn permanent
$\permscs(\mtheta)$, respectively, which are approximations of
$\perm(\mtheta)$. Finally, we introduce the degree-$M$ Bethe permanent
$\permbM{M}(\mtheta)$ and the degree-$M$ scaled Sinkhorn permanent
$\permscsM{M}(\mtheta)$.


\begin{definition}
  \label{sec:1:def:18}

  Consider the S-NFG $ \sfN(\mtheta) $. The Gibbs free energy function
  associated with $ \sfN(\mtheta) $ (see~\cite[Section III]{Yedidia2005}) is
  defined to be
  \begin{align*}
    \FGthe: \ \PiA{\mtheta} &\to \sR \\
              \vp           &\mapsto \UGthe(\vp) - \HG(\vp)
  \end{align*}
  with
  \begin{align*}
    \UGthe(\vp)   
      &\defeq 
         - 
         \sum_{\mP_{\sigma} \in \setP_{[n]}}
           p( \mP_{\sigma} ) \cdot \log \bigl( g( \mP_{\sigma} ) \bigr)\, , \\
    \HG(\vp)   
      &\defeq 
         - 
         \sum_{\mP_{\sigma} \in \setP_{[n]}} 
         p( \mP_{\sigma} ) \cdot \log \bigl( p( \mP_{\sigma} ) \bigr)\, ,
  \end{align*}
  where $ \UGthe $ and $ \HG $ are called the Gibbs average
  energy function and the Gibbs entropy function, respectively.\footnote{Note
    that $\set{P}_{[n]} \subseteq \set{X}$, where the configuration set
    $\set{X}$ of $\sfN(\mtheta)$ was defined in Definition~\ref{def: details
      of S-NFG of permanent}. With this, $\mP_{\sigma} \in \setP_{[n]}$ is a
    configuration of $\sfN(\mtheta)$ and $g(\mP_{\sigma})$ is well defined.}
  \edefinition
\end{definition}


One can verify that
\begin{align*}
  Z\bigl( \sfN(\mtheta) \bigr) 
    &= \exp
         \biggl( 
           - \min_{\vp \in \PiA{\mtheta}} \FGthe(\vp)
         \biggr)\, .
\end{align*}


Recall the Birkhoff--von Neumann theorem (see, \eg,~\cite[Theorem
4.3.49]{Horn2012}), which states that the convex hull of all permutation
matrices of size $n \times n$ equals the set of all doubly stochastic matrices
of size $ n \times n $. Because $\vp \in \PiA{\mtheta}$ is a probability mass
function over the set of all permutation matrices of size $n \times n$ with support contained in the support of $ \mtheta $, the
Gibbs free energy function can be reformulated as follows as a function over
$\Gamma_{n}(\mtheta)$ instead of over~$\PiA{\mtheta}$.



\begin{definition}
  \label{sec:1:def:18:part:2}

  Consider the S-NFG $ \sfN(\mtheta) $. The modified Gibbs free energy
  function associated with $ \sfN(\mtheta) $ is defined to be
  \begin{align*}
    \FGthe': \ \Gamma_{n}(\mtheta) &\to \sR \\
               \vgam               &\mapsto \UGthe'(\vgam) - \HG'(\vgam)
  \end{align*}
  with
  \begin{align*}
    \UGthe'(\vgam)
      &\defeq 
         - 
         \sum_{ i,j }  
           \gamma(i,j) \cdot \log\bigl( \theta(i,j) \bigr)\, , \\ 
    \HG'(\vgam)
      &\defeq
         \max_{ \vp \in \PiA{\vgam}( \vgam ) }
           \HG(\vp)\, ,
  \end{align*}
  where $ \UGthe' $ and $ \HG' $ are called the modified Gibbs average energy
  function, and the modified Gibbs entropy function, respectively.%
  \footnote{\label{footnote:PiA:support:condition:1}%
    Note that the definition
    $\HG'(\vgam) \defeq \max_{ \vp \in \PiA{\mtheta}( \vgam ) } \HG(\vp)$
    would have been more natural/straightforward. However, we prefer the
    definition
    $\HG'(\vgam) \defeq \max_{ \vp \in \PiA{\vgam}( \vgam ) } \HG(\vp)$, as it
    does not depend on $\mtheta$ and, most importantly is equivalent to the
    former definition for $\vgam \in \Gamma_{n}(\mtheta)$. This equivalence
    stems from the following observation. Namely, let
    $\vp \in \PiA{\mtheta}( \vgam )$. Then $p(\mP_{\sigma}) > 0$ only if
    $\supp(\mP_{\sigma}) \subseteq \supp(\vgam)$.} \\
  \mbox{} \edefinition
\end{definition}



One can verify that
\begin{align*}
  Z\bigl( \sfN(\mtheta) \bigr) 
    &= \exp
     \biggl( 
       - \min_{\vgam \in \Gamnthe} \FGthe'(\vgam)
     \biggr)\, .
\end{align*}


However, neither formulation of the Gibbs free energy function leads to a
tractable optimization problem in general unless~$n$ is small. This motivates
the study of approximations of the permanent based on approximating the Gibbs
free energy function, especially based on approximating the modified Gibbs
free energy function in the reformulation presented in
Definition~\ref{sec:1:def:18:part:2}. In the following, we will first discuss
an approximation known as the Bethe free energy function and then an
approximation known as the scaled Sinkhorn free energy function.


In general, the Bethe free energy function associated with an S-NFG is a
function over the so-called local marginal polytope associated with the
S-NFG. (For a general discussion of the local marginal polytope associated
with an S-NFG, see, \eg, \cite[Section 4]{Wainwright2008}.) The paper
\cite[Section~IV]{Vontobel2013a} shows that the local marginal polytope of the
$\sfN(\mtheta)$ can be parameterized by the set of doubly stochastic
matrices. This is a consequence of the Birkhoff--von Neumann theorem.


\begin{definition}
  \label{sec:1:def:16}
    
  Consider the S-NFG $ \sfN(\mtheta) $. The Bethe free energy function
  associated with $ \sfN(\mtheta) $ (see~\cite[Corollary 15]{Vontobel2013a})
  is defined to be 
  \begin{align*}
    \FBthe: \ \Gamma_{n}(\mtheta) &\to \sR \\
              \vgam              &\mapsto \UBthe( \vgam ) - \HBthe( \vgam )
  \end{align*}
  with
  \begin{align*}
    \UBthe(\vgam)
      &\defeq
        -\sum_{ i,j }
        \gamma(i,j) \cdot \log\bigl( \theta(i,j) \bigr)\, , \\
    \HBthe( \vgam )
      &\defeq 
        -
        \sum_{ i,j }
          \gamma(i,j) \cdot \log\bigl(  \gamma(i,j) \bigr) \\
      &\quad\ 
        + 
        \sum_{ i,j }
          \bigl( 1 - \gamma(i,j) \bigr) 
            \cdot 
            \log\bigl( 1 - \gamma(i,j) \bigr)\, .
  \end{align*}
  Here, $ \UBthe $ and $ \HBthe $ are called the Bethe average energy function
  and the Bethe entropy function, respectively. With this, the Bethe permanent
  is defined to be
  \begin{align}
    \permb(\mtheta)
      &\defeq 
         \exp
           \biggl(
             - \min_{\vgam \in \Gamnthe} \FBthe( \vgam ) 
           \biggr)\, .
        \label{sec:1:eqn:192}
  \end{align}
  \edefinition
\end{definition}


The degree-$M$ Bethe permanent of $\mtheta$ was defined
in~\cite{Vontobel2013a} and yields a combinatorial characterization of the
Bethe permanent. (Note that the definition of $\permb(\mtheta)$
in~\eqref{sec:1:eqn:192} is analytical in the sense that $\permb(\mtheta)$ is
given by the solution of an optimization problem.)


\begin{definition}
  \label{def:matrix:degree:M:cover:1}

  Consider the S-NFG $ \sfN(\mtheta) $. Let $M \in \sZpp$.  The degree-$M$
  Bethe permanent is defined to be
  \begin{align*}
    \permbM{M}(\mtheta)
      &\defeq
        \sqrt[M]{
          \bigl\langle
            \perm( \mtheta^{\uparrow \mP_{M}} )
          \bigr\rangle_{\mP_{M} \in \tPsi_{M} }
       } \, , 
  \end{align*}
  where 
  \begin{align*}
    \tPsi_{M}
      &\defeq 
        \left\{ 
          \mP_{M} \defeq \Bigl( \mP^{(i,j)} \Bigr)_{\!\! i,j\in [n]}
        \ \middle| \ 
          \mP^{(i,j)} \in \setP_{[M]}
        \right\}\, ,
  \end{align*}
  where $ \setP_{[M]} $ is defined in~\eqref{sec:1:eqn:181} and the
  $ \mP_{M} $-lifting of $ \mtheta $ is defined to be a real-valued matrix of size $ M n $-by-$Mn$:
  \begin{align*}
    \!\!
    \mtheta^{\uparrow \mP_{M}}
      &\defeq 
        \begin{pmatrix}
            \theta(1,1) \cdot  \mP^{(1,1)} 
            & \cdots & \theta(1,n) \cdot  \mP^{(1,n)} \\
            \vdots & \ddots & \vdots \\
            \theta(n,1) \cdot  \mP^{(n,1)} 
            & \cdots & \theta(n,n) \cdot  \mP^{(n,n)} 
           \end{pmatrix} 
        \, .
  \end{align*}
  \edefinition
\end{definition}


\begin{proposition}
  \label{sec:1:thm:8}
    
  Consider the S-NFG $ \sfN(\mtheta) $. It holds that
  \begin{align}
    \limsup_{M \to \infty} \,
      \permbM{M}(\mtheta)
      &= \permb(\mtheta)\, . 
            \label{sec:1:eqn:193}
  \end{align}
  The limit in~\eqref{sec:1:eqn:193} is visualized in
  Fig.~\ref{fig:combinatorial:characterization:1}~(left).
\end{proposition}


\begin{proof}
  See~\cite[Section~VI]{Vontobel2013a}.
\end{proof}


The Sinkhorn free energy function was defined in~\cite{Vontobel2014}. Here we
consider a variant called the scaled Sinkhorn free energy function that was
introduced in~\cite{N.Anari2021}. These free energy functions are also defined
over $\Gamma_{n}(\mtheta)$.


\begin{definition} 
  \label{sec:1:def:20}

  Consider the S-NFG $ \sfN(\mtheta) $. The scaled Sinkhorn free energy
  function associated with $ \sfN(\mtheta) $ is defined to be
  \begin{align*}
    \FscSthe: \ \Gamma_{n}(\vtheta) &\to \sR \\
                \vgam          &\mapsto \UscSthe( \vgam ) - \HscSthe( \vgam )
  \end{align*}
  with
  \begin{align*}
    \UscSthe(\vgam) 
      &\defeq
         - 
         \sum_{ i,j }
           \gamma(i,j) \cdot \log\bigl( \theta(i,j) \bigr)\, , \\
    \HscSthe(\vgam)
      &\defeq 
         - \,
         n
         -
         \sum_{ i,j }
           \gamma(i,j) \cdot 
             \log\bigl( \gamma(i,j) \bigr)\, .
    \end{align*}
    Here, $ \UscSthe $ and $ \HscSthe $ are called the scaled Sinkhorn average
    energy function and the scaled Sinkhorn entropy function,
    respectively. With this, the scaled Sinkhorn permanent is defined to be
    \begin{align}
      \permscs(\mtheta)
        &\defeq 
           \exp
             \biggl(
               - \min_{\vgam \in \Gamnthe} \FscSthe( \vgam ) 
             \biggr)\, .
          \label{sec:1:eqn:192:part:2}
    \end{align}
    \edefinition
\end{definition}


Note that $\HscSthe(\vgam)$ can be obtained from $\HBthe( \vgam )$ by
approximating $\bigl( 1 - \gamma(i,j) \bigr) \cdot \log( 1 - \gamma(i,j) )$ by
the first-order Taylor series expansion term $- \gamma(i,j)$ and using the
fact that the sum of all entries of the matrix
$ \vgam \in \Gamma_{n}(\vtheta) $ equals $n$.


In the following, we introduce the degree-$M$ scaled Sinkhorn permanent of
$\mtheta$, which yields a combinatorial characterization of the scaled
Sinkhorn permanent of $\mtheta$. (Note that the definition of
$\permscs(\mtheta)$ in~\eqref{sec:1:eqn:192:part:2} is analytical in the sense
that $\permscs(\mtheta)$ is given by the solution of an optimization problem.)


\begin{definition}
  \label{def:Sinkhorn:degree:M:cover:1}

  Consider the S-NFG $ \sfN(\mtheta) $. Let $M \in \sZpp$.  The degree-$M$
  scaled Sinkhorn permanent is defined to be
  \begin{align*}
    \permscsM{M}(\mtheta)
      &\defeq 
         \sqrt[M]{
           \perm
             \Bigl(
               \bigl\langle
                 \mtheta^{\uparrow \mP_{M}}
               \bigr\rangle_{\mP_{M} \in \tPsi_{M} }
             \Bigr)
         } \nonumber \\
      &= \sqrt[M]{
           \perm\bigl( \mtheta \otimes \mU_{M,M}\bigr)
         } \, , 
  \end{align*}
  where $\otimes$ denotes the Kronecker product of two matrices and where
  $ \mU_{M,M} $ is the matrix of size $M \times M$ with all entries equal to
  $1/M$.\edefinition
\end{definition}


\begin{proposition}
  \label{sec:1:prop:14}
    
  Consider the S-NFG $ \sfN(\mtheta) $. It holds that
  \begin{align}
    \limsup_{M \to \infty} \,
      \permscsM{M}(\mtheta)
      &= \permscs(\mtheta)\, . 
            \label{sec:1:eqn:193:part:2}
  \end{align}
  The limit in~\eqref{sec:1:eqn:193:part:2} is visualized in
  Fig.~\ref{fig:combinatorial:characterization:1}~(right).
\end{proposition}


\begin{proof}
  See Appendix~\ref{apx:24}. (Note that this proof needs some results from
  Lemma~\ref{lem: expression of permanents w.r.t. C} and
  Proposition~\ref{prop:coefficient:asymptotitic:characterization:1}, which
  are proven only in the upcoming sections. Of course, Lemma~\ref{lem:
    expression of permanents w.r.t. C} and
  Proposition~\ref{prop:coefficient:asymptotitic:characterization:1} are
  proven independently of the statement in Proposition~\ref{sec:1:prop:14}.)
\end{proof}


Except for small values of $n$, computing $ \permbM{M}(\mtheta) $ and
$ \permscsM{M}(\mtheta) $ appears to be intractable in general for finite
$M$.\footnote{More formally, we leave it as an open problem to show that
  computing $ \permbM{M}(\mtheta) $ and $ \permscsM{M}(\mtheta) $ is in the
  complexity class \#P for finite $ M \in \sZpp $.} However, as $ M $ goes to
infinity, the limit superior of $ \permbM{M}(\mtheta) $ and
$ \permscsM{M}(\mtheta) $ are equal to $ \permb(\mtheta) $ and
$ \permscs(\mtheta) $, respectively, and they can be computed
efficiently~\cite{Vontobel2013a,Linial2000}.

Note that the modified Gibbs entropy function $ \HG'(\vgam) $, the Bethe entropy function  $ \HBthe( \vgam ) $, and the scaled Sinkhorn entropy function $ \HscSthe(\vgam) $, as defined in Definitions~\ref{sec:1:def:18:part:2},~\ref{sec:1:def:16}, and~\ref{sec:1:def:20}, respectively, are well defined for $ \vgam \in \Gamma_n $, not just for $ \vgam \in \Gamma_n(\vtheta) $. This observation allows us to consider $ \vgam \in \Gamma_{M,n} $ 
in Proposition~\ref{prop:coefficient:asymptotitic:characterization:1}, \ie, $\vgam$ is an element of a set that is independent of $\mtheta$.


\ifx\withclearpagecommands\x
\clearpage
\fi

\section{Recursions for \texorpdfstring{the \\
              Coefficients
              $\CM{n}( \vgam )$,
              $\CBM{n}( \vgam )$, 
              $\CscSgen{M}{n}( \vgam )$}{}}
\label{sec:recursions:coefficients:1}


In the first part of this section, we prove Lemma~\ref{lem: expression of
  permanents w.r.t. C}. Afterwards, we prove some lemmas that will help us in
the proofs of Theorems~\ref{thm: inequalities for the coefficients}
and~\ref{th:main:permanent:inequalities:1} in the next section. These latter
lemmas are recursions (in $M$) on the coefficients $\CM{n}( \vgam )$,
$\CBM{n}( \vgam )$, and $\CscSgen{M}{n}( \vgam )$. Note that these recursions
on the coefficients $\CM{n}( \vgam )$, $\CBM{n}( \vgam )$, and
$\CscSgen{M}{n}( \vgam )$ are somewhat analogous to recursions for multinomial
coefficients, in particular also recursions for binomial coefficients like the
well-known recursion ${M \choose k} = {M-1 \choose k-1} + {M-1 \choose
  k}$. (See also the discussion in
Example~\ref{example:pascal:triangle:generalization:1}).


Throughout this section, the positive integer $n$ and the matrix
$\mtheta \in \sR_{\geq 0}^{n \times n}$ are fixed. Recall the definition of
the notation $\vsigma_{[M]}$ in Definition~\ref{sec:1:def:10}.


\begin{definition}
  \label{sec:1:lem:43}

  Consider $M \in \sZpp $ and $ \vgam \in \Gamma_{M,n} $.
  \begin{enumerate}

  \item The coefficient $ \CM{n}( \vgam ) $ is defined to be the
    number of $ \vsigma_{[M]} $ in $ \setS_{[n]}^{M} $ such that
    $ \vsigma_{[M]} $ decomposes $ \vgam $, \ie,
    \begin{align}
      \CM{n}( \vgam )
        &\defeq
           \sum_{ \vsigma_{[M]} \in \setS_{[n]}^{M} }
             \Bigl[ 
               \vgam = \bigl\langle \mP_{\sigma_{m}} \bigr\rangle_{m \in [M]}
             \Bigr]\, ,
               \label{sec:1:eqn:56}
    \end{align}
    where
    $\bigl\langle \mP_{\sigma_{m}} \bigr\rangle_{m \in [M]} \defeq \frac{1}{M}
    \cdot \sum\limits_{m\in [M]} \mP_{\sigma_{m}}$ and where the notation
    $[S]$ represents the Iverson bracket, \ie, $[S] \defeq 1$ if the statement
    $S$ is true and $[S] \defeq 0$ if the statement is false.
    
  \item The coefficient $ \CBM{n}(\vgam) $ is defined to be
    \begin{align*}
      \CBM{n}(\vgam) 
        &\defeq
           (M!)^{ 2n -  n^2} 
           \cdot
           \prod_{i,j}
             \frac{ \bigl(M - M \cdot \gamma(i,j) \bigr)! }
                  { \bigl(M \cdot \gamma(i,j) \bigr)! } \, .
    \end{align*}
    
  \item The coefficient $ \CscSgen{M}{n}(\vgam) $ is defined to be
    \begin{align*}
      \CscSgen{M}{n}(\vgam)
        &\defeq
           M^{-n\cdot M} 
           \cdot
           \frac{ (M!)^{2n} }
                { \prod_{i,j}
                    \bigl( M \cdot \gamma(i,j) \bigr)! 
                }\, .
    \end{align*}

  \end{enumerate}
  Note that for $M = 1$ it holds that
  \begin{align*}
    \Cgen{1}{n}(\vgam) 
      &= \CBgen{1}{n}(\vgam) 
       = \CscSgen{1}{n}(\vgam) 
       = 1\, ,
           \quad \vgam \in \Gamma_{1,n}\, .
  \end{align*}
  \edefinition
\end{definition}


We are now in a position to prove Lemma~\ref{lem: expression of permanents
  w.r.t. C} in Section~\ref{sec:main:constributions:1}.


\medskip

\noindent
\textbf{Proof of Lemma~\ref{lem: expression of permanents w.r.t. C}}.
\begin{enumerate}

\item The expression in~\eqref{sec:1:eqn:43} follows from
  \begin{align*}
    \bigl( \perm(\mtheta) \bigr)^{\! M}
      &= \Biggl(
           \sum_{\sigma \in \setS_{[n]}} 
             \prod_{i \in [n]}
               \theta\bigl( i, \sigma(i) \bigr)
         \Biggr)^{\!\!\! M} \\
      &\overset{(a)}{=}
         \sum_{\vgam \in \GamMnthe} 
           \mtheta^{ M \cdot \vgam }
           \cdot
           \CM{n}( \vgam )\, , 
  \end{align*}
  where step~$(a)$ follows from expanding the right-hand side of the first
  line, along with using the Birkhoff--von Neumann theorem and
  Eq.~\eqref{sec:1:eqn:56}.

  While it is clear that $\CM{n}( \vgam ) \geq 0$ for all
    $\vgam \in \GamMnthe$, note that the first inequality
    in~\eqref{sec:1:eqn:58} in Theorem~\ref{thm: inequalities for the
      coefficients} (proven later in this paper) can be used to show that
    $\CM{n}( \vgam ) > 0$ for all $\vgam \in \GamMnthe$.

\item 

    Eq.~\eqref{sec:1:eqn:190} follows from~\cite[Lemma~29]{Vontobel2013},
    where the role of the set $\set{B}'_M$ in~\cite[Lemma~29]{Vontobel2013} is
    played by the set $\Gamma'_{M,n}(\mtheta)$ in this paper.

    The set $\Gamma'_{M,n}(\mtheta)$ is defined as follows. Namely, as
    mentioned in Section~\ref{sec:free:energy:functions:1}, the paper
    \cite[Section~IV]{Vontobel2013a} shows that the local marginal polytope of
    $\sfN(\mtheta)$ can be parameterized by the set of doubly stochastic
    matrices, more precisely it can be parameterized by
    $\Gamma_{n}(\mtheta)$. For arbitrary $ M \in \sZ_{\geq 1} $, the set
    $\Gamma'_{M,n}(\mtheta) \subseteq \Gamma_{n}(\mtheta)$ is then defined to
    be the parameterization of the set of pseudo-marginal vectors in the local
    marginal polytope that are $M$-cover lift-realizable pseudo-marginal
    vectors. (For the definition of $M$-cover lift-realizable pseudo-marginal
    vectors, see~\cite[Section~VI]{Vontobel2013}.)

    We claim $\Gamma'_{M,n}(\mtheta) = \Gamma_{M,n}(\mtheta)$. Indeed, it
    follows from the construction of pseudo-marginal vectors based on
    $M$-covers that $\Gamma'_{M,n}(\mtheta) \subseteq \Gamma_{M,n}(\mtheta)$,
    and it follows from the construction in the proof
    of~\cite[Lemma~29]{Vontobel2013} that for every element of
    $\Gamma_{M,n}(\mtheta)$ the corresponding pseudo-marginal vector is
    $M$-cover lift-realizable, \ie,
    $\Gamma'_{M,n}(\mtheta) \supseteq \Gamma_{M,n}(\mtheta)$.

\item Eq.~\eqref{sec:1:eqn:168} follows
  from~\cite[Theorem~2.1]{Friedland1979}.

\elemmaproof

\end{enumerate}


The matrices that are introduced in the following definition will turn out to
be useful to formulate the upcoming results.


\begin{definition}    
  \label{sec:1:def:10:part:2}
  \label{sec:1:def:15}

  Consider a matrix $ \vgam \in \Gamma_n $.
  \begin{enumerate}
    
  \item We define
    \begin{align*}
      \setR 
        &\defeq 
           \setR(\vgam)
         \defeq 
           \bigl\{ 
             i \in [n] 
           \bigm|
             \exists j\in [n], \text{ s.t. } 0 < \gamma(i,j) < 1 
           \bigr\}\, , \\
      \setC
        &\defeq 
           \setC(\vgam)
         \defeq
           \bigl\{ 
             j \in [n]
           \bigm|
             \exists i \in [n], \text{ s.t. } 0 < \gamma(i,j) < 1
           \bigr\}\, ,
    \end{align*}
    \ie, the set $\setR$ consists of the row indices of rows where $\vgam$
    contains at least one fractional entry and, similarly, the set $\setC$
    consists of the column indices of columns where $\vgam$ contains at least
    one fractional entry. Based on $\setR$ and $\setC$, we define
    \begin{align*}
      \vgamRCp 
        &\defeq 
           \bigl( \gamma(i,j) \bigr)_{\! (i,j) \in \setR \times \setC}\, .
    \end{align*}
    On the other hand, the set $[n] \setminus \setR$ consists of the row
    indices of rows where $\vgam$ contains exactly one $1$ and $0$s
    otherwise. Similarly, the set $[n] \setminus \setC$ consists of the
    column indices of columns where $\vgam$ contains exactly one $1$ and $0$s
    otherwise.
  
  \item We define $ r(\vgam) \defeq |\setR|$. One can verify the
    following statements:
    \begin{itemize}

    \item $|\setR| = |\setC|$.

    \item Either $r(\vgam) = 0$ or $r(\vgam) \geq 2$, \ie,
      $r(\vgam) = 1$ is impossible.

    \item If $r(\vgam) = 0$, then
      $\bigl| \setS_{[n]}(\vgam) \bigr| = 1$.

    \item If $r(\vgam) > 0$, then $\vgamRCp \in \Gamma_{r(\vgam)}$, \ie, the
      matrix $\vgamRCp$ is a doubly stochastic matrix of size
      $r(\vgam) \times r(\vgam)$.

  \end{itemize}

  \item Based on the matrix $\vgamRCp$, we define the matrix
    \begin{align*}
      \hvgamRCp
        &\defeq
           \bigl(
             \hgamRCp(i,j)
           \bigr)_{\! (i,j) \in \setR \times \setC}\, ,
    \end{align*}
    where
    \begin{align*}
      \hgamRCp(i,j)
        &\defeq 
          \frac{\gamRCp(i,j) \cdot \bigl( 1- \gamRCp(i,j) \bigr)}
               {\Biggl(
                  \prod\limits_{(i',j') \in \setR \times \setC} 
                     \bigl( 1 - \gamRCp(i',j') \bigr)
                \Biggr)^{\!\!\! 1 / r(\vgam)}
               }\, .
    \end{align*}
    Note that the permanent of $ \hvgamRCp $ is given by
    \begin{align}
      &\perm(\hvgamRCp) \nonumber \\
        &= \sum\limits_{\sigma \in \setS_{\setR\to\setC} }
             \prod_{i \in \setR}
               \frac{\gamRCp\bigl(i,\sigma(i) \bigr) 
                     \cdot
                     \Bigl( 1- \gamRCp\bigl(i,\sigma(i) \bigr) \Bigr)
                    }
                 {\Biggl(
                    \prod\limits_{(i',j') \in \setR \times \setC} 
                       \bigl( 1 - \gamRCp(i',j') \bigr)
                  \Biggr)^{\!\!\! 1 / r(\vgam)}
                 } 
        \nonumber \\
        &= \frac{ 
                  \sum\limits_{\sigma \in \setS_{\setR\to\setC} }
                    \prod\limits_{i \in \setR}
                      \gamRCp\bigl(i,\sigma(i) \bigr)
                      \cdot
                      \Bigl( 1- \gamRCp\bigl(i,\sigma(i) \bigr) \Bigr)
                }
                { 
                  \prod\limits_{i \in \setR,j \in \setC}
                    \bigl( 1- \gamRCp(i,j) \bigr) 
                }\, .
        \label{eqn: permament of hvgamRCp}
    \end{align}

  \end{enumerate}
  \edefinition
\end{definition}


\begin{example}
  \label{example:vgam:1:1}

  Consider the matrix
  \begin{align*}
    \vgam 
       \defeq
         \frac{1}{3}
         \cdot
         \begin{pmatrix}
           3 & 0 & 0 & 0 \\
           0 & 0 & 3 & 0 \\
           0 & 1 & 0 & 2 \\
           0 & 2 & 0 & 1
         \end{pmatrix}
           \in \Gamma_4\, .
  \end{align*}
  We obtain $\setR = \{ 3, 4 \}$, $\setC = \{ 2, 4 \}$, $r(\vgam) =
  2$, 
  \begin{align*}
    \vgamRCp
      &= \frac{1}{3}
         \cdot
         \begin{pmatrix}
           1 & 2 \\
           2 & 1
         \end{pmatrix}\, ,
    \qquad
    \hvgamRCp
       = \begin{pmatrix}
           1 & 1 \\
           1 & 1
         \end{pmatrix}\, ,
  \end{align*}
  and $\perm(\hvgamRCp) = 2$.
  \eexample
\end{example}


\begin{definition}    
  \label{sec:1:def:10:part:2:add}

  Consider $M \in \sZ_{\geq 2}$, $ \vgam \in \Gamma_{M,n} $, and
  $ \sigma_1 \in \setS_{[n]}(\vgam) $. (The discussion after
    Definition~\ref{sec:1:def:10} ensures that such a $ \sigma_1 $ exists.)
  We define
  \begin{align*}
    \vgam_{\sigma_1}
      &\defeq
         \frac{1}{M-1} 
         \cdot
         \bigl( M \cdot \vgam - \mP_{\sigma_1} \bigr)\, .
  \end{align*}
  Note that $ \vgam_{\sigma_1} \in \Gamma_{M-1,n} $. 
  The matrix
  $\vgam_{\sigma_1}$ can be thought of, modulo normalizations, as being
  obtained by ``peeling off'' the permutation matrix $\mP_{\sigma_1}$ from
  $\vgam$. 
  \edefinition
\end{definition}


\begin{example}
  \label{example:vgam:1:2}

  Consider again the matrix $\vgam \in \Gamma_4$ that was introduced in
  Example~\ref{example:vgam:1:1}. Let $M = 3$. Note that
  $ \vgam \in \Gamma_{3,4} $. If $\sigma_1 \in \setS_{[n]}(\vgam)$ is chosen
  to be
  \begin{align*}
    \sigma_1(1) = 1\, , \ \ 
    \sigma_1(2) = 3\, , \ \ 
    \sigma_1(3) = 2\, , \ \ 
    \sigma_1(4) = 4\, ,
  \end{align*}
  then
  \begin{align*}
    \vgam_{\sigma_1}
       \defeq
         \frac{1}{2}
         \cdot
         \begin{pmatrix}
           2 & 0 & 0 & 0 \\
           0 & 0 & 2 & 0 \\
           0 & 0 & 0 & 2 \\
           0 & 2 & 0 & 0
         \end{pmatrix}
           \in \Gamma_{2,4}\, .
  \end{align*}
  On the other hand, if $\sigma_1 \in \setS_{[n]}(\vgam)$ is
  chosen to be
  \begin{align*}
    \sigma_1(1) = 1\, , \ \ 
    \sigma_1(2) = 3\, , \ \ 
    \sigma_1(3) = 4\, , \ \ 
    \sigma_1(4) = 2\, ,
  \end{align*}
  then
  \begin{align*}
    \vgam_{\sigma_1}
       \defeq
         \frac{1}{2}
         \cdot
         \begin{pmatrix}
           2 & 0 & 0 & 0 \\
           0 & 0 & 2 & 0 \\
           0 & 1 & 0 & 1 \\
           0 & 1 & 0 & 1
         \end{pmatrix}
           \in \Gamma_{2,4}\, .
  \end{align*}
  \eexample
\end{example}


The following three lemmas are the key technical results that will allow us to
prove Theorems~\ref{thm: inequalities for the coefficients}
and~\ref{th:main:permanent:inequalities:1} in the next section.


\begin{lemma}
  \label{sec:1:lem:29:copy:2}
  
  Let $ M \in \sZ_{\geq 2} $ and $ \vgam \in \Gamma_{M,n} $. It holds that
  \begin{align*}
     \CM{n}( \vgam )
        &= \sum_{\sigma_1 \in \setS_{[n]}( \vgam )}
             \Cgen{M-1}{n}\bigl( 
                \vgam_{\sigma_1}
        \bigr)\, .
    \end{align*}
\end{lemma}


\begin{proof}
  If $r(\vgam) = 0$, then $\bigl| \setS_{[n]}( \vgam ) \bigr| = 1$ and
  the result in the lemma statement is straightforward. If
  $r(\vgam) \geq 2$, then the result in the lemma statement follows
  from the fact that all permutations in $ \setS_{[n]}(\vgam) $ are distinct
  and can be used as $ \sigma_1 $ in $ \vsigma_{[M]} $ for decomposing the
  matrix $ \vgam $.
\end{proof}


\begin{lemma}
  \label{sec:1:lem:15:copy:2}

  Let $M \in \sZ_{\geq 2}$ and $\vgam \in \Gamma_{M,n}$. Based on $\vgam$, we
  define the set $\setR$, the set $\setC$, and the matrix $\hvgamRCp$
  according to Definition~\ref{sec:1:def:15}. Moreover, for any
  $\sigma_1 \in \setS_{[n]}( \vgam )$, we define $\vgam_{\sigma_1}$ according
  to Definition~\ref{sec:1:def:10:part:2:add}. It holds
  that\footnote{Consider
       the case where both $\setR$ and $\setC$ are non-empty sets. Because each entry
      in $ \hvgamRCp $ is positive-valued, we have $\perm(\hvgamRCp) > 0$, and
      so the fraction $1 / \perm(\hvgamRCp)$ is well defined.}
  \begin{align*}
    \CBM{n}(\vgam)
      &= \frac{1}{\perm(\hvgamRCp)}
         \cdot
         \sum_{\sigma_1 \in \setS_{[n]}( \vgam )}
           \CBgen{M-1}{n}
             \bigl( 
               \vgam_{\sigma_1}
             \bigr)\, ,
  \end{align*}
  where we define $\perm(\hvgamRCp) \defeq 1$ if $\setR$ and $\setC$ are
  empty sets.
\end{lemma}


\begin{proof}
  See Appendix~\ref{app:sec:1:lem:15:copy:2}.
\end{proof}


\begin{lemma}
  \label{sec:1:lem:30}
  
  Let $M \in \sZ_{\geq 2}$ and $\vgam \in \Gamma_{M,n}$. For any
  $\sigma_1 \in \setS_{[n]}( \vgam )$, we define $\vgam_{\sigma_1}$ according
  to Definition~\ref{sec:1:def:10:part:2:add}. It holds that\footnote{Note
    that $\vgam \in \Gamma_{M,n}$ and so $\perm(\vgam) \geq n! / n^n$ by van
    der Waerden's inequality. With this, the fraction $1 / \perm(\vgam)$ is
    well defined.}
  \begin{align*}
    \CscSgen{M}{n}(\vgam)
      &\!=\! \frac{1}{
           \bigr( \chi(M) \bigl)^{\! n}
           \, \cdot \,
           \perm(\vgam)
         }
        \cdot \!\!\!
        \sum_{\sigma_1 \in \setS_{[n]}( \vgam )} \!\!\!
          \CscSgen{M-1}{n}( \vgam_{\sigma_1} )\, , 
  \end{align*}
  where
  \begin{align*}
    \chi(M)
      &\defeq
         \biggl( \frac{M}{M-1} \biggr)^{\!\! M-1}\, ,
  \end{align*}
  which starts at $\bigl. \chi(M) \bigr|_{M = 2} \! = \! 2$
      and increases to $\bigl. \chi(M) \bigr|_{M \to \infty} \! = \! \e$.
\end{lemma}


\begin{proof}
  See Appendix~\ref{apx:23}.
\end{proof}


All objects and quantities in this paper assume $M \in \sZ_{\geq
  1}$. Interestingly, it is possible to define these objects and quantities
also for $M = 0$, at least formally. However, we will not do this except in
the following example, as it leads to nicer figures (see
Figs.~\ref{fig:pascal:triangle:generalization:1}%
--\ref{fig:pascal:triangle:generalization:3}).


\begin{example}
  \label{example:pascal:triangle:generalization:1}

  The purpose of this example is to visualize the recursions in
  Lemmas~\ref{sec:1:lem:29:copy:2}--\ref{sec:1:lem:30} for the special case
  $n = 2$, \ie, for matrices $\vgam$ of size $2 \times 2$. For,
  $k_1, k_2 \in \sZ_{\geq 0}$, define\footnote{The case $k_1 = k_2 = 0$ is
    special: the matrix $\vgam^{(0,0)}$ has to be considered as some formal
    matrix and the set $\Gamma_{0,2}$ is defined to be the set containing only
    the matrix $\vgam^{(0,0)}$.}
  \begin{align*}
    \vgam^{(k_1,k_2)}
      &\defeq
         \frac{1}{k_1 + k_2} 
         \cdot
         \begin{pmatrix}
           k_1 & k_2 \\
           k_2 & k_1
         \end{pmatrix}
           \in \Gamma_{k_1+k_2,2}\, .
             \quad
  \end{align*}
  Define
  \begin{align*}
    \Cgen{0}{n}\bigl( \vgam^{(0,0)} \bigr)
      &\defeq
         1\, , \ 
    \CBgen{0}{n}\bigl( \vgam^{(0,0)} \bigr)
       \defeq
         1\, , \ 
    \CscSgen{0}{n}\bigl( \vgam^{(0,0)} \bigr)
       \defeq
         1\, .
  \end{align*}
  One can the verify that with these definitions, the recursions in
  Lemmas~\ref{sec:1:lem:29:copy:2}--\ref{sec:1:lem:30} hold also for $M = 1$.


  With these definitions in place, we can draw diagrams that can be seen as
  generalizations of Pascal's triangle,\footnote{These generalizations of
    Pascal's triangle are drawn such that $M$ increases from the bottom to the
    top. This is different from the typical drawings of Pascal's triangle
    (where $M$ increases from the top to the bottom), but matches better the
    convention used in Fig.~\ref{fig:combinatorial:characterization:1} and
    similar figures in~\cite{Vontobel2013} (\ie, \cite[Figs.~7
    and~10]{Vontobel2013}), along with better matching the fact that
    $\permbM{M}(\mtheta)$ is based on \emph{liftings} of~$\mtheta$.} see
  Figs.~\ref{fig:pascal:triangle:generalization:1}%
  --\ref{fig:pascal:triangle:generalization:3}. The meaning of these figures
  is as follows. For example, consider the quantity
  $\Cgen{3}{n}\bigl(\vgam^{(1,2)}\bigr)$. According to
  Lemma~\ref{sec:1:lem:29:copy:2}, this quantity can be obtained via
  \begin{align}
    \Cgen{3}{n}\bigl(\vgam^{(1,2)}\bigr)
      &= 1 
           \cdot
           \Cgen{2}{n}\bigl(\vgam^{(1,1)}\bigr)
         +
         1
           \cdot
           \Cgen{2}{n}\bigl(\vgam^{(0,2)}\bigr) \nonumber \\
      &= 1
           \cdot 2
         +
         1
           \cdot 1 \nonumber \\
      &= 3 \, .
      \label{eq:CM:recursion:special:case:1}
  \end{align}
  The recursion in~\eqref{eq:CM:recursion:special:case:1} is visualized in
  Fig.~\ref{fig:pascal:triangle:generalization:1} as follows. First, the
  vertex/box above the label ``$\Cgen{3}{n}\bigl(\vgam^{(1,2)}\bigr)$''
  contains the value of $\Cgen{3}{n}\bigl(\vgam^{(1,2)}\bigr)$. Second, the
  arrows indicate that $\Cgen{3}{n}\bigl(\vgam^{(1,2)}\bigr)$ is obtained by
  adding $1$ times $\Cgen{2}{n}\bigl(\vgam^{(1,1)}\bigr)$ and $1$ times
  $\Cgen{3}{n}\bigl(\vgam^{(0,2)}\bigr)$.

  Moreover, notice that the value of $\Cgen{3}{n}\bigl(\vgam^{(1,2)}\bigr)$
  equals the number of directed paths from the vertex labeled
  ``$\Cgen{0}{n}\bigl(\vgam^{(0,0)}\bigr)$'' to the vertex labeled
  ``$\Cgen{3}{n}\bigl(\vgam^{(1,2)}\bigr)$''. In general, for
  $k_1, k_2 \in \sZ_{\geq 0}$, the value of
  $\Cgen{k_1+k_2}{n}\bigl(\vgam^{(k_1,k_2)}\bigr)$ equals the number of
  directed paths from the vertex labeled
  ``$\Cgen{0}{n}\bigl(\vgam^{(0,0)}\bigr)$'' to the vertex labeled
  ``$\Cgen{k_1+k_2}{n}\bigl(\vgam^{(k_1,k_2)}\bigr)$''. In this context,
  recall from
  Proposition~\ref{prop:coefficient:asymptotitic:characterization:1} that
  \begin{align*}
    \Cgen{k_1+k_2}{n}\bigl(\vgam^{(k_1,k_2)}\bigr)
      &\approx
         \exp
           \Bigl(
             (k_1 + k_2) \cdot \HG'\bigl( \vgam^{(k_1,k_2)} \bigr)
           \Bigr)\, ,
  \end{align*}
  where the approximation is up to $o(k_1+k_2)$ in the exponent.


  Similar statements can be made for
  $\CBgen{k_1+k_2}{n}\bigl(\vgam^{(k_1,k_2)}\bigr)$ and
  $\CscSgen{k_1+k_2}{n}\bigl(\vgam^{(k_1,k_2)}\bigr)$ in
  Fig.~\ref{fig:pascal:triangle:generalization:2} and
  Fig.~\ref{fig:pascal:triangle:generalization:3}, respectively. Namely, for
  $k_1, k_2 \in \sZ_{\geq 0}$, the value of
  $\CBgen{k_1+k_2}{n}\bigl(\vgam^{(k_1,k_2)}\bigr)$ equals the sum of weighted
  directed paths from the vertex labeled
  ``$\CBgen{0}{n}\bigl(\vgam^{(0,0)}\bigr)$'' to the vertex labeled
  ``$\CBgen{k_1+k_2}{n}\bigl(\vgam^{(k_1,k_2)}\bigr)$'', and
  $\CscSgen{k_1+k_2}{n}\bigl(\vgam^{(k_1,k_2)}\bigr)$ equals the sum of
  weighted directed paths from the vertex labeled
  ``$\CscSgen{0}{n}\bigl(\vgam^{(0,0)}\bigr)$'' to the vertex labeled
  ``$\CscSgen{k_1+k_2}{n}\bigl(\vgam^{(k_1,k_2)}\bigr)$''. (Here, the weight
  of a directed path is given by the product of the weights of the directed
  edges in the path.) In this context, recall from
  Proposition~\ref{prop:coefficient:asymptotitic:characterization:1} that
  \begin{align*}
  \CBgen{k_1+k_2}{n}\bigl(\vgam^{(k_1,k_2)}\bigr)
    &\approx
       \exp
         \Bigl(
           (k_1+k_2) \cdot \HBthe\bigl( \vgam^{(k_1,k_2)} \bigr)
         \Bigr)\, , \\
  \CscSgen{k_1+k_2}{n}\bigl(\vgam^{(k_1,k_2)}\bigr)
    &\approx
       \exp
         \Bigl(
           (k_1+k_2) \cdot \HscSthe\bigl( \vgam^{(k_1,k_2)} \bigr)
         \Bigr)\, ,
  \end{align*}
  where the approximation is up to $o(k_1+k_2)$ in the exponent.
  \eexample
\end{example}


\begin{figure}
  \begin{center}
    \begin{tikzpicture}
  {\scriptsize

  \node at (0,0) 
    {\parbox{2cm}{\centering
                    \fbox{$1$} \\[0.05cm]
                    $\Cgen{0}{n}\bigl(\vgam^{(0,0)}\bigr)$
                    
                 }
    };

  \draw[->] (-0.25,+0.50) -- node[below left]  {$1$} (-0.75,+1.50);
  \draw[->] (+0.25,+0.50) -- node[below right] {$1$} (+0.75,+1.50);

  \node at (-1,+2) 
    {\parbox{2cm}{\centering
                    \fbox{$1$} \\[0.05cm]
                    $\Cgen{1}{n}\bigl(\vgam^{(1,0)}\bigr)$
                 }
    };

  \node at (+1,+2) 
    {\parbox{2cm}{\centering
                    \fbox{$1$} \\[0.05cm]
                    $\Cgen{1}{n}\bigl(\vgam^{(0,1)}\bigr)$
                 }
    };

  \draw[->] (-1.25,+2.50) -- node[below left]  {$1$} (-1.75,+3.50);
  \draw[->] (-0.75,+2.50) -- node[below right] {$1$} (-0.25,+3.50);

  \draw[->] (+0.75,+2.50) -- node[below left]  {$1$} (+0.25,+3.50);
  \draw[->] (+1.25,+2.50) -- node[below right] {$1$} (+1.75,+3.50);

  \node at (-2,+4) 
    {\parbox{2cm}{\centering
                    \fbox{$1$} \\[0.05cm]
                    $\Cgen{2}{n}\bigl(\vgam^{(2,0)}\bigr)$
                 }
    };

  \node at ( 0,+4) 
    {\parbox{2cm}{\centering
                    \fbox{$2$} \\[0.05cm]
                    $\Cgen{2}{n}\bigl(\vgam^{(1,1)}\bigr)$
                 }
    };

  \node at (+2,+4) 
    {\parbox{2cm}{\centering
                    \fbox{$1$} \\[0.05cm]
                    $\Cgen{2}{n}\bigl(\vgam^{(0,2)}\bigr)$
                 }
    };

  \draw[->] (-2.25,+4.50) -- node[below left]  {$1$} (-2.75,+5.50);
  \draw[->] (-1.75,+4.50) -- node[below right] {$1$} (-1.25,+5.50);

  \draw[->] (-0.25,+4.50) -- node[below left]  {$1$} (-0.75,+5.50);
  \draw[->] (+0.25,+4.50) -- node[below right] {$1$} (+0.75,+5.50);

  \draw[->] (+1.75,+4.50) -- node[below left]  {$1$} (+1.25,+5.50);
  \draw[->] (+2.25,+4.50) -- node[below right] {$1$} (+2.75,+5.50);

  \node at (-3,+6) 
    {\parbox{2cm}{\centering
                    \fbox{$1$} \\[0.05cm]
                    $\Cgen{3}{n}\bigl(\vgam^{(3,0)}\bigr)$
                 }
    };

  \node at (-1,+6) 
    {\parbox{2cm}{\centering
                    \fbox{$3$} \\[0.05cm]
                    $\Cgen{3}{n}\bigl(\vgam^{(2,1)}\bigr)$
                 }
    };

  \node at (+1,+6) 
    {\parbox{2cm}{\centering
                    \fbox{$3$} \\[0.05cm]
                    $\Cgen{3}{n}\bigl(\vgam^{(1,2)}\bigr)$
                 }
    };

  \node at (+3,+6) 
    {\parbox{2cm}{\centering
                    \fbox{$1$} \\[0.05cm]
                    $\Cgen{3}{n}\bigl(\vgam^{(0,3)}\bigr)$
                 }
    };

  }
\end{tikzpicture}
    \medskip
    \caption{Generalization of Pascal's triangle visualizing the recursion in
      Lemma~\ref{sec:1:lem:29:copy:2}.}
    \label{fig:pascal:triangle:generalization:1}

    \vspace{1cm}

    \begin{tikzpicture}
  {\scriptsize

  \node at (0,0) 
    {\parbox{2cm}{\centering
                    \fbox{$1$} \\[0.05cm]
                    $\CBgen{0}{n}\bigl(\vgam^{(0,0)}\bigr)$
                 }
    };

  \draw[->] (-0.25,+0.50) -- node[below left]  {$1$} (-0.75,+1.50);
  \draw[->] (+0.25,+0.50) -- node[below right] {$1$} (+0.75,+1.50);

  \node at (-1,+2) 
    {\parbox{2cm}{\centering
                    \fbox{$1$} \\[0.05cm]
                    $\CBgen{1}{n}\bigl(\vgam^{(1,0)}\bigr)$
                 }
    };

  \node at (+1,+2) 
    {\parbox{2cm}{\centering
                    \fbox{$1$} \\[0.05cm]
                    $\CBgen{1}{n}\bigl(\vgam^{(0,1)}\bigr)$
                 }
    };

  \draw[->] (-1.25,+2.50) -- node[below left]  {$1$} (-1.75,+3.50);
  \draw[->] (-0.75,+2.50) -- node[below right] {$\frac{1}{2}$} (-0.25,+3.50);

  \draw[->] (+0.75,+2.50) -- node[below left]  {$\frac{1}{2}$} (+0.25,+3.50);
  \draw[->] (+1.25,+2.50) -- node[below right] {$1$} (+1.75,+3.50);

  \node at (-2,+4) 
    {\parbox{2cm}{\centering
                    \fbox{$1$} \\[0.05cm]
                    $\CBgen{2}{n}\bigl(\vgam^{(2,0)}\bigr)$
                 }
    };

  \node at ( 0,+4) 
    {\parbox{2cm}{\centering
                    \fbox{$1$} \\[0.05cm]
                    $\CBgen{2}{n}\bigl(\vgam^{(1,1)}\bigr)$
                 }
    };

  \node at (+2,+4) 
    {\parbox{2cm}{\centering
                    \fbox{$1$} \\[0.05cm]
                    $\CBgen{2}{n}\bigl(\vgam^{(0,2)}\bigr)$
                 }
    };

  \draw[->] (-2.25,+4.50) -- node[below left]  {$1$} (-2.75,+5.50);
  \draw[->] (-1.75,+4.50) -- node[below right] {$\frac{1}{2}$} (-1.25,+5.50);

  \draw[->] (-0.25,+4.50) -- node[below left]  {$\frac{1}{2}$} (-0.75,+5.50);
  \draw[->] (+0.25,+4.50) -- node[below right] {$\frac{1}{2}$} (+0.75,+5.50);

  \draw[->] (+1.75,+4.50) -- node[below left]  {$\frac{1}{2}$} (+1.25,+5.50);
  \draw[->] (+2.25,+4.50) -- node[below right] {$1$} (+2.75,+5.50);

  \node at (-3,+6) 
    {\parbox{2cm}{\centering
                    \fbox{$1$} \\[0.05cm]
                    $\CBgen{3}{n}\bigl(\vgam^{(3,0)}\bigr)$
                 }
    };

  \node at (-1,+6) 
    {\parbox{2cm}{\centering
                    \fbox{$1$} \\[0.05cm]
                    $\CBgen{3}{n}\bigl(\vgam^{(2,1)}\bigr)$
                 }
    };

  \node at (+1,+6) 
    {\parbox{2cm}{\centering
                    \fbox{$1$} \\[0.05cm]
                    $\CBgen{3}{n}\bigl(\vgam^{(1,2)}\bigr)$
                 }
    };

  \node at (+3,+6) 
    {\parbox{2cm}{\centering
                    \fbox{$1$} \\[0.05cm]
                    $\CBgen{3}{n}\bigl(\vgam^{(0,3)}\bigr)$
                 }
    };

  }
\end{tikzpicture}
    \medskip
    \caption{Generalization of Pascal's triangle visualizing the recursion in
      Lemma~\ref{sec:1:lem:15:copy:2}.}
    \medskip
    \label{fig:pascal:triangle:generalization:2}
  
    \vspace{1cm}

    \begin{tikzpicture}
  {\scriptsize

  \node at (0,0) 
    {\parbox{2cm}{\centering
                    \fbox{$1$} \\[0.05cm]
                    $\CscSgen{0}{n}\bigl(\vgam^{(0,0)}\bigr)$
                 }
    };

  \draw[->] (-0.25,+0.50) -- node[below left]  {$1$} (-0.75,+1.50);
  \draw[->] (+0.25,+0.50) -- node[below right] {$1$} (+0.75,+1.50);

  \node at (-1,+2) 
    {\parbox{2cm}{\centering
                    \fbox{$1$} \\[0.05cm]
                    $\CscSgen{1}{n}\bigl(\vgam^{(1,0)}\bigr)$
                 }
    };

  \node at (+1,+2) 
    {\parbox{2cm}{\centering
                    \fbox{$1$} \\[0.05cm]
                    $\CscSgen{1}{n}\bigl(\vgam^{(0,1)}\bigr)$
                 }
    };

  \draw[->] (-1.25,+2.50) -- node[below left]  {$\frac{1}{4}$} (-1.75,+3.50);
  \draw[->] (-0.75,+2.50) -- node[below right] {$\frac{1}{2}$} (-0.25,+3.50);

  \draw[->] (+0.75,+2.50) -- node[below left]  {$\frac{1}{2}$} (+0.25,+3.50);
  \draw[->] (+1.25,+2.50) -- node[below right] {$\frac{1}{4}$} (+1.75,+3.50);

  \node at (-2,+4) 
    {\parbox{2cm}{\centering
                    \fbox{$\frac{1}{4}$} \\[0.05cm]
                    $\CscSgen{2}{n}\bigl(\vgam^{(2,0)}\bigr)$
                 }
    };

  \node at ( 0,+4) 
    {\parbox{2cm}{\centering
                    \fbox{$1$} \\[0.05cm]
                    $\CscSgen{2}{n}\bigl(\vgam^{(1,1)}\bigr)$
                 }
    };

  \node at (+2,+4) 
    {\parbox{2cm}{\centering
                    \fbox{$\frac{1}{4}$} \\[0.05cm]
                    $\CscSgen{2}{n}\bigl(\vgam^{(0,2)}\bigr)$
                 }
    };

  \draw[->] (-2.25,+4.50) -- node[below left]  {$\frac{16}{81}$} (-2.75,+5.50);
  \draw[->] (-1.75,+4.50) -- node[below right] {$\frac{16}{45}$} (-1.25,+5.50);

  \draw[->] (-0.25,+4.50) -- node[below left]  {$\frac{16}{45}$} (-0.75,+5.50);
  \draw[->] (+0.25,+4.50) -- node[below right] {$\frac{16}{45}$} (+0.75,+5.50);

  \draw[->] (+1.75,+4.50) -- node[below left]  {$\frac{16}{45}$} (+1.25,+5.50);
  \draw[->] (+2.25,+4.50) -- node[below right] {$\frac{16}{81}$} (+2.75,+5.50);

  \node at (-3,+6) 
    {\parbox{2cm}{\centering
                    \fbox{$\frac{4}{81}$} \\[0.05cm]
                    $\CscSgen{3}{n}\bigl(\vgam^{(3,0)}\bigr)$
                 }
    };

  \node at (-1,+6) 
    {\parbox{2cm}{\centering
                    \fbox{$\frac{4}{9}$} \\[0.05cm]
                    $\CscSgen{3}{n}\bigl(\vgam^{(2,1)}\bigr)$
                 }
    };

  \node at (+1,+6) 
    {\parbox{2cm}{\centering
                    \fbox{$\frac{4}{9}$} \\[0.05cm]
                    $\CscSgen{3}{n}\bigl(\vgam^{(1,2)}\bigr)$
                 }
    };

  \node at (+3,+6) 
    {\parbox{2cm}{\centering
                    \fbox{$\frac{4}{81}$} \\[0.05cm]
                    $\CscSgen{3}{n}\bigl(\vgam^{(0,3)}\bigr)$
                 }
    };

  }
\end{tikzpicture}
    \medskip
    \caption{Generalization of Pascal's triangle visualizing the recursion in
      Lemma~\ref{sec:1:lem:30}.}
    \medskip
    \label{fig:pascal:triangle:generalization:3}
  \end{center}
\end{figure}


In the context of Example~\ref{example:pascal:triangle:generalization:1}, note
that for $n = 2$ the matrix
$\matr{1}_{2 \times 2} \defeq \bigl( \begin{smallmatrix} 1 & 1 \\ 1 &
  1 \end{smallmatrix} \bigr)$ achieves the upper bound $2^{n/2} = 2$
in~\eqref{eq:ratio:permanent:bethe:permanent:1}, and with that the matrix
$\matr{1}_{2 \times 2}$ is the ``worst-case matrix'' for $n = 2$ in the sense
that the ratio $\frac{\perm(\mtheta)}{\permb(\mtheta)} = 2$ for
$\mtheta = \matr{1}_{2 \times 2}$ is the furthest away from the ratio
$\frac{\perm(\mtheta)}{\permb(\mtheta)} = 1$ for diagonal
matrices~$\mtheta$. However, this behavior of the all-one matrix for $n = 2$
is atypical. Namely, letting $\matr{1}_{n \times n}$ be the all-one matrix for
$n \in \sZ_{\geq 1}$, one obtains, according to
\cite[Lemma~48]{Vontobel2013a},
\begin{align*}
  \left.
    \frac{\perm(\mtheta)}{\permb(\mtheta)}
  \right|_{\mtheta = \matr{1}_{n \times n}}
    &= \sqrt{\frac{2 \pi n}{\e}}
       \cdot
       \bigl( 1 + o_n(1) \bigr)\, ,
\end{align*}
\ie, the ratio
$\frac{\perm(\matr{1}_{n \times n})}{\permb(\matr{1}_{n \times n})}$ is much
closer to $1$ than to $2^{n/2}$ for large values of $n$. (In the above
expression, the term $o_n(\ldots)$ is based on the usual little-o notation for
a function in $n$.)


In the case of the scaled Sinkhorn permanent, the ``worst-case matrix'' is
given by the all-one matrix not just for $n = 2$, but for all
$n \in \sZ_{\geq 2}$, \ie, the ratio
$\frac{\perm(\mtheta)}{\permscs(\mtheta)} = \e^n \cdot \frac{n!}{n^n}$ for
$\mtheta = \matr{1}_{n \times n}$ is the furthest away from the value of the
ratio $\frac{\perm(\mtheta)}{\permscs(\mtheta)} = \e^n$ for diagonal matrices
$\mtheta$ for all $n \in \sZ_{\geq 2}$.


\ifx\withclearpagecommands\x
\clearpage
\fi

\section{Bounding the Permanent of \texorpdfstring{$\mtheta$}{}}
\label{sec:5}


In this section, we first prove Theorem~\ref{thm: inequalities for the
  coefficients} with the help of Lemmas~\ref{sec:1:lem:15:copy:2}
and~\ref{sec:1:lem:30}, where, in order for Lemmas~\ref{sec:1:lem:15:copy:2}
and~\ref{sec:1:lem:30} to be useful, we need to find lower and upper bounds on
$\perm(\hvgamRCp)$ and on $\perm(\vgam)$, respectively. Afterwards, we use
Theorem~\ref{thm: inequalities for the coefficients} to prove
Theorem~\ref{th:main:permanent:inequalities:1}.


\subsection{\texorpdfstring{Bounding the ratio $\perm(\mtheta) /  
                                                                                     \permbM{M}(\mtheta)$}{}}


In this subsection, we first prove the inequalities in~\eqref{sec:1:eqn:58} in
Theorem~\ref{thm: inequalities for the coefficients} that give bounds on the
ratio $\CM{n}(\vgam) / \CBM{n}(\vgam)$. Then, based on these bounds, we prove
the inequalities in~\eqref{sec:1:eqn:147} in
Theorem~\ref{th:main:permanent:inequalities:1} that give bounds on the ratio
$\perm(\mtheta) / \permbM{M}(\mtheta)$.


\begin{lemma}
  \label{sec:1:lem:17}

  Consider the setup in Lemma~\ref{sec:1:lem:15:copy:2}. The matrix
  $\hvgamRCp$ satisfies
  \begin{align}
    1
      &\leq  
         \perm(\hvgamRCp)
       \leq
         2^{n/2}\, .
           \label{sec:1:eqn:145}
  \end{align}
\end{lemma}


\begin{proof}
  It is helpful to first prove that $\permb(\hvgamRCp) = 1$. Indeed,
  \begin{align*}
    \permb(\hvgamRCp)
      &\overset{(a)}{=} 
         \exp
           \biggl(
             - \min_{\vgam \in \Gamma_{|\setR|}(\hvgamRCp)}
                 F_{\mathrm{B},\hvgamRCp}( \vgam )
           \biggr) \\
      &\overset{(b)}{=}
         \exp\bigl( - F_{\mathrm{B},\hvgamRCp}( \vgamRCp ) \bigr) \\
      &\overset{(c)}{=} \exp(0) \\
      &= 1\, ,
  \end{align*}
  where step~$(a)$ follows from the definition of the Bethe permanent in
  Definition~\ref{sec:1:def:16}, where step~$(b)$ follows from the convexity
  of the minimization problem
  $\min_{\vgam \in \Gamma_{|\setR|}(\hvgamRCp)} F_{\mathrm{B},\hvgamRCp}(
  \vgam )$ (which was proven in~\cite{Vontobel2013a}) and from studying the
  Karush–Kuhn–Tucker (KKT) conditions for this minimization problem at $\vgam = \vgamRCp$ (the
  details are omitted), and where step~$(c)$ follows from simplifying the
  expression for $F_{\mathrm{B},\hvgamRCp}( \vgamRCp )$.

  With this, the first inequality in~\eqref{sec:1:eqn:145} follows from
  \begin{align*}
    1
      &= \permb(\hvgamRCp)
       \overset{(a)}{\leq}
         \perm(\hvgamRCp)\, ,
  \end{align*}
  where step~$(a)$ follows from the first inequality
  in~\eqref{eq:ratio:permanent:bethe:permanent:1}, and the second inequality
  in~\eqref{sec:1:eqn:145} follows from
  \begin{align*}
    \perm(\hvgamRCp) 
      &\overset{(b)}{\leq}
         \underbrace
         {
           2^{|\setR|/2}
         }_{\leq \ 2^{n/2}} 
         \cdot
         \underbrace{
           \permb(\hvgamRCp)
         }_{= \ 1}
       \leq 
         2^{n/2}\, ,
  \end{align*}
  where step~$(b)$ follows from the second inequality
  in~\eqref{eq:ratio:permanent:bethe:permanent:1}.

  An alternative proof of the first inequality in~\eqref{sec:1:eqn:145} goes
  as follows. Namely, this inequality can be obtained by 
  combining~\eqref{eqn: permament of hvgamRCp} and
  \begin{align*}
    \hspace{4cm}&\hspace{-4cm}
    \sum_{\sigma \in \setS_{\setR \to \setC}}
      \prod_{i \in \setR} 
        \gamRCp\bigl(i,\sigma(i)\bigr)
        \cdot 
        \Bigl( 1 \! - \! \gamRCp\bigl(i,\sigma(i) \bigr) \Bigr) \\
      &\overset{(a)}{\geq}
         \!\! \prod_{(i,j) \in \setR \times \setC} \!\!
           \bigl( 1 \! - \! \gamRCp(i,j) \bigr)\, ,
  \end{align*}
  where step~$(a)$ follows from~\cite[Corollary~1c of
  Theorem~1]{Schrijver1998}. Note that this alternative proof for the first
  inequality in~\eqref{sec:1:eqn:145} is not really that different from the
  above proof because the first inequality
  in~\eqref{eq:ratio:permanent:bethe:permanent:1} is also based
  on~\cite[Corollary~1c of Theorem~1]{Schrijver1998}
\end{proof}


\noindent
\textbf{Proof of the inequalities in~\eqref{sec:1:eqn:58} in Theorem~\ref{thm:
    inequalities for the coefficients}}.
\begin{itemize}

\item We prove the first inequality in~\eqref{sec:1:eqn:58} by induction on
  $M$. The base case is $M = 1$. From Definition~\ref{sec:1:lem:43}, it
  follows that $\CM{n}(\vgam) = \CBM{n}(\vgam) = 1$ for
  $\vgam \in \Gamma_{M,n}$ and so
  \begin{align*}
    \frac{\CM{n}(\vgam)}
         {\CBM{n}(\vgam)}
      &\geq 1\, ,
         \quad \vgam \in \Gamma_{M,n}\, ,
  \end{align*}
  is indeed satisfied for $M = 1$.

  Consider now an arbitrary $M \in \sZ_{\geq 2}$. By the induction hypothesis,
  we have
  \begin{align}
    \frac{ \Cgen{M-1}{n}(\vgam) }{\CBgen{M-1}{n}(\vgam) }
      &\geq 
        1\, , \qquad \vgam \in \Gamma_{M-1,n}\, .
    \label{eqn: ineq. used in induction hypothesis w.r.t. cbm}
  \end{align}
  Then we obtain
  \begin{align*}
    \CM{n}( \vgam )
    &\overset{(a)}{=} \sum_{\sigma_1 \in \setS_{[n]}( \vgam )}
         \Cgen{M-1}{n}\bigl( 
            \vgam_{\sigma_1}
    \bigr) 
    \nonumber \\
    &\overset{(b)}{\geq}
    \sum_{\sigma_1 \in \setS_{[n]}( \vgam )}
         \CBgen{M-1}{n}\bigl( 
            \vgam_{\sigma_1}
    \bigr)
    \nonumber \\ 
    &\overset{(c)}{=} \perm(\hvgamRCp) \cdot \CBM{n}(\vgam)
    \nonumber \\
    &\overset{(d)}{\geq} \CBM{n}(\vgam)\, ,
\end{align*}
where step~$(a)$ follows from Lemma~\ref{sec:1:lem:29:copy:2}, where step
$(b)$ follows from the induction hypothesis in~\eqref{eqn: ineq. used in
  induction hypothesis w.r.t. cbm}, where step~$(c)$ follows from
Lemma~\ref{sec:1:lem:15:copy:2}, and where step~$(d)$ follows from
Lemma~\ref{sec:1:lem:17}. 

\item A similar line of reasoning can be used to prove the second inequality
  in~\eqref{sec:1:eqn:58}; the details are omitted.

\end{itemize}
\etheoremproof


\medskip

\noindent
\textbf{Proof of the inequalities in~\eqref{sec:1:eqn:147} in
  Theorem~\ref{th:main:permanent:inequalities:1}}.
\begin{itemize}

\item We prove the first inequality in~\eqref{sec:1:eqn:147} by observing that
  \begin{align*}
    \bigl( \perm(\mtheta) \bigr)^{\!M}
      &\overset{(a)}{=} 
         \sum_{\vgam \in \GamMnthe} 
           \mtheta^{ M \cdot \vgam }
           \cdot
           \CM{n}( \vgam )
             \nonumber \\
      &\overset{(b)}{\geq}
         \sum_{\vgam \in \GamMnthe} 
           \mtheta^{ M \cdot \vgam }
           \cdot
           \CBM{n}( \vgam )
             \nonumber \\
      &\overset{(c)}{=}
         \bigl( \permbM{M}(\mtheta) \bigr)^{\!M}\, ,
  \end{align*}
  where step~$(a)$ follows from~\eqref{sec:1:eqn:43}, where step~$(b)$ follows
  from the first inequality in~\eqref{sec:1:eqn:58}, and where step~$(c)$
  follows from~\eqref{sec:1:eqn:190}. The desired inequality is then obtained
  by taking the $M$-th root on both sides of the expression.

\item A similar line of reasoning can be used to prove the second inequality
  in~\eqref{sec:1:eqn:147}; the details are omitted.

\end{itemize}
\etheoremproof


\medskip

The expression in the following lemma can be used to obtain a slightly
alternative proof of the inequalities in~\eqref{sec:1:eqn:147} in
Theorem~\ref{th:main:permanent:inequalities:1}. However, this result is also
of interest in itself.


\begin{lemma}
  \label{lemma:ratio:perm:permbM:1}

  Let $\pmtheta$ be the probability mass function on $\setS_{[n]}$ induced
  by~$\mtheta$ (see Definition~\ref{sec:1:def:10}). It holds that
  \begin{align}
    &
    \frac{\perm(\mtheta)}
         {\permbM{M}(\mtheta)} \nonumber \\
      &= \left( \!
           \sum_{ \vsigma_{[M]} \in \setS_{[n]}^{M} }
             \left(
               \prod_{m \in [M]} \!\!
                 \pmtheta(\sigma_m)
             \right)
             \! \cdot \!
             \frac{\CBM{n}
                     \Bigl(
                       \bigl\langle \mP_{\sigma_{m}} \bigr\rangle_{m \in [M]}
                     \Bigr)}
                  {\CM{n}
                     \Bigl(
                       \bigl\langle \mP_{\sigma_{m}} \bigr\rangle_{m \in [M]}
                     \Bigr)}
         \right)^{\!\!\! -1/M} \!\! .
           \label{eq:ratio:perm:permbM:1}
  \end{align}
  Note that from the first inequality in~\eqref{sec:1:eqn:58}
    in Theorem~\ref{thm: inequalities for the coefficients} it follows that
    $ \CM{n}(\vgam) \geq \CBM{n}(\vgam) > 0 $ for all $ \vgam \in \GamMn $,
    \ie, the ratio appearing in~\eqref{eq:ratio:perm:permbM:1} is well
    defined.
\end{lemma}


\begin{proof}
  See Appendix~\ref{app:proof:lemma:ratio:perm:permbM:1}.
\end{proof}

Based on Lemma~\ref{lemma:ratio:perm:permbM:1}, we  
give an alternative proof of the inequalities in~\eqref{sec:1:eqn:147} in
  Theorem~\ref{th:main:permanent:inequalities:1} in Appendix~\ref{apx: alternative proof of perm geq permb}.

Note that the first inequality in~\eqref{sec:1:eqn:147} is tight. Indeed,
choosing the matrix $\mtheta$ to be diagonal with strictly positive diagonal
entries shows that $\perm(\mtheta) / \permbM{M}(\mtheta) = 1$ for any
$M \in \sZpp$. Furthermore,
$\perm(\mtheta) / \permb(\mtheta) = \liminf_{M \to \infty} \perm(\mtheta) /
\permbM{M}(\mtheta)= 1$.


On the other hand, the second inequality in~\eqref{sec:1:eqn:147} is
\emph{not} tight for finite $M \geq 2$. Indeed, for this inequality to be
tight, all the inequalities
in~\eqref{eq:ratio:perm:permbM:proof:upper:bound:1} in Appendix~\ref{app:proof:lemma:ratio:perm:permbM:1} need to be tight. In
particular, the inequality in step~$(a)$ needs to be tight, which means that
for all $\vsigma_{[M]} \in \setS_{[n]}^{M}$ with
$\prod_{m \in [M]} \pmtheta(\sigma_m) > 0$, it needs to hold that
$\CBM{n}(\ldots) / \CM{n}(\ldots) = (2^{n/2})^{-(M-1)}$. However, this is not
the case for $\vsigma_{[M]} = (\sigma, \ldots, \sigma)$ with
$\sigma \in \setS_{[n]}(\vtheta)$, for which it holds that
$\CBM{n}(\ldots) / \CM{n}(\ldots) = 1$.


Observe that the non-tightness of the second inequality
in~\eqref{sec:1:eqn:147} for finite $M \geq 2$ is in contrast to the case
$M \to \infty$, where it is known that for
$\mtheta = \matr{I}_{(n/2) \times (n/2)} \otimes \bigl( \begin{smallmatrix} 1
  & 1 \\ 1 & 1 \end{smallmatrix} \bigr)$ it holds that
$\perm(\mtheta) / \permb(\mtheta) = \liminf_{M \to \infty} \perm(\mtheta) /
\permbM{M}(\mtheta)= 2^{n/2}$. (Here, $n$ is assumed to be even and
$\matr{I}_{(n/2) \times (n/2)}$ is defined to be the identity matrix of size
$(n/2) \times (n/2)$.)


\vspace{0.2 cm}

\noindent
We also give another alternative proof of the inequalities in~\eqref{sec:1:eqn:147} in
  Theorem~\ref{th:main:permanent:inequalities:1} in Appendix~\ref{apx: alternative prooof of perm geq permb: 2}.
In particular, Lemma~\ref{prop: a special case of Navin conjecture} proven in Appendix~\ref{apx: alternative prooof of perm geq permb: 2} proves the $M_{1}= 1$ case of the following conjecture.
\begin{conjecture}\label{Navin conjecture}
  For any integers $ M_{1} \in \sZpp $ and $ M_{2} \in \sZpp $, it holds that
  \begin{align*}
    \frac{ \bigl( \permbM{M_{1}}(\mtheta) \bigr)^{\! M_{1}} \cdot
    \bigl( \permbM{M_{2}}(\mtheta) \bigr)^{\! M_{2}}
    }{ \bigl( \permbM{M_{1} + M_{2}}(\mtheta) \bigr)^{\! M_{1} + M_{2}}}
    \geq 1\, .
  \end{align*}
  \econjecture
\end{conjecture}
A useful consequence of Conjecture~\ref{Navin conjecture} is that
\begin{align*}
  \lim_{M \to \infty} 
  \permbM{M}(\mtheta) 
  = \limsup_{M \to \infty} 
  \permbM{M}(\mtheta) 
  =
  \permb(\mtheta)\, ,
\end{align*}
where the first equality follows from Conjecture~\ref{Navin conjecture} and Fekete’s lemma~\cite{Fekete1923} (see, \eg,~\cite[Lemma 11.6]{Lint2001}), and where the second equality follows from~\cite[Theorem 39]{Vontobel2013a}.


\subsection{\texorpdfstring{Bounding the ratio $\perm(\mtheta) /  
                                                                                     \permscsM{M}(\mtheta)$}{}}


In this subsection, we first prove the inequalities in~\eqref{sec:1:eqn:180}
in Theorem~\ref{thm: inequalities for the coefficients} that give bounds on
the ratio $\CM{n}(\vgam) / \CscSgen{M}{n}(\vgam)$. Then, based on these
bounds, we prove the inequalities in~\eqref{sec:1:eqn:200} in
Theorem~\ref{th:main:permanent:inequalities:1} that give bounds on the ratio
$\perm(\mtheta) / \permscsM{M}(\mtheta)$.


\begin{lemma}
  \label{sec:1:lem:36}

  Consider the setup in Lemma~\ref{sec:1:lem:30}. The matrix $\vgam$ satisfies
  \begin{align*}
    \frac{n!}{n^{n}}
      &\leq 
         \perm(\vgam)
       \leq 1\, .
  \end{align*}
\end{lemma}


\begin{proof}
  The lower bound follows from van der Waerden's inequality proven
  in~\cite[Theorem 1]{Egorychev1981} and~\cite[Theorem 1]{Falikman1981}. The
  upper bound follows from
  \begin{align*}
    \perm(\vgam) 
      &\overset{(a)}{\leq }
         \prod_{i} 
           \Biggl( \sum_{j} \gamma(i,j) \Biggr)
       = 
         \prod_{i} 
           1
       = 1\, ,
  \end{align*}
  where step~$(a)$ follows from the observation that $\perm(\vgam)$ contains a
  subset of the (non-negative) terms in
  $\prod_{i} \bigl( \sum_{j} \gamma(i,j) \bigr)$.
\end{proof}


\noindent
\textbf{Proof of the inequalities in~\eqref{sec:1:eqn:180} in Theorem~\ref{thm:
    inequalities for the coefficients}}.
\begin{itemize}

\item We prove the first inequality in~\eqref{sec:1:eqn:180} by induction on
  $M$. The base case is $M = 1$. From Definition~\ref{sec:1:lem:43}, it
  follows that $\CM{n}(\vgam) = \CscSgen{M}{n}(\vgam) = 1$ for all
  $\vgam \in \Gamma_{M,n}$ and so
  \begin{align*}
    \frac{\CM{n}(\vgam)}
         {\CscSgen{M}{n}(\vgam)}
      &\geq
         1\, , 
           \qquad \vgam \in \Gamma_{M,n}\, .
  \end{align*}
  (Note that when $M = 1$, then the first inequality in~\eqref{sec:1:eqn:180}
  simplifies to $1$.)

  Consider now an arbitrary $ M \in \sZ_{\geq 2} $. By the induction
  hypothesis, we have
  \begin{align}
    \frac{\Cgen{M-1}{n}(\vgam)}
         {\CscSgen{M-1}{n}(\vgam)} 
      &\! \geq \!
         \Biggl( \!
           \frac{(M \! - \! 1)^{M-1}}
                {(M \! - \! 1)!} \!
         \Biggr)^{\!\!\! n} \!\!
         \cdot \!
         \biggl(
           \frac{ n! }
                { n^{n} }
         \biggr)^{\!\! M-2} \!\!\!\! ,
           \, \vgam \in \Gamma_{M-1,n}\, .
    \label{eqn: ineq. used in induction hypothesis w.r.t. cscsm}
  \end{align}
  Then we obtain
  \begin{align*}
    &\CM{n}( \vgam ) \nonumber \\
    &\overset{(a)}{=}
       \sum_{\sigma_1 \in \setS_{[n]}( \vgam )}
         \Cgen{M-1}{n}
           \bigl( 
             \vgam_{\sigma_1}
           \bigr) 
             \nonumber \\
    &\overset{(b)}{\geq}
       \biggl( \frac{(M \! - \! 1)^{M-1}}{(M \! - \! 1)!} \biggr)^{\!\! n}
       \cdot
       \biggl(
         \frac{ n! }{ n^{n} }
       \biggr)^{\!\! M-2} 
       \!\!\!\!\!\!
       \cdot
       \sum_{\sigma_1 \in \setS_{[n]}( \vgam )} \!\!
         \CscSgen{M-1}{n}(\vgam_{\sigma_1})
           \nonumber \\ 
    &\overset{(c)}{=} 
       \biggl( \frac{M^{M}}{M!} \biggr)^{\!\! n}
       \cdot
       \biggl(
         \frac{ n! }{ n^{n} }
       \biggr)^{\!\! M-2} 
       \cdot
       \perm(\vgam)
       \cdot
       \CscSgen{M}{n}(\vgam)
         \nonumber \\
    &\overset{(d)}{\geq}
       \biggl( \frac{M^{M}}{M!} \biggr)^{\!\! n}
       \cdot
       \biggl(
         \frac{ n! }{ n^{n} }
       \biggr)^{\!\! M-1} 
       \cdot 
       \CscSgen{M}{n}(\vgam)\, ,
  \end{align*}
  where step~$(a)$ follows from Lemma~\ref{sec:1:lem:29:copy:2}, where step
  $(b)$ follows from the induction hypothesis 
  in~\eqref{eqn: ineq. used in induction hypothesis w.r.t. cscsm}, 
  where step~$(c)$ follows from
  Lemma~\ref{sec:1:lem:30}, and where step~$(d)$ follows from
  Lemma~\ref{sec:1:lem:36}.

\item A similar line of reasoning can be used to prove the second inequality
  in~\eqref{sec:1:eqn:180}; the details are omitted.

\end{itemize}
\etheoremproof


\medskip

\noindent
\textbf{Proof of the inequalities in~\eqref{sec:1:eqn:200} in
  Theorem~\ref{th:main:permanent:inequalities:1}}.
\begin{itemize}

\item We prove the first inequality in~\eqref{sec:1:eqn:200} by observing that
  \begin{align*}
    \hspace{0.25cm}&\hspace{-0.25cm}
    \bigl( \perm(\mtheta) \bigr)^{\!M} \\
      &\overset{(a)}{=}
         \sum_{\vgam \in \GamMnthe} 
           \mtheta^{ M \cdot \vgam }
         \cdot
         \CM{n}( \vgam )
           \nonumber \\
      &\overset{(b)}{\geq} 
         \biggl( \frac{M^{M}}{M!} \biggr)^{\!\! n}
         \cdot
         \biggl(
           \frac{ n! }{ n^{n} }
         \biggr)^{\!\! M-1} 
         \! \cdot \!\!\!\!\!
         \sum_{\vgam \in \GamMnthe} 
           \mtheta^{ M \cdot \vgam }
         \cdot
         \CscSgen{M}{n}( \vgam )
          \nonumber \\
      &\overset{(c)}{=} 
         \biggl( \frac{M^{M}}{M!} \biggl)^{\!\! n}
         \cdot
         \biggl(
           \frac{ n! }{ n^{n} }
         \biggl)^{\!\! M-1} 
         \cdot 
         \bigl( \permscsM{M}(\mtheta) \bigr)^{\!M}\, ,
  \end{align*}
  where step~$(a)$ follows from~\eqref{sec:1:eqn:43}, where step~$(b)$ follows
  from the first inequality in~\eqref{sec:1:eqn:180}, and where step~$(c)$
  follows from~\eqref{sec:1:eqn:168}. The desired inequality is then obtained
  by taking the $M$-th root on both sides of the expression.

\item A similar line of reasoning can be used to prove the second inequality
  in~\eqref{sec:1:eqn:200}; the details are omitted.

\etheoremproof
\end{itemize}


The expression in the following lemma can be used to obtain a slightly
alternative proof of the inequalities in~\eqref{sec:1:eqn:200} in
Theorem~\ref{th:main:permanent:inequalities:1}. However, this result is also
of interest in itself.


\begin{lemma}
  \label{lemma:ratio:perm:permscM:1}

  Let $\pmtheta$ be the probability mass function on $\setS_{[n]}$ induced
  by~$\mtheta$ (see Definition~\ref{sec:1:def:10}). It holds that
  \begin{align*}
    &
    \frac{\perm(\mtheta)}
         {\permscsM{M}(\mtheta)} \nonumber \\
      &= \left( \!
           \sum_{ \vsigma_{[M]} \in \setS_{[n]}^{M} } \!
             \Biggl(
               \prod_{m \in [M]} \!\!\!
                 \pmtheta(\sigma_m)
             \Biggr)
             \!\! \cdot \!
             \frac{\CscSgen{M}{n}
                     \Bigl(
                       \bigl\langle \mP_{\sigma_{m}} \bigr\rangle_{m \in [M]}
                     \Bigr)}
                  {\CM{n}
                     \Bigl(
                       \bigl\langle \mP_{\sigma_{m}} \bigr\rangle_{m \in [M]}
                     \Bigr)} \!
         \right)^{\!\!\! -1/M} \!\! .
  \end{align*}
\end{lemma}


\begin{proof}
  The proof of this statement is analogous to the proof of
  Lemma~\ref{lemma:ratio:perm:permbM:1} in
  Appendix~\ref{app:proof:lemma:ratio:perm:permbM:1}; one simply has to
  replace $\permbM{M}(\mtheta)$ and $\CBM{n}(\ldots)$ by
  $\permscsM{M}(\mtheta)$ and $\CscSgen{M}{n}(\ldots)$, respectively.
\end{proof}


\noindent
\textbf{Alternative proof of the inequalities in~\eqref{sec:1:eqn:200} in
  Theorem~\ref{th:main:permanent:inequalities:1}}.

\noindent
Based on Lemma~\ref{lemma:ratio:perm:permscM:1}, an alternative proof of the
inequalities in~\eqref{sec:1:eqn:200} can be given that is very similar to the
alternative proof presented
in Appendix~\ref{apx: alternative proof of perm geq permb}; 
the details are omitted.
\etheoremproof


\medskip

Note that the second inequality in~\eqref{sec:1:eqn:200} is tight. Indeed,
choosing the matrix $\mtheta$ to be diagonal with strictly positive diagonal
entries shows that
$\perm(\mtheta) / \permscsM{M}(\mtheta) = \frac{M^{n}}{(M!)^{n/M}}$ for any
$M \in \sZpp$. Furthermore,
$\perm(\mtheta) / \permscs(\mtheta) = \liminf_{M \to \infty} \perm(\mtheta) /
\permscsM{M}(\mtheta)= \e^n$.


On the other hand, the first inequality in~\eqref{sec:1:eqn:200} is \emph{not}
tight for finite $M \geq 2$. (The proof is similar to the proof of the
non-tightness of the second inequality in~\eqref{sec:1:eqn:147} for finite
$M \geq 2$.) This is in contrast to the case $M \to \infty$, where it is known
that for $\mtheta$ being equal to the all-one matrix of size $n \times n$ it
holds that
$\perm(\mtheta) / \permscs(\mtheta) = \liminf_{M \to \infty} \perm(\mtheta) /
\permscsM{M}(\mtheta)= \e^n \cdot \frac{n!}{n^n}$.


\ifx\withclearpagecommands\x
\clearpage
\fi

\section{Asymptotic Expressions for the 
              Coefficients \texorpdfstring{$\CM{n}( \vgam )$, 
                                   $\CBM{n}( \vgam )$, $\CscSgen{M}{n}( \vgam )$}{}}
\label{sec:asymptotic:1}


In this section, we prove
Proposition~\ref{prop:coefficient:asymptotitic:characterization:1}. Whereas
Definitions~\ref{sec:1:def:18:part:2}, \ref{sec:1:def:16},
and~\ref{sec:1:def:20} give an analytical definition of the entropy functions
$\HG'$, $ \HBthe $, and $ \HscSthe $, respectively, the expressions in
Eqs.~\eqref{sec:1:eqn:125}--\eqref{sec:1:eqn:207} in
Proposition~\ref{prop:coefficient:asymptotitic:characterization:1} give a
combinatorial characterization of these entropy functions. The proof of
Proposition~\ref{prop:coefficient:asymptotitic:characterization:1} is based on
standard results of the method of types and considerations that are similar to
the considerations in~\cite[Sections~IV and~VII]{Vontobel2013}.



\begin{definition}
  \label{sec:1:def:17}

  Fix some $ M \in \sZpp $ and $ \vgam \in \Gamma_{n} $. 
  Recall the definition of the notations
  $\vsigma_{[M]}$ and $\mP_{\vsigma_{[M]}}$ in
  Definition~\ref{sec:1:def:10}. We define the following objects.
  \begin{enumerate}

  \item Let $\mP_{\vsigma_{[M]}} \in \setA(\vgam)^{M}$. The type of
    $ \mP_{\vsigma_{[M]}} $ is defined to be
    \begin{align*}
      \bm{t}_{\mP_{\vsigma_{[M]}}}
        &\defeq 
           \bigl(
             t_{ \mP_{\vsigma_{[M]}} }( \mP_{\sigma} ) 
          \bigr)_{ \! \mP_{\sigma} \in \setA(\vgam) }
           \in \PiA{\vgam}\, ,
    \end{align*}
    where
    \begin{align*}
      t_{\mP_{\vsigma_{[M]}}}( \mP_{\sigma} )
        &\defeq
           \frac{1}{M} 
           \cdot 
           \Bigl| 
             \bigl\{ 
               m \in [M] 
             \bigm|
               \mP_{\sigma_{m}} = \mP_{\sigma} 
             \bigr\}
           \Bigr|\, .
    \end{align*}

  \item Let $ \set{B}_{\setA(\vgam)^{M}} $ be the set of all the possible types
   that are based on the elements in $\setA(\vgam)^{M}$:
    \begin{align*}
      \set{B}_{\setA(\vgam)^{M}}
        &\defeq 
           \left\{ 
             \bm{t} 
               \defeq 
                 \Bigl( t( \mP_{\sigma} ) \Bigr)_{ \!\! \mP_{\sigma} \in \setA(\vgam) }
           \ \middle|
             \begin{array}{c}
               \exists \mP_{\vsigma_{[M]}} \in \setA(\vgam)^{M} \\
               \text{s.t.} \ \bm{t}_{\mP_{\vsigma_{[M]}}} = \bm{t}
             \end{array}
           \right\}\, .
    \end{align*}

  \item Let $ \bm{t} \in \set{B}_{\setA(\vgam)^{M}} $. The type class of
    $ \bm{t} $ is defined to be
    \begin{align*}
      \set{T}_{\bm{t}}
        &\defeq 
           \bigl\{ 
             \mP_{\vsigma_{[M]}} \in \setA(\vgam)^{M}
           \bigm|
             \bm{t}_{\mP_{\vsigma_{[M]}}} = \bm{t}
           \bigr\}\, .
    \end{align*}

  \end{enumerate}
  \edefinition
\end{definition}


\begin{lemma}
  \label{sec:1:lem:21}

  Consider $ M \in \sZpp $ and $ \vgam \in \Gamma_{n} $. It holds that
  \begin{align*}
    \bigl| \set{B}_{\setA(\vgam)^{M}} \bigr|
      &= \binom{M + |\setA(\vgam)| - 1}{|\setA(\vgam)| - 1} 
           \in O\bigl( M^{|\setA(\vgam)|-1} \bigr)\, , \nonumber \\
    |\set{T}_{\bm{t}}| 
      &= \frac{ M! }
              { \prod\limits_{ \mP_{\sigma} \in \setA(\vgam) } 
                  \bigl( M \cdot t(\mP_{\sigma}) \bigr)! }\, ,
           \quad \bm{t} \in \set{B}_{\setA(\vgam)^{M}}\, .
  \end{align*}
\end{lemma}

\begin{proof}
  These are standard combinatorial results and so the details of the
  derivation are omitted.
\end{proof}


\noindent
\textbf{Proof sketch of
  Proposition~\ref{prop:coefficient:asymptotitic:characterization:1}}.
\begin{itemize}

\item The expression in~\eqref{sec:1:eqn:125} can be proven as
  follows. Namely, let $\vgam \in \GamMn$. For a large integer $M$, we obtain
  \begin{align*}
    \CM{n}( \vgam )
      &\overset{(a)}{=}
         \sum_{ \vsigma_{[M]} \in \setS_{[n]}^{M} }
           \Bigl[ 
             \vgam = \bigl\langle \mP_{\sigma_{m}} \bigr\rangle_{m \in [M]}
           \Bigr] \\
      &\overset{(b)}{=}
         \sum_{ \bm{t} \in \set{B}_{\setA(\vgam)^{M}} }
           \underbrace
           {
             \sum_{ \mP_{\vsigma_{[M]}} \in \set{T}_{\bm{t}}}
               \Bigl[ 
                 \vgam = \bigl\langle \mP_{\sigma_{m}} \bigr\rangle_{m \in [M]}
               \Bigr]
           }_{\defeq \ (*)} \\
      &\overset{(c)}{=}
         \sum_{ \bm{t} \in \set{B}_{\setA(\vgam)^{M}} \cap \PiA{\vgam}( \vgam )}
           |\set{T}_{\bm{t}}| \\
      &\overset{(d)}{=}
         \sum_{ \bm{t} \in \set{B}_{\setA(\vgam)^{M}} \cap \PiA{\vgam}( \vgam )}
           \exp\bigl( M \cdot \HG( \bm{t} ) + o( M ) \bigr) \\
      &\overset{(e)}{=}
         \max_{ \bm{t} \in \set{B}_{\setA(\vgam)^{M}} \cap \PiA{\vgam}( \vgam )}
           \exp\bigl( M \cdot \HG( \bm{t} ) + o( M ) \bigr) \, ,
  \end{align*}
  where step~$(a)$ follows from the definition of $\CM{n}( \vgam )$
  in~\eqref{sec:1:eqn:56}, where step~$(b)$ follows from the definition of
  $\set{B}_{\setA(\vgam)^{M}}$ and $\set{T}_{\bm{t}}$, where
  step~$(c)$ follows from the observation that $(*)$ equals
  $|\set{T}_{\bm{t}}|$ if 
  $\bm{t} \in \set{B}_{\setA(\vgam)^{M}} \cap \PiA{\vgam}(\vgam)$ and equals zero
  if $\bm{t} \in \set{B}_{\setA(\vgam)^{M}} \setminus \PiA{\vgam}(\vgam)$,
  where step~$(d)$ follows from 
  Stirling's approximation of the factorial function (see, \eg,~\cite{Robbins1955}), 
  Lemma~\ref{sec:1:lem:21}, and the definition of
  $\HG$ in Definition~\ref{sec:1:def:18}, and where step~$(e)$ follows from
  the fact that the size of the set $\set{B}_{\setA(\vgam)^{M}}$ is polynomial
  in $M$ (see Lemma~\ref{sec:1:lem:21}) and therefore can be absorbed in the
  $o(M)$ term in the exponent.
  
  In the limit $M \to \infty$, the set $\set{B}_{\setA(\vgam)^{M}}$ is dense
  in $\PiA{\vgam}( \vgam )$, and so
  \begin{align*}
    \CM{n}( \vgam )
      &= \max_{ \bm{t} \in \PiA{\vgam}( \vgam )}
           \exp\bigl( M \cdot \HG( \bm{t} ) + o( M ) \bigr) \\
      &\overset{(a)}{=}
         \exp\bigl( M \cdot \HG'( \vgam ) + o( M ) \bigr)\, ,
  \end{align*}
  where step~$(a)$ used the definition of $\HG'$ in
  Definition~\ref{sec:1:def:18:part:2}.

\item For the proof of the expression in~\eqref{sec:1:eqn:206},
  see~\cite[Section~VII]{Vontobel2013}.

\item A similar line of reasoning as above can be used to prove the expression
  in~\eqref{sec:1:eqn:207}; the details are omitted.

  \epropositionproof

\end{itemize}


\ifx\withclearpagecommands\x
\clearpage
\fi

\section{Results for the Permanent and the 
               \texorpdfstring{\\ Degree-$2$ Bethe Permanent}{}}
\label{apx:22}


The upcoming proposition, Proposition~\ref{prop:ratio:perm:permbM:2:1}, is the
same as~\cite[Proposition~2]{KitShing2022}. The purpose of this section is
to rederive this proposition based on some of the results that have been
established so far in this paper.


\begin{proposition}
  \label{prop:ratio:perm:permbM:2:1}
  
  Let $\pmtheta$ be the probability mass function on $\setS_{[n]}$ induced
  by~$\mtheta$ (see Definition~\ref{sec:1:def:10}). It holds that
  \begin{align}
    \frac{ \perm(\mtheta) }{ \permbM{2}(\mtheta) }
      &= \Biggl(
            \sum\limits_{ \sigma_1,\sigma_{2} \in \setS_{[n]} } 
              \pmtheta( \sigma_1 )
              \cdot
              \pmtheta( \sigma_{2} ) 
              \cdot
              2^{-c(\sigma_1,\sigma_2)}
         \Biggr)^{\!\!\! -1/2}\, ,
           \label{eq:ratio:perm:permbM:2:result:1}
  \end{align}
  where $c(\sigma_1,\sigma_2)$ is the number of cylces of length larger than
  one in the cycle notation
  expression\footnote{\label{footnote:cycle:notation:example:1}%
    Let us recall the cycle notation for permutations with the help of an
    example. Namely, consider the permutation $\sigma: [7] \to [7]$ defined by
    $\sigma(1) = 1$, $\sigma(2) = 2$, $\sigma(3) = 4$, $\sigma(4) = 3$,
    $\sigma(5) = 6$, $\sigma(6) = 7$, $\sigma(7) = 5$. In terms of the cycle
    notation, $\sigma = (1)(2)(34)(567)$. As can be seen from this expression,
    the permutation $\sigma$ contains two cycles of length $1$ (one cycle
    involving only~$1$ and one cycle involving only~$2$), one cycle of length
    $2$ (involving $3$ and $4$), and one cycle of length $3$ (involving $5$,
    $6$, and $7$).} of the permutation ${\sigma_1 \circ \sigma_2^{-1}}$. Here,
  $\sigma_1 \circ \sigma_2^{-1}$ represents the permutation that is obtained
  by first applying the inverse of the permutation $\sigma_2$ and then the
  permutation $\sigma_1$.
  \eproposition
 \end{proposition}


Note that $c(\sigma_1,\sigma_2)$ satisfies\footnote{While the lower bound on
  $c(\sigma_1,\sigma_2)$ is immediately clear, the upper bound on
  $c(\sigma_1,\sigma_2)$ follows from permutations of $[n]$ that contain
  $\lfloor n/2 \rfloor $ cycles of length~$2$.}
\begin{align*}
  0 
    &\leq 
       c(\sigma_1,\sigma_2) 
     \leq 
       \Bigl\lfloor \frac{n}{2} \Bigr\rfloor \, , 
       \qquad \sigma_1,\sigma_{2} \in \setS_{[n]}(\mtheta)\, .
\end{align*}
These bounds on $c(\sigma_1,\sigma_2)$ yield
\begin{align}
  1
    &\overset{(a)}{\leq}
       \frac{ \perm(\mtheta) }{ \permbM{2}(\mtheta) }
     \overset{(b)}{\leq}
       \sqrt{2^{\lfloor n/2 \rfloor}}
     \leq
        2^{n/4}\, ,
          \label{eq:bounds:on"ratio:perm:permbM:2:1}
\end{align}
where the inequality~$(a)$ is obtained by lower bounding the expression
in~\eqref{eq:ratio:perm:permbM:2:result:1} with the help of
$c(\sigma_1,\sigma_2) \geq 0$ and where the inequality~$(b)$ is obtained by
upper bounding the expression in~\eqref{eq:ratio:perm:permbM:2:result:1} with
the help of $c(\sigma_1,\sigma_2) \leq \lfloor n/2 \rfloor$. Note that the
lower and upper bounds in~\eqref{eq:bounds:on"ratio:perm:permbM:2:1} on the
ratio $\perm(\mtheta) / \permbM{2}(\mtheta)$ match exactly the lower and upper
bounds in~\eqref{sec:1:eqn:147} for $M = 2$.


\begin{remark}
  A variant of Proposition~\ref{prop:ratio:perm:permbM:2:1} (and with that
  of~\cite[Proposition 2]{KitShing2022}), appears in~\cite[Lemma
  4.2]{Csikvari2017}. The setup in~\cite{Csikvari2017} is more special than
  that in~\cite{KitShing2022} in the sense that Csikv{\'a}ri
  in~\cite{Csikvari2017} considered $\{0,1\}$-valued matrices only, not
  arbitrary non-negative matrices. However, the result in~\cite{Csikvari2017}
  is more general than the result in~\cite{KitShing2022} in the sense that
  Csikv{\'a}ri studied matchings of size $k \in [n]$, and not just matchings
  of size $n$ (\ie, perfect matchings).  \eremark
\end{remark}


\noindent
\textbf{Alternative proof Proposition~\ref{prop:ratio:perm:permbM:2:1}.} 
Let us now show how the results in this paper can be used to rederive
Proposition~\ref{prop:ratio:perm:permbM:2:1}. Namely, we obtain
\begin{align*}
  &
  \frac{\perm(\mtheta)}
       {\permbM{2}(\mtheta)} \nonumber \\
    &\overset{(a)}{=}
       \left( \!
         \sum_{ \sigma_1, \sigma_2 \in \setS_{[n]} }
           \pmtheta(\sigma_1)
           \cdot
           \pmtheta(\sigma_2)
           \cdot
           \frac{\CBtwo{n}\Bigl( 
                            \frac{1}{2} 
                            \bigl( 
                              \mP_{\sigma_1} \! + \! \mP_{\sigma_2}
                            \bigr)
                          \Bigr)}
                {\Ctwo{n}\Bigl( 
                           \frac{1}{2} 
                           \bigl( 
                             \mP_{\sigma_1} \! + \! \mP_{\sigma_2}
                           \bigr)
                         \Bigr)}
       \right)^{\!\!\! -1/2} \\
    &\overset{(b)}{=}
       \Biggl( 
         \sum_{ \sigma_1, \sigma_2 \in \setS_{[n]} }
           \pmtheta(\sigma_1)
           \cdot
           \pmtheta(\sigma_2)
           \cdot
           \frac{1}
                {2^{c(\sigma_1,\sigma_2)}}
       \Biggr)^{\!\!\! -1/2}\, ,
\end{align*}
which is the same expression as
in~\eqref{eq:ratio:perm:permbM:2:result:1}. Here, step~$(a)$ follows from
Lemma~\ref{lemma:ratio:perm:permbM:1} for $M = 2$ and step~$(b)$ follows from
\begin{alignat}{2}
  \CBtwo{n}\biggl( 
             \frac{1}{2} 
             \bigl( 
               \mP_{\sigma_1} \! + \! \mP_{\sigma_2} 
             \bigr)
           \biggr)
    &= 1 \, , 
         &&\quad \sigma_1, \sigma_2 \in \setS_{[n]}\, ,
             \label{eq:CBtwo:result:1} \\
  \Ctwo{n}\biggl( 
            \frac{1}{2} 
            \bigl( 
              \mP_{\sigma_1} \! + \! \mP_{\sigma_2} 
            \bigr)
          \biggr)
    &= 2^{c(\sigma_1,\sigma_2)},
         &&\quad \sigma_1, \sigma_2 \in \setS_{[n]}\, .
             \label{eq:Ctwo:result:1}
\end{alignat}
Eqs.~\eqref{eq:CBtwo:result:1} and~\eqref{eq:Ctwo:result:1} are established in
Appendix~\ref{app:alt:proof:prop:ratio:perm:permbM:2:1}.
\epropositionproof


\ifx\withclearpagecommands\x
\clearpage
\fi

\section{Conclusion and Open Problems}
\label{sec:7}


We have shown that it is possible to bound the permanent of a non-negative
matrix by its degree-$M$ Bethe and scaled Sinkhorn permanents, thereby, among
other statements, proving a conjecture
in~\cite{Vontobel2013a}. A key result are the recursive
  expressions for the coefficients appearing in the expressions for the $M$-th
  power of the degree-$M$ Bethe permanent and the $M$-th power of the
  degree-$M$ scaled Sinkhorn permanent.
   

While the results in this paper do not yield any new numerical techniques for
bounding the permanent of a non-negative matrix, they can, besides being of
inherent combinatorial interest, potentially be used to obtain new analytical
results about the Bethe and the Sinkhorn approximation of the permanent,
similar to the results in the paper~\cite{KitShing2022}. (See also the
discussion at the end of Section~\ref{sec:introduction:1}.) 


Some open problems include the following:
\begin{itemize}

\item Our proofs used some rather strong results from~\cite{Schrijver1998,
    Gurvits2011, Anari2019, Egorychev1981, Falikman1981}. We leave it as an
  open problem to find ``more basic'' proofs for some of the inequalities that
  were established in this paper.

\item One can verify that proving the first inequality
  in~\eqref{eq:ratio:permanent:bethe:permanent:1} for all~$\vtheta$ is
  equivalent to proving $ \HG'(\vgam) \geq \HBthe( \vgam ) $ for all
  $ \vgam \in \Gamma_{n} $. Note that proving the first inequality
  in~\eqref{eq:ratio:permanent:bethe:permanent:1} for all $ \vtheta $ is also
  equivalent to proving Schrijver's inequality~\cite{Schrijver1998}. 
  We leave it as an open problem
  to prove $ \HG'(\vgam) \geq \HBthe( \vgam ) $ for all
  $ \vgam \in \Gamma_{n} $ via an approach different from Schrijver's
  inequality and its variations.


\item It was remarked that both the second inequality in~\eqref{sec:1:eqn:147}
  and the first inequality in~\eqref{sec:1:eqn:200} are not tight. We leave it
  as an open problem to find tighter bounds, or even tight bounds.

\item Recently, Huang and Vontobel in~\cite{Huang2024} extended Vontobel's
  results in~\cite{Vontobel2013a} by considering a class of bipartite S-NFGs
  where each local function is defined based on a (possibly different)
  multi-affine homogeneous real stable polynomial. These factor graphs can be
  seen as an extension of the class of S-NFGs studied in this
  paper. Generalizing the results of this paper to the S-NFGs discussed
  in~\cite{Huang2024} remains an open problem for future research.

\end{itemize}


\ifx\withclearpagecommands\x
\clearpage
\fi

\begin{appendices}



\section{Proof of Proposition~\ref{sec:1:prop:14}}
\label{apx:24}


For a large integer $M$, we obtain
\begin{align*}
  &\bigl( \permscsM{M}(\mtheta) \bigr)^{\! M} \\
  &\overset{(a)}{=}
     \sum_{ \vgam \in \GamMnthe }
       \mtheta^{ M \cdot \vgam } \cdot \CscSgen{M}{n}( \vgam ) \\
  &\overset{(b)}{=}  \sum_{ \vgam \in \GamMnthe }
      \exp\bigl( - M \cdot \UscSthe(\vgam) \bigr)
      \cdot
      \CscSgen{M}{n}( \vgam ) \\
  &\overset{(c)}{=}  
    \exp( o(M) ) \cdot \!
    \max_{ \vgam \in \GamMnthe }
      \exp\bigl( - M \cdot \UscSthe(\vgam) \bigr)
      \cdot
      \CscSgen{M}{n}(\vgam)
    \nonumber \\
  &\overset{(d)}{=}  
    \exp( o(M) ) \cdot 
    \max_{ \vgam \in \GamMnthe } \exp\bigl( - M \cdot \FscSthe(\vgam) \bigr)\, ,
\end{align*}
where step~$(a)$ follows from~\eqref{sec:1:eqn:168}, where step~$(b)$ follows
from the expression of $ \UscSthe $ in Definition~\ref{sec:1:def:20}, where
step~$(c)$ follows from the fact that $ \bigl| \GamMnthe \bigr| $ grows
polynomially with respect to $ M $, and where step~$(d)$ follows
from~\eqref{sec:1:eqn:207} and the definition of $\FscSthe(\vgam)$ in
Definition~\ref{sec:1:def:20}. Then we have
\begin{align*}
  &\limsup_{M \to \infty}
     \permscsM{M}(\mtheta) \\
  &= \limsup_{M \to \infty} 
       \exp\biggl( \frac{o(M)}{M} \biggr)
       \cdot
       \sqrt[M]{\max_{ \vgam \in \GamMnthe } 
                    \exp\bigl( - M \cdot \FscSthe(\vgam) \bigr)} \\
  &= \limsup_{M \to \infty} 
       \max_{ \vgam \in \GamMnthe } 
         \exp\bigl( - \FscSthe(\vgam) \bigr) \\
  &\overset{(a)}{=}
     \sup_{ \vgam \in \Gampnthe } 
       \exp\bigl( - \FscSthe(\vgam) \bigr) \\
  &= \exp
       \biggl(
         -
         \inf_{ \vgam \in \Gampnthe }
           \FscSthe(\vgam)
       \biggr) \\
  &\overset{(b)}{=} 
     \exp
       \biggl(
         -
         \min_{ \vgam \in \Gamnthe }
           \FscSthe(\vgam)
       \biggr) \\
  &\overset{(c)}{=} 
     \permscs(\vtheta)\, ,
\end{align*}
where step~$(a)$ is based on the definition
\begin{align*}
  \Gampnthe
    &\defeq
       \bigcup_{M \in \sZpp}
         \GamMnthe\, ,
\end{align*}
where step~$(b)$ follows from the set $\Gampnthe$ being dense in the compact
set $\Gamnthe$, and where step~$(c)$ follows from the definition of
$\permscs(\vtheta)$ in~\eqref{sec:1:eqn:192:part:2}.


\ifx\withclearpagecommands\x
\clearpage
\fi

\section{Proof of Lemma~\ref{sec:1:lem:15:copy:2}}
\label{app:sec:1:lem:15:copy:2}


In this appendix, we use the following, additional notation and conventions:
\begin{itemize}

\item We use the abbreviation $r$ for $r(\vgam)$.

\item We define the matrices
  \begin{align*}
    \tvgam
      &\defeq 
         M \cdot \vgam\, , \\ 
    \tvgam_{\sigma_1}
      &\defeq 
         (M-1) \cdot \vgam_{\sigma_1}\, .
  \end{align*}
  Note that $\tvgam_{\sigma_1} = \tvgam - \mP_{\sigma_1}$.

\item A product over an empty index set evaluates to $1$. For example, if
  $\setR$ is the empty set, then $\prod_{i \in \setR} (...) = 1$.

\end{itemize}

\bigformulatop{30}{31}{%
\begin{align}
  \hspace{0.25cm}&\hspace{-0.25cm}
  \frac{1}{\CBM{n}(\vgam)}
       \cdot
       \sum_{\sigma_1 \in \setS_{[n]}( \vgam )}
         \CBgen{M-1}{n}
           \bigl( 
             \vgam_{\sigma_1}
           \bigr)
    \nonumber \\
    &\overset{(a)}{=}
      \frac{1}{\bigl( M! \bigr)^{ \! 2n - n^2}}
      \cdot
      \Biggl( 
        \prod_{i,j}
          \frac{ \bigl(\tgam(i,j) \bigr)! }
               { \bigl( M - \tgam(i,j) \bigr)! }
      \Biggr)
      \cdot
      \sum_{\sigma_1 \in \setS_{[n]}( \vgam )}
        \bigl( (M-1)! \bigr)^{ \! 2n -  n^2} 
        \cdot
        \Biggl( 
          \prod_{i,j}
            \frac{ \bigl( M-1- \tgam_{\sigma_1}(i,j) \bigr)!}
                 { \bigl( \tgam_{\sigma_1}(i,j) \bigr)! } 
        \Biggr) \nonumber \\
    &\overset{(b)}{=}
       \frac{1}{M^{ 2n -  n^2}}
       \cdot
       \sum_{\sigma_1 \in \setS_{[n]}( \vgam )}
       \Biggl( 
         \prod_{i,j}
           \underbrace
           {
             \frac{  \bigl(\tgam(i,j) \bigr)! }
                  { \bigl( \tgam_{\sigma_1}(i,j) \bigr)! } 
             \cdot
             \frac{ \bigl( M-1- \tgam_{\sigma_1}(i,j) \bigr)! }
                  { \bigl( M - \tgam(i,j) \bigr)! }
           }_{\defeq \ (*)}
       \Biggr) \nonumber \\
    &\overset{(c)}{=}
       \frac{1}{M^{ 2n -  n^2}}
       \cdot
       \sum_{\sigma_1 \in \setS_{[n]}( \vgam )}
       \Biggl( 
         \prod_{\substack{i,j \\ j = \sigma_1(i)}}
           \tgam(i,j)
       \Biggr)
       \cdot
       \Biggl( 
         \prod_{\substack{i,j \\ j \neq \sigma_1(i)}}
           \frac{ 1 }
                { M - \tgam(i,j) }
       \Biggr) \nonumber \\
    &\overset{(d)}{=}
       \frac{1}{M^{ 2n -  n^2}}
       \cdot
       \sum_{\sigma_1 \in \setS_{[n]}( \vgam )}
       \Biggl( 
         \prod_{\substack{(i,j) \in \setR \times \setC \\ j = \sigma_1(i)}}
           \tgam(i,j)
       \Biggr)
       \cdot
       \underbrace
       {
         \Biggl( 
           \prod_{\substack{(i,j) \notin \setR \times \setC \\ j = \sigma_1(i)}}
             \tgam(i,j)
         \Biggr)
       }_{\overset{(e)}{=} \ M^{n-r}}
       \cdot
       \Biggl( 
         \prod_{\substack{(i,j) \in \setR \times \setC \\ j \neq \sigma_1(i)}}
           \frac{ 1 }
                { M - \tgam(i,j) }
       \Biggr)
       \cdot
       \underbrace{
         \Biggl( 
           \prod_{\substack{(i,j) \notin \setR \times \setC \\ j \neq \sigma_1(i)}}
             \frac{ 1 }
                  { M - \tgam(i,j) }
         \Biggr)
       }_{\overset{(f)}{=} \ \frac{1}{M^{n^2 - r^2 - (n - r)}}} \nonumber \\
    &= \frac{1}{M^{ 2r -  r^2}}
       \cdot
       \sum_{\sigma_1 \in \setS_{[n]}( \vgam )}
       \Biggl( 
         \prod_{\substack{(i,j) \in \setR \times \setC \\ j = \sigma_1(i)}}
           \tgam(i,j)
       \Biggr)
       \cdot
       \Biggl( 
         \prod_{\substack{(i,j) \in \setR \times \setC \\ j \neq \sigma_1(i)}}
           \frac{ 1 }
                { M - \tgam(i,j) }
       \Biggr) \nonumber \\
    &= \sum_{\sigma_1 \in \setS_{[n]}( \vgam )}
       \Biggl( 
         \prod_{\substack{(i,j) \in \setR \times \setC \\ j = \sigma_1(i)}}
           \gamma(i,j)
       \Biggr)
       \cdot
       \Biggl( 
         \prod_{\substack{(i,j) \in \setR \times \setC \\ j \neq \sigma_1(i)}}
           \frac{ 1 }
                { 1 - \gamma(i,j) }
       \Biggr) \nonumber \\
    &= \sum_{\sigma_1 \in \setS_{[n]}( \vgam )} 
         \frac{\prod_{i \in \setR}
                 \gamma(i,\sigma_1(i))
                 \, \cdot \,
                 \bigr( 1 - \gamma\isigmaoi \bigl)
              }
              {\prod\limits_{(i',j') \in \setR \times \setC}
                 \bigr( 1 - \gamma(i',j') \bigl)
              } \nonumber \\
    &= \sum_{\sigma_1 \in \setS_{[n]}( \vgam )} 
         \prod_{i \in \setR}
           \underbrace
           {
             \frac{
                   \gamma(i,\sigma_1(i))
                   \, \cdot \,
                   \bigr( 1 - \gamma\isigmaoi \bigl)
                  }
                  {
                   \Biggl(
                     \prod\limits_{(i',j') \in \setR \times \setC}
                     \bigr( 1 - \gamma(i',j') \bigl)
                   \Biggr)^{\!\!\! 1 / r}
                  }
           }_{\overset{(g)}{=} \ \hgamRCp\isigmaoi}
    \nonumber \\
    &= \sum_{\sigma_1 \in \setS_{[n]}( \vgam )} \ 
         \prod_{i \in \setR}
           \hgamRCp\isigmaoi
             \nonumber \\
    &\overset{(h)}{=}
       \perm(\hvgamRCp)\, .
   \label{eqn: derivations of recursion of CBM}
\end{align}
}



The expression in~\eqref{eqn: derivations of recursion of CBM} at the top of
the next page is then obtained as follows:
\begin{itemize}

\item Step~$(a)$ follows from Definition~\ref{sec:1:lem:43}.

\item Step~$(b)$ follows from merging terms and using
  $\frac{(M-1)!}{M!} \! = \! \frac{1}{M}$.

\item Step~$(c)$ follows from the fact that the value of $(*)$ equals
  $\tgam(i,j)$ if $j = \sigma_1(i)$ and equals $\frac{1}{M - \tgam(i,j)}$ if
  $j \neq \sigma_1(i)$.\footnote{Note that, if $j \neq \sigma_1(i)$, then
    $\tgam(i,j) < M$, and with that $M - \tgam(i,j) > 0$.}
  The value of $(*)$ can be computed using the following relationships. \\
  If $j = \sigma_1(i)$ then
  \begin{align*}
    \tgam_{\sigma_1}(i,j)
      &= \tgam(i,j) - 1\, , \\
    (M-1) - \tgam_{\sigma_1}(i,j)
      &= M - \tgam(i,j)\, .
  \end{align*}
  If $j \neq \sigma_1(i)$ then
  \begin{align*}
    \tgam_{\sigma_1}(i,j)
      &= \tgam(i,j)\, , \\
    (M-1) - \tgam_{\sigma_1}(i,j)
      &= M - \tgam(i,j) - 1\, .
  \end{align*}

\item Step~$(d)$ follows from splitting the product over
  $(i,j) \in [n] \times [n]$ into two products: one over
  $(i,j) \in \setR \times \setC$ and one over
  $(i,j) \notin \setR \times \setC$.

\item Step~$(e)$ follows from the fact that this product has $n - r$ terms
  that each evaluate to $M$.

\item Step~$(f)$ follows from the fact that this product has $n^2 - r^2 -
  (n \! - \! r)$ terms that each evaluate to $\frac{1}{M}$.

\item Step~$(g)$ follows from the definition of the matrix~$\hvgamRCp$ in
  Definition~\ref{sec:1:def:10:part:2}.

\item Step~$(h)$ follows from the following considerations.

  If the sets $\setR$ and $\setC$ are empty, which happens exactly when
  $\vgam$ is a permutation matrix, we obtain
  \begin{align*}
    \sum_{\sigma_1 \in \setS_{[n]}( \vgam )} \ 
      \prod_{i \in \setR}
        \hgamRCp\isigmaoi 
      &= \!\!\!\!\!\! \sum_{\sigma_1 \in \setS_{[n]}( \vgam )} \!\!\!\!\!\!
           1
       = 1
       = \perm(\hvgamRCp)\, ,
  \end{align*}
  where the first equality follows from the convention that a product over an
  empty index set evaluates to $1$, where the second equality follows from
  $\setS_{[n]}( \vgam )$ containing only a single element, and where the third
  equality follows from the definition $\perm(\hvgamRCp) \defeq 1$ when the
  sets $\setR$ and $\setC$ are empty.

  If the sets $\setR$ and $\setC$ are non-empty, the result follows from
  \begin{align*}
    \hspace{0.5cm}&\hspace{-0.50cm}
    \sum_{\sigma_1 \in \setS_{[n]}( \vgam )} \ 
      \prod_{i \in \setR} 
        \hgamRCp\isigmaoi \\
      &= \sum_{\sigma_1 \in \setS_{\setR \to \setC}( \hvgamRCp )} \ 
           \prod_{i \in \setR} 
             \hgamRCp\isigmaoi
       = \perm(\hvgamRCp)\, ,
  \end{align*}
  where the first equality follows from the observation that there is a
  bijection between the set of permutations $\setS_{[n]}( \vgam )$ and the set
  of permutations $\setS_{\setR \to \setC}(\hvgamRCp)$, which is given by
  restricting the permutations in $\setS_{[n]}( \vgam )$ to take arguments
  only in $\setR$.


\end{itemize}
Finally, multiplying both sides of~\eqref{eqn: derivations of recursion of
  CBM} by $\frac{\CBM{n}(\vgam)}{\perm(\hvgamRCp)}$ yields the result promised
in the lemma statement.


\ifx\withclearpagecommands\x
\clearpage
\newpage
\fi

\section{Proof of Lemma~\ref{sec:1:lem:30}}
\label{apx:23}


\bigformulatop{31}{32}{%
\begin{align}
  \frac{1}{\CscSgen{M}{n}(\vgam)}
    \cdot
    \sum_{\sigma_1 \in \setS_{[n]}( \vgam )}
      \CscSgen{M-1}{n}(\vgam_{\sigma_1})
    &\overset{(a)}{=} 
       M^{n \cdot M}
       \cdot 
       \frac{ 
             \prod_{i,j} 
               \tgam(i,j)!
            } 
            { \bigr( M! \bigr)^{\! 2n} }
       \cdot
       \sum_{\sigma_1 \in \setS_{[n]}( \vgam )} 
         (M-1)^{- n \cdot (M-1)}
         \cdot 
         \frac{ \bigr( (M-1)! \bigr)^{\! 2n} }
              { \prod_{i,j} 
                  \tgam_{\sigma_1}(i,j)!  
              } \nonumber \\
    &\overset{(b)}{=} 
       \underbrace
       {
         \biggl( \frac{M}{M-1} \biggr)^{\!\! n \cdot (M-1)}
       }_{= \ ( \chi(M) )^{n}}
       \cdot \ 
       \frac{1}{M^n}
       \cdot
       \sum_{\sigma_1 \in \setS_{[n]}( \vgam )}
         \prod_{i,j}
           \underbrace
           {
             \frac{ \tgam(i,j)! }
                  { \tgam_{\sigma_1}(i,j)! }
           }_{\defeq \ (*)}
             \nonumber \\
    &\overset{(c)}{=} 
       \bigl( \chi(M) \bigr)^{\! n}
       \cdot
       \frac{1}{M^n}
       \cdot
       \sum_{\sigma_1 \in \setS_{[n]}( \vgam )} 
         \prod_{\substack{i,j \\ j = \sigma_1(i)}}
           \tgam(i,j)
             \nonumber \\
    &= \bigl( \chi(M) \bigr)^{\! n}
       \cdot
       \frac{1}{M^n}
       \cdot
       \sum_{\sigma_1 \in \setS_{[n]}( \vgam )} 
         \prod_i
           \tgam\isigmaoi
             \nonumber \\
    &= \bigl( \chi(M) \bigr)^{\! n}
       \cdot
       \sum_{\sigma_1 \in \setS_{[n]}( \vgam )} 
         \prod_i
           \gamma\isigmaoi
             \nonumber \\
    &= \bigl( \chi(M) \bigr)^{\! n}
       \cdot 
       \perm( \vgam )\, .
         \label{eqn: derivations of recursion of CscSM}
\end{align}
}


We start by defining the matrices
\begin{align*}
  \tvgam
    &\defeq 
       M \cdot \vgam\, , \\ 
  \tvgam_{\sigma_1}
    &\defeq 
       (M-1) \cdot \vgam_{\sigma_1}\, .
\end{align*}
Note that $\tvgam_{\sigma_1} = \tvgam - \mP_{\sigma_1}$.


The expression in~\eqref{eqn: derivations of recursion of CscSM}
at the top of the next page is then obtained as
follows:
\begin{itemize}

\item Step~$(a)$ follows from Definition~\ref{sec:1:lem:43}.

\item Step~$(b)$ follows from merging terms and using
  $\frac{(M-1)!}{M!} \! = \! \frac{1}{M}$.

\item Step~$(c)$ follows from the fact that the value of $(*)$ equals
  $\tgam(i,j)$ if $j = \sigma_1(i)$ and equals $1$ if $j \neq \sigma_1(i)$.
  The value of $(*)$ can be computed using the following relationships. \\
  If $j = \sigma_1(i)$ then
    \begin{align*}
      \tgam_{\sigma_1}(i,j)
        &= \tgam(i,j) - 1\, .
    \end{align*}
    If $j \neq \sigma_1(i)$ then
    \begin{align*}
      \tgam_{\sigma_1}(i,j)
        &= \tgam(i,j)\, .
    \end{align*}

\end{itemize}
Finally, multiplying both sides of~\eqref{eqn: derivations of recursion of
  CscSM} by
$\frac{\CscSgen{M}{n}(\vgam_{\sigma_1})}{\bigl( \chi(M) \bigr)^{\! n} \cdot
  \perm( \vgam )}$ yields the result promised in the lemma statement.


\ifx\withclearpagecommands\x
\clearpage
\mbox{}
\newpage
\fi

\section{Proof of Lemma~\ref{lemma:ratio:perm:permbM:1}}
\label{app:proof:lemma:ratio:perm:permbM:1}


We obtain
\begin{align*}
  &\bigl( \permbM{M}(\mtheta) \bigr)^{\!M} \\
    &\overset{(a)}{=}
       \sum\limits_{ \vgam \in \GamMn } 
         \mtheta^{ M \cdot \vgam }
         \cdot
         \CBM{n}( \vgam ) \\
    &\overset{(b)}{=}
       \sum\limits_{ \vgam \in \GamMn } 
         \mtheta^{ M \cdot \vgam }
         \cdot
         \frac{\CBM{n}(\vgam)}
              {\CM{n}( \vgam )}
         \cdot
         \CM{n}( \vgam ) \\
    &\overset{(c)}{=}
       \sum\limits_{ \vgam \in \GamMn } 
         \mtheta^{ M \cdot \vgam }
         \cdot
         \frac{\CBM{n}(\vgam)}
              {\CM{n}( \vgam )}
         \cdot
         \sum_{ \vsigma_{[M]} \in \setS_{[n]}^{M} }
           \Bigl[ 
             \vgam = \bigl\langle \mP_{\sigma_{m}} \bigr\rangle_{m \in [M]}
           \Bigr] \\
    &= \sum_{ \vsigma_{[M]} \in \setS_{[n]}^{M} }
         \sum\limits_{ \vgam \in \GamMn } 
           \mtheta^{ M \cdot \vgam }
           \cdot
           \frac{\CBM{n}(\vgam)}
                {\CM{n}( \vgam )}
           \cdot
           \Bigl[ 
             \vgam = \bigl\langle \mP_{\sigma_{m}} \bigr\rangle_{m \in [M]}
           \Bigr] \\
    &= \sum_{ \vsigma_{[M]} \in \setS_{[n]}^{M} }
         \sum\limits_{ \vgam \in \GamMn } 
           \mtheta^{ M \cdot \langle \mP_{\sigma_{m}} \rangle_{m \in [M]}}
           \cdot
           \frac{\CBM{n}\Bigl(
                          \bigl\langle \mP_{\sigma_{m}} \bigr\rangle_{m \in [M]}
                        \Bigr)}
                {\CM{n}\Bigl(
                         \bigl\langle \mP_{\sigma_{m}} \bigr\rangle_{m \in [M]}
                       \Bigr)} \\
    &\hspace{5.5cm}
           \cdot
           \Bigl[ 
             \vgam = \bigl\langle \mP_{\sigma_{m}} \bigr\rangle_{m \in [M]}
           \Bigr] \\
    &= \sum_{ \vsigma_{[M]} \in \setS_{[n]}^{M} }
         \mtheta^{ M \cdot \langle \mP_{\sigma_{m}} \rangle_{m \in [M]}}
         \cdot
         \frac{\CBM{n}\Bigl(
                        \bigl\langle \mP_{\sigma_{m}} \bigr\rangle_{m \in [M]}
                      \Bigr)}
              {\CM{n}\Bigl(
                        \bigl\langle \mP_{\sigma_{m}} \bigr\rangle_{m \in [M]}
                      \Bigr)} \\
    &\hspace{4.5cm}
         \cdot
         \underbrace
         {
           \sum\limits_{ \vgam \in \GamMn } 
           \Bigl[ 
             \vgam = \bigl\langle \mP_{\sigma_{m}} \bigr\rangle_{m \in [M]}
           \Bigr]
         }_{= \ 1} \\
    &= \sum_{ \vsigma_{[M]} \in \setS_{[n]}^{M} }
         \mtheta^{ M \cdot \langle \mP_{\sigma_{m}} \rangle_{m \in [M]}}
         \cdot
         \frac{\CBM{n}\Bigl(
                        \bigl\langle \mP_{\sigma_{m}} \bigr\rangle_{m \in [M]}
                      \Bigr)}
              {\CM{n}\Bigl(
                        \bigl\langle \mP_{\sigma_{m}} \bigr\rangle_{m \in [M]}
                      \Bigr)} \\
    &= \sum_{ \vsigma_{[M]} \in \setS_{[n]}^{M} } \!
         \Biggl(
           \prod_{m \in [M]} \!
             \mtheta^{ \mP_{\sigma_{m}} } \!
         \Biggr)
         \cdot
         \frac{\CBM{n}\Bigl(
                        \bigl\langle \mP_{\sigma_{m}} \bigr\rangle_{m \in [M]}
                      \Bigr)}
              {\CM{n}\Bigl(
                       \bigl\langle \mP_{\sigma_{m}} \bigr\rangle_{m \in [M]}
                     \Bigr)}\, ,
\end{align*}
where step~$(a)$ follows from~\eqref{sec:1:eqn:190}, where step~$(b)$ is valid
because $\CM{n}(\vgam) \geq \CBM{n}(\vgam) > 0$ for
  $\vgam \in \GamMn$, as stated in the lemma statement, and where step~$(c)$
follows from~\eqref{sec:1:eqn:56}. Dividing the above expression for
$\bigl( \permbM{M}(\mtheta) \bigr)^{\!M}$ by $\bigl( \perm(\mtheta) \bigr)^{\!M}$, we
get
\begin{align*}
  \hspace{0.5cm}&\hspace{-0.5cm}
  \Biggl( 
    \frac{\permbM{M}(\mtheta)}
         {\perm(\mtheta)}
  \Biggr)^{\!\!\! M} \\
    &= \sum_{ \vsigma_{[M]} \in \setS_{[n]}^{M} }
         \Biggl(
           \prod_{m \in [M]}
             \frac{\mtheta^{ \mP_{\sigma_{m}} }}
                  {\perm(\mtheta)}
         \Biggr)
         \cdot
         \frac{\CBM{n}\Bigl(
                        \bigl\langle \mP_{\sigma_{m}} \bigr\rangle_{m \in [M]}
                      \Bigr)}
              {\CM{n}\Bigl(
                        \bigl\langle \mP_{\sigma_{m}} \bigr\rangle_{m \in [M]}
                      \Bigr)} \\
    &= \sum_{ \vsigma_{[M]} \in \setS_{[n]}^{M} }
         \Biggl(
           \prod_{m \in [M]}
             \pmtheta(\sigma_m)
         \Biggr)
         \cdot
         \frac{\CBM{n}\Bigl(
                        \bigl\langle \mP_{\sigma_{m}} \bigr\rangle_{m \in [M]}
                      \Bigr)}
              {\CM{n}\Bigl(
                        \bigl\langle \mP_{\sigma_{m}} \bigr\rangle_{m \in [M]}
                      \Bigr)}\, .
\end{align*}



\ifx\withclearpagecommands\x
\clearpage
\fi
\section{An Alternative proof of the inequalities in~\eqref{sec:1:eqn:147} in
  Theorem~\ref{th:main:permanent:inequalities:1}}
\label{apx: alternative proof of perm geq permb}
\begin{itemize}

\item We prove the first inequality in~\eqref{sec:1:eqn:147} by observing that
  \begin{align*}
    \frac{\perm(\mtheta)}
         {\permbM{M}(\mtheta)}
      &\overset{(a)}{\geq}
         \left( \!
           \sum_{ \vsigma_{[M]} \in \setS_{[n]}^{M} }
             \Biggl(
               \prod_{m \in [M]}
                 \pmtheta(\sigma_m)
             \Biggr)
         \right)^{\!\!\! -1/M} \\
      &= \Biggl( 
           \prod_{m \in [M]}
             \sum_{\sigma_m \in \setS_{[n]}}
               \pmtheta(\sigma_m)
         \Biggr)^{\!\!\! -1/M} \\
      &\overset{(b)}{=}
         \Biggl( 
           \prod_{m \in [M]}
             1
         \Biggr)^{\!\!\! -1/M} \\
      &= 1\, ,
  \end{align*}
  where step~$(a)$ is obtained from~\eqref{eq:ratio:perm:permbM:1} by upper
  bounding the ratio $\CBM{n}(\ldots) / \CM{n}(\ldots)$ therein with the help
  of the first inequality in~\eqref{sec:1:eqn:58}, and where step~$(b)$ used
  the fact that $\pmtheta$ is a probability mass function.

\item We prove the second inequality in~\eqref{sec:1:eqn:147} by observing that
  \begin{align}
    &
    \frac{\perm(\mtheta)}
         {\permbM{M}(\mtheta)} \nonumber \\
      &\overset{(a)}{\leq}
         \left( \!
           \sum_{ \vsigma_{[M]} \in \setS_{[n]}^{M} }
             \Biggl(
               \prod_{m \in [M]}
                 \pmtheta(\sigma_m)
             \Biggr)
             \cdot
             \bigl( 2^{n/2} \bigr)^{\! -(M-1)}
         \right)^{\!\!\! -1/M} \nonumber \\
      &= \bigl( 2^{n/2} \bigr)^{\!\! \frac{M-1}{M}}
         \cdot
         \Biggl( 
           \prod_{m \in [M]}
             \sum_{\sigma_m \in \setS_{[n]}}
               \pmtheta(\sigma_m)
         \Biggr)^{\!\!\! -1/M} \nonumber \\
      &\overset{(b)}{=}
         \bigl( 2^{n/2} \bigr)^{\!\! \frac{M-1}{M}}
         \cdot
         \Biggl(
           \prod_{m \in [M]}
             1
         \Biggr)^{\!\!\! -1/M} \nonumber \\
      &= \bigl( 2^{n/2} \bigr)^{\!\! \frac{M-1}{M}}\, ,
           \label{eq:ratio:perm:permbM:proof:upper:bound:1}
  \end{align}
  where step~$(a)$ is obtained from~\eqref{eq:ratio:perm:permbM:1} by lower
  bounding the ratio $\CBM{n}(\ldots) / \CM{n}(\ldots)$ therein with the help
  of the second inequality in~\eqref{sec:1:eqn:58}, and where step~$(b)$ used
  the fact that $\pmtheta$ is a probability mass function.

\end{itemize}

\ifx\withclearpagecommands\x
\clearpage
\fi

\section{Another Alternative proof of the inequalities in~\eqref{sec:1:eqn:147} in
  Theorem~\ref{th:main:permanent:inequalities:1}}
\label{apx: alternative prooof of perm geq permb: 2}

We start this alternative proof by proving the following lemma.
\begin{lemma}\label{prop: a special case of Navin conjecture}
  For any integers $ M_{1} \in \sZpp $ and $ M_{2} \in \sZpp $, we have
  \begin{align}
      1 \leq 
      \frac{ 
        \bigl( \perm(\mtheta) \bigr)^{\! M_{1}} \cdot 
        \bigl( \permbM{M_{2}}(\mtheta) \bigr)^{\! M_{2}} 
      }{ \bigl( \permbM{M_{1}+M_{2}}(\mtheta) \bigr)^{\! M_{1}+M_{2}} }
      &\leq \bigl( 2^{n/2} \bigr)^{\! M_{1}}\, . 
      \label{eqn: a special case of Navin conjecture:2}
  \end{align}
\end{lemma}
\begin{proof}
  Consider an arbitrary integer $ M_{3} \in \sZ_{\geq 1} $. It holds that
  \begin{align}
    \hspace{0.25cm}&\hspace{-0.25cm}
    \perm(\mtheta) \cdot \bigl( \permbM{M_{3}}(\mtheta) \bigr)^{\! M_{3}}
    \nonumber\\
    &\overset{(a)}{=}
    \Biggl(
          \sum_{\sigma_1 \in \setS_{[n]}( \mtheta )} 
          \mtheta^{ \mP_{\sigma_1} }
        \Biggr)
       \cdot 
       \sum_{\vgam \in \Gamma_{M_{3},n}(\mtheta)} 
       \mtheta^{ M_{3} \cdot \vgam }
       \cdot
      \CBgen{M_{3}}{n}(\vgam) 
    \nonumber\\
    &\overset{(b)}{=}
    \sum_{\vgam \in \Gamma_{M_{3}+1,n}(\mtheta)} 
           \mtheta^{ (M_{3}+1) \cdot \vgam }
           \cdot
           \sum_{\sigma_1 \in \setS_{[n]}( \vgam )}
           \CBgen{M_{3}}{n}
               \bigl( 
                 \vgam_{\sigma_1}
               \bigr)
    \nonumber\\
    &\overset{(c)}{=}
    \sum_{\vgam \in \Gamma_{M_{3}+1,n}(\mtheta)} 
           \mtheta^{ (M_{3}+1) \cdot \vgam }
           \cdot
           \perm(\hvgamRCp)
           \cdot
           \CBgen{M_{3}+1}{n}(\vgam)
   \nonumber\\
    &\overset{(d)}{\geq}
    \sum_{\vgam \in \Gamma_{M_{3}+1,n}(\mtheta)} 
           \mtheta^{ (M_{3}+1) \cdot \vgam }
           \cdot
           \CBgen{M_{3}+1}{n}(\vgam)
           \nonumber\\
    &\overset{(e)}{=}
    \bigl( \permbM{M_{3} +1}(\mtheta) \bigr)^{\! M_{3}+1}\, , 
    \label{eqn: a special case of Navin conjecture:3}
  \end{align} 
  where step $(a)$ follows from the definition of the permanent and the expression in~\eqref{sec:1:eqn:190} for $ M = M_{3} $,
  where step $(b)$ follows from the definition of $ \vgam_{\sigma_1} $ in Definition~\ref{sec:1:def:10:part:2:add} for $ M = M_{3} + 1 $,
  where step $(c)$ follows from Lemma~\ref{sec:1:lem:15:copy:2} for $ M = M_{3} + 1 $,
  where step $(d)$ follows from Lemma~\ref{sec:1:lem:17},
  and where step $(e)$ follows from the expression in~\eqref{sec:1:eqn:190} for $ M = M_{3} + 1 $.

  Because the inequality in~\eqref{eqn: a special case of Navin conjecture:3} holds for any $ M_{3} \in \sZ_{\geq 1} $, we have
  \begin{align*}
    \hspace{0.25cm}&\hspace{-0.25cm}
    \bigl( \permbM{M_{1}+M_{2}}(\mtheta) \bigr)^{\! M_{1}+M_{2}}
    \nonumber\\
    &\leq \perm(\mtheta) \cdot 
    \bigl( \permbM{M_{1}+M_{2}-1}(\mtheta) \bigr)^{\! M_{1}+M_{2}-1}
    \nonumber\\
    &\leq \bigl( \perm(\mtheta) \bigr)^{\! 2} 
    \cdot \bigl( \permbM{M_{1}+M_{2}-2}(\mtheta) \bigr)^{\! M_{1}+M_{2}-2}
    \nonumber\\
    &\ \cdots
    \nonumber\\
    &\leq \bigl( \perm(\mtheta) \bigr)^{\! M_{1}} 
    \cdot \bigl( \permbM{M_{2}}(\mtheta) \bigr)^{\! M_{2}}\, .
  \end{align*}

  A similar line of reasoning can be used to prove the upper bound in~\eqref{eqn: a special case of Navin conjecture:2}.
  The details are omitted.
\end{proof}
The inequalities~\eqref{sec:1:eqn:147} can now be proven by observing that they are special cases of Lemma~\ref{prop: a special case of Navin conjecture}. Namely, they follow from~\eqref{eqn: a special case of Navin conjecture:2} for $M_{1} = M-1$ and $M_{2} = 1$, along with using $ \perm(\mtheta) = \permbM{1}(\mtheta) $.

\ifx\withclearpagecommands\x
\clearpage
\fi

\section{Details of the Alternative Proof of 
               Proposition~\ref{prop:ratio:perm:permbM:2:1}}
\label{app:alt:proof:prop:ratio:perm:permbM:2:1}


\medskip

Toward establishing \eqref{eq:CBtwo:result:1} and~\eqref{eq:Ctwo:result:1}, we 
start by studying the matrix
\begin{align*}
  \frac{1}{2} 
  \bigl( 
  \mP_{\sigma_1} \! + \! \mP_{\sigma_2} 
  \bigr)
\end{align*}
for arbitrary $\sigma_1, \sigma_2 \in \setS_{[n]}$.


\begin{lemma}
  \label{lemma:vgam:graph:properties:1}

  Let $\sigma_1, \sigma_2 \in \setS_{[n]}$ be fixed. Define the matrix
  \begin{align*}
    \vgam
      &\defeq
         \frac{1}{2} 
         \bigl( 
           \mP_{\sigma_1} \! + \! \mP_{\sigma_2} 
          \bigr).
  \end{align*}
  Let the graph $\sfG(\sigma_1,\sigma_2)$ be a bipartite graph with two times
  $n$ vertices: $n$ vertices on the LHS and $n$ vertices on the RHS, each of
  them labeled by $[n]$. Let there be an edge $(i,j)$ connecting the $i$-th
  vertex on the LHS with the $j$-th vertex on the RHS if
  $\gamma(i,j) = \frac{1}{2}$. Let $c\bigl( \sfG(\sigma_1,\sigma_2) \bigr)$ be
  the number of cycles in $\sfG(\sigma_1,\sigma_2)$.

  The matrix $\vgam$ has the following properties:
  \begin{enumerate}
    
  \item $\vgam \in \Gamtwon$.

  \item For every $i \in [n]$, the $i$-th row $\vgam(i,:)$ of $\vgam$ has one
    of the following two possible compositions:
    \begin{itemize}
  
    \item One entry is equal to $1$; $n \! - \! 1$ entries are equal to~$0$.
  
    \item Two entries are equal to $\frac{1}{2}$; $n \! - \! 2$ entries are
      equal to~$0$.

    \end{itemize}

  \end{enumerate}
  The graph $\sfG(\sigma_1,\sigma_2)$ has the following properties:
  \begin{enumerate}
    
  \item The edge set of $\sfG(\sigma_1,\sigma_2)$ forms a collection of
    disjoint cycles of even length.

  \item Every cycle in $\sfG(\sigma_1,\sigma_2)$ of length $2L$ corresponds to
    a cycle of length $L$ in the cycle notation of
    $\sigma_1 \circ \sigma_2^{-1}$.

  \item $c\bigl( \sfG(\sigma_1,\sigma_2) \bigr) = c(\sigma_1,\sigma_2)$.

  \end{enumerate}

\end{lemma}


\begin{proof}
  The derivation of the properties of $\vgam$ is rather straightforward and is
  therefore omitted. The properities of $\sfG(\sigma_1,\sigma_2)$ follow from the following
  observations:
  \begin{itemize}
    
  \item Because $\sfG(\sigma_1,\sigma_2)$ is a bipartite graph, all cycles
    must have even length.

  \item Let $i \in [n]$. If $\sigma_1(i) = \sigma_2(i)$, then the $i$-th
    vertex on the LHS has no indicent edges.
    
  \item Let $i \in [n]$. If $\sigma_1(i) \neq \sigma_2(i)$, then the $i$-th
    vertex on the LHS has degree two and is part of a cycle.

  \item Let $j \in [n]$. If $\sigma_1^{-1}(j) = \sigma_2^{-1}(j)$, or,
    equivalently, if $\sigma_1\bigl( \sigma_2^{-1}(j) \bigr) = j$, then the
    $j$-th vertex on the RHS has no indicent edges.
    
  \item Let $j \in [n]$. If $\sigma_1^{-1}(j) \neq \sigma_2^{-1}(j)$, or,
    equivalently, if $\sigma_1\bigl( \sigma_2^{-1}(j) \bigr) \neq j$, then the
    $j$-th vertex on the RHS has degree two and is part of a cycle.

  \end{itemize}
\end{proof}


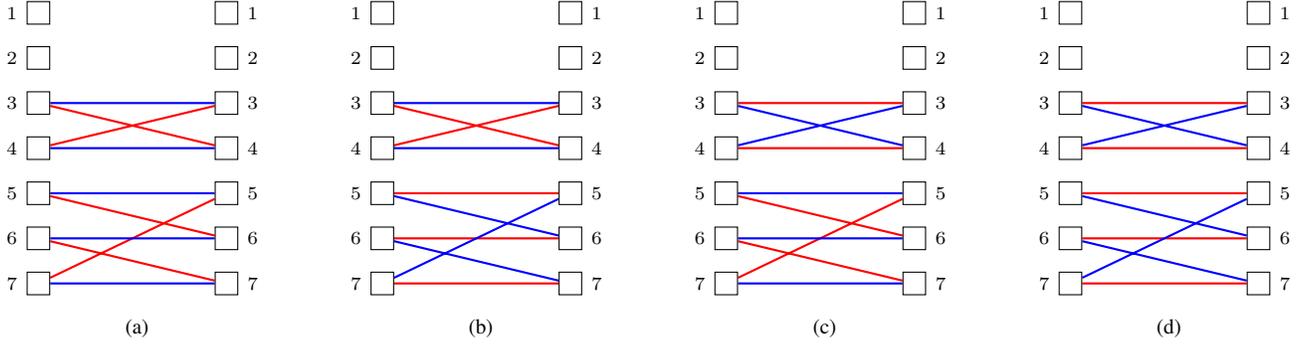
\begin{figure*}[t]
  \begin{centering}
    \subfloat[]{\begin{tikzpicture}[node distance=2.2cm, on grid,auto]
  \tikzstyle{state}=[shape=rectangle,fill=white,draw,minimum size=0.3cm]
  \pgfmathsetmacro{\ldis}{2.5} 
  \pgfmathsetmacro{\sdis}{0.5} 
  \begin{pgfonlayer}{glass}
    \foreach \i in {1,...,7}
    {
      \node[state] (f\i) at (0,-0.6*\i) [label=left: \scriptsize$\i$] {};
      \node[state] (g\i) at (\ldis,-0.6*\i) [label=right: \scriptsize$\i$] {};
    }
    \end{pgfonlayer}
     \begin{pgfonlayer}{background}
         \draw[thick,red]
            (f3) -- (g4)      
            (f4) -- (g3)
            (f5) -- (g6)
            (f6) -- (g7)
            (f7) -- (g5)
            ;
        \draw[thick,blue]
            (f3) -- (g3)      
            (f4) -- (g4)
            (f5) -- (g5)
            (f6) -- (g6)
            (f7) -- (g7)
            ;
    \end{pgfonlayer}
\end{tikzpicture}}
    \hspace{0.75cm}
    \subfloat[]{\begin{tikzpicture}[node distance=2.2cm, on grid,auto]
  \tikzstyle{state}=[shape=rectangle,fill=white,draw,minimum size=0.3cm]
  \pgfmathsetmacro{\ldis}{2.5} 
  \pgfmathsetmacro{\sdis}{0.5} 
  \begin{pgfonlayer}{glass}
    \foreach \i in {1,...,7}
    {
      \node[state] (f\i) at (0,-0.6*\i) [label=left: \scriptsize$\i$] {};
      \node[state] (g\i) at (\ldis,-0.6*\i) [label=right: \scriptsize$\i$] {};
    }
    \end{pgfonlayer}
    \begin{pgfonlayer}{background}
         \draw[thick,red]
            (f3) -- (g4)      
            (f4) -- (g3)
            (f5) -- (g5)
            (f6) -- (g6)
            (f7) -- (g7)
            ;
        \draw[thick,blue]
            (f3) -- (g3)      
            (f4) -- (g4)
            (f5) -- (g6)
            (f6) -- (g7)
            (f7) -- (g5)
            ;
  \end{pgfonlayer}
  \vspace{0.5cm}
\end{tikzpicture}}
    \hspace{0.75cm}
    \subfloat[]{\begin{tikzpicture}[node distance=2.2cm, on grid,auto]
  \tikzstyle{state}=[shape=rectangle,fill=white,draw,minimum size=0.3cm]
  \pgfmathsetmacro{\ldis}{2.5} 
  \pgfmathsetmacro{\sdis}{0.5} 
  \begin{pgfonlayer}{glass}
    \foreach \i in {1,...,7}
    {
      \node[state] (f\i) at (0,-0.6*\i) [label=left: \scriptsize$\i$] {};
      \node[state] (g\i) at (\ldis,-0.6*\i) [label=right: \scriptsize$\i$] {};
    }
    \end{pgfonlayer}
    \begin{pgfonlayer}{background}
         \draw[thick,red]
            (f3) -- (g3)      
            (f4) -- (g4)
            (f5) -- (g6)
            (f6) -- (g7)
            (f7) -- (g5)
            ;
        \draw[thick,blue]
            (f3) -- (g4)      
            (f4) -- (g3)
            (f5) -- (g5)
            (f6) -- (g6)
            (f7) -- (g7)
            ;
  \end{pgfonlayer}
  \vspace{0.5cm}
\end{tikzpicture}}
    \hspace{0.75cm}
    \subfloat[]{\begin{tikzpicture}[node distance=2.2cm, on grid,auto]
  \tikzstyle{state}=[shape=rectangle,fill=white,draw,minimum size=0.3cm]
  \pgfmathsetmacro{\ldis}{2.5} 
  \pgfmathsetmacro{\sdis}{0.5} 
  \begin{pgfonlayer}{glass}
    \foreach \i in {1,...,7}
    {
      \node[state] (f\i) at (0,-0.6*\i) [label=left: \scriptsize$\i$] {};
      \node[state] (g\i) at (\ldis,-0.6*\i) [label=right: \scriptsize$\i$] {};
    }
    \end{pgfonlayer}
    \begin{pgfonlayer}{background}
         \draw[thick,red]
            (f3) -- (g3)      
            (f4) -- (g4)
            (f5) -- (g5)
            (f6) -- (g6)
            (f7) -- (g7)
            ;
        \draw[thick,blue]
            (f3) -- (g4)      
            (f4) -- (g3)
            (f5) -- (g6)
            (f6) -- (g7)
            (f7) -- (g5)
            ;
  \end{pgfonlayer}
\end{tikzpicture}}
    \medskip
    \caption{Graphs discussed in Examples~\ref{example:sfG:1}
      and~\ref{example:sfG:2}. (The red color is used for edges based on
      $\sigma_1$; the blue color is used for edges based on $\sigma_2$.)}
    \label{fig:example:sfG:1}
  \end{centering}
\end{figure*}


\begin{example}
  \label{example:sfG:1}

  Let $\sigma_1: [7] \to [7]$ defined by $\sigma_1(1) = 1$, $\sigma_1(2) = 2$,
  $\sigma_1(3) = 4$, $\sigma_1(4) = 3$, $\sigma_1(5) = 6$, $\sigma_1(6) = 7$,
  $\sigma_1(7) = 5$. Let $\sigma_2: [7] \to [7]$ defined by $\sigma_2(i) = i$
  for all $i \in [7]$. Note that $\sigma_1 \circ \sigma_2^{-1}$ equals the
  permuation $\sigma$ in Footnote~\ref{footnote:cycle:notation:example:1}. Let
  \begin{align}
    \vgam
      &\defeq
         \frac{1}{2} 
         \bigl( 
           \mP_{\sigma_1} \! + \! \mP_{\sigma_2} 
          \bigr)
       = \frac{1}{2}
           \left(
           \begin{smallmatrix}
             2 & 0     & 0   & 0   & 0   & 0   & 0   \\
             0      & 2 & 0   & 0   & 0   & 0   & 0   \\
             0      & 0     & 1 & 1 & 0   & 0   & 0   \\
             0      & 0     & 1 & 1 & 0   & 0   & 0   \\
             0      & 0     & 0   & 0   & 1 & 1 & 0   \\
             0      & 0     & 0   & 0   & 0   & 1 & 1 \\
             0      & 0     & 0   & 0   & 1 & 0   & 1
           \end{smallmatrix}
         \right)
           \label{eq:example:sfG:1:gamma:1}
  \end{align}
  and define the graph $\sfG(\sigma_1,\sigma_2)$ as in
  Lemma~\ref{lemma:vgam:graph:properties:1}. The resulting graph
  $\sfG(\sigma_1,\sigma_2)$ is shown in Fig.~\ref{fig:example:sfG:1}(a). Note
  that the graph $\sfG(\sigma_1,\sigma_2)$ has two cycles, \ie,
  $c\bigl( \sfG(\sigma_1,\sigma_2) \bigr) = 2$.
  \eexample
\end{example}


\begin{lemma}
  \label{lemma:sfGtwo:properties:1}

  Let $\vgam \in \Gamtwon$ be fixed. Define the graph
  $\sfG(\sigma_1,\sigma_2)$ as in
  Lemma~\ref{lemma:vgam:graph:properties:1}. The number of pairs
  $(\sigma_1, \sigma_2) \in (\setS_{[n]})^2$ such that\footnote{Note that the
    ordering in $(\sigma_1, \sigma_2)$ matters, \ie, $(\sigma_1, \sigma_2)$
    and $(\sigma_2, \sigma_1)$ are considered to be distinct pairs if
    $\sigma_1 \neq \sigma_2$.}
  \begin{align*}
    \vgam
      &= \frac{1}{2} 
         \bigl( 
           \mP_{\sigma_1} \! + \! \mP_{\sigma_2} 
          \bigr)
  \end{align*}
  equals $2^{c(\sfG(\sigma_1,\sigma_2))}$.
\end{lemma}


\begin{proof}
  For every cycle in $\sfG(\sigma_1,\sigma_2)$, there are two ways of choosing
  the corresponding function values of $\sigma_1$ and $\sigma_2$. Overall,
  because there are $c\bigl( \sfG(\sigma_1,\sigma_2) \bigr)$ cycles in
  $\sfG(\sigma_1,\sigma_2)$, there are $2^{c(\sfG(\sigma_1,\sigma_2))}$ ways
  of choosing $\sigma_1$ and $\sigma_2$.
\end{proof}


\begin{example}
  \label{example:sfG:2}

  Consider the matrix $\vgam$ in~\eqref{eq:example:sfG:1:gamma:1}. Because
  $c\bigl( \sfG(\sigma_1,\sigma_2) \bigr) = 2$, there are $2^2$ choices for
  $(\sigma_1, \sigma_2) \in (\setS_{[n]})^2$ such that
  $\vgam = \frac{1}{2} \bigl( \mP_{\sigma_1} \! + \! \mP_{\sigma_2}
  \bigr)$. They are shown in Figs.~\ref{fig:example:sfG:1}(a)--(d). \\
  \mbox{} \eexample
\end{example}


\begin{lemma}
  \label{lemma:CMtwo:properties:1}

  It holds that
  \begin{align*}
    \CM{n}
      \Bigl(
        \frac{1}{2} 
        \bigl( 
          \mP_{\sigma_1} \! + \! \mP_{\sigma_2} 
        \bigr)
      \Bigr)
      &= 2^{c(\sigma_1,\sigma_2)}\, , 
           \qquad \sigma_1, \sigma_2 \in \setS_{[n]}\, .
  \end{align*}
\end{lemma}


\begin{proof}
  Let
  $\vgam \defeq \frac{1}{2} \bigl( \mP_{\sigma_1} \! + \! \mP_{\sigma_2}
  \bigr) \in \Gamtwon$ and define the graph $\sfG(\sigma_1,\sigma_2)$ as in
  Lemma~\ref{lemma:vgam:graph:properties:1}. We obtain
  \begin{align*}
    \CM{n}(\vgam)
      &\overset{(a)}{=}
         \sum_{ (\sigma'_1, \sigma'_2) \in (\setS_{[n]})^2 }
           \biggl[ 
             \vgam = \frac{1}{2} \bigl( 
                                   \mP_{\sigma'_1} \! + \! \mP_{\sigma'_2} 
                                 \bigr)
           \biggr] \\
      &\overset{(b)}{=}
         2^{c(\sfG(\sigma_1,\sigma_2))} \\
      &\overset{(c)}{=}
         2^{c(\sigma_1,\sigma_2)}\, ,
  \end{align*}
  where step~$(a)$ follows from~\eqref{sec:1:eqn:56}, where step~$(b)$ follows
  from Lemma~\ref{lemma:sfGtwo:properties:1} and where step~$(c)$ follows from
  Lemma~\ref{lemma:vgam:graph:properties:1}.
\end{proof}


\begin{lemma}
  \label{lemma:CBMtwo:properties:1}

  It holds that
  \begin{align*}
    \CBM{n}
      \biggl(
        \frac{1}{2} 
        \bigl( 
          \mP_{\sigma_1} \! + \! \mP_{\sigma_2} 
        \bigr)
      \biggr)
      &= 1\, , 
           \qquad \sigma_1, \sigma_2 \in \setS_{[n]}\, .
  \end{align*}
\end{lemma}


\begin{proof}
  Let
  $\vgam \defeq \frac{1}{2} \bigl( \mP_{\sigma_1} \! + \! \mP_{\sigma_2}
  \bigr) \in \Gamtwon$. We have
  \begin{align*}
    \CBtwo{n}(\vgam) 
      &\overset{(a)}{=}
         (2!)^{ 2n -  n^2} 
           \cdot
           \prod_{i,j}
             \frac{ \bigl(2 - 2 \gamma(i,j) \bigr)! }
                  { \bigl(2 \gamma(i,j) \bigr)! } \\
      &= \prod_i
           \underbrace
           {
             \Biggl(
               (2!)^{ 2 -  n} 
               \cdot
               \prod_j
                 \frac{ \bigl(2 - 2 \gamma(i,j) \bigr)! }
                      { \bigl(2 \gamma(i,j) \bigr)! }
             \Biggr)
           }_{\overset{(b)}{=} \ 1} \\
      &= 1\, ,
  \end{align*}
  where step~$(a)$ follows from Definition~\ref{sec:1:lem:43} and where
  step~$(b)$ follows from observing that for every $i \in [n]$, the $i$-th row
  $\vgam(i,:)$ of $\vgam$ has one of the following two possible compositions:
  \begin{itemize}
  
  \item one entry is equal to $1$ and $n - 1$ entries are equal to $0$;
  
  \item two entries are equal to $\frac{1}{2}$ and $n - 2$ entries are equal
    to $0$.

  \end{itemize}
  (See also Lemma~\ref{lemma:vgam:graph:properties:1}.) In both cases, one can
  verify that the expression in big parentheses evaluates to $1$.
\end{proof}


\end{appendices}


\ifx\withclearpagecommands\x
\clearpage 
\fi

{
   \balance

   \begin{footnotesize}
      \bibliographystyle{IEEEtran}
      \bibliography{biblio}
   \end{footnotesize}
}

\end{document}